\documentclass[smallcondensed]{svjour3}       
\smartqed  
\usepackage{float}
\usepackage[ruled,linesnumbered,lined]{algorithm2e}
\usepackage{graphicx}
\usepackage[ left=1in, right=1in]{geometry}
\usepackage[colorlinks=true,breaklinks=true,bookmarks=true,urlcolor=blue,
     citecolor=blue,linkcolor=blue,bookmarksopen=false,draft=false]{hyperref}

\usepackage{enumerate}

\usepackage{makeidx} 
\usepackage{wrapfig}
\usepackage{amssymb}
\usepackage[TABBOTCAP,tight]{subfigure}

\usepackage{graphicx}
\usepackage[colorlinks=true,citecolor=blue]{hyperref}
\usepackage{pdfsync}
\usepackage{verbatim}

\usepackage{amsmath} 

\newcommand{\Real}{\mathbb{R}}                              
\newcommand{\set}[1]{\left\{#1\right\}}                     
\newcommand{\abs}[1]{\left|#1\right|}                       
\newcommand{\bra}[1]{\left(#1\right)}                       
   
\newcommand{\white}[1]{{\color{white} #1}}                  
\newcommand {\beq}{\begin{equation}}
\newcommand {\eeq}{\end{equation}}
\newcommand {\beqn}{\begin{equation*}}
\newcommand {\eeqn}{\end{equation*}}
\newcommand {\bear}{\begin{eqnarray}}
\newcommand {\eear}{\end{eqnarray}}
\newcommand {\bearn}{\begin{eqnarray*}}
\newcommand {\eearn}{\end{eqnarray*}}

\newcommand{\norm}[1]{\left\lVert#1\right\rVert}

\newcommand{\red}[1]{{#1}}
\newcommand{\redd}[1]{{#1}}

\DeclareMathOperator{\conv}{conv}
\DeclareMathOperator{\cone}{cone}

\DeclareMathOperator{\epi}{epi}

\DeclareMathOperator{\proj}{Proj}

\newcommand{\Int}{\mathbb{Z}}
\newcommand {\bcom}{}
\DeclareMathOperator{\nlp}{\mathsf{CP}}

\DeclareMathOperator{\LB}{LB}
\DeclareMathOperator{\NB}{UB}

\DeclareMathOperator{\micp}{\mathsf{MICQP}}

\DeclareMathOperator{\lp}{\mathsf{LP}}
\DeclareMathOperator{\milp}{\mathsf{MILP}}
\DeclareMathOperator{\opt}{obj}

\DeclareMathOperator{\proc}{\mathsf{REFINE}}
\DeclareMathOperator{\bproc}{\mathsf{Branch.REFINE}}
\DeclareMathOperator{\cproc}{\mathsf{Cut.REFINE}}
\usepackage{thmtools,thm-restate}

\usepackage{lineno}
\usepackage{lipsum}

\begin{document}

\title{Extended Formulations in Mixed Integer Conic Quadratic Programming
}

\titlerunning{Extended Formulations in Mixed Integer Conic Quadratic Programming}        

\author{
        Juan Pablo Vielma, Iain Dunning, Joey Huchette and Miles Lubin
}

\authorrunning{Vielma et al.} 

\institute{
	J. P. Vielma \at
              Sloan School of Management, Massachusetts Institute of Technology, Cambridge, MA 02139\\
		\email{jvielma@mit.edu}	\\
    I. Dunning \at
              Operations Research Center, Massachusetts Institute of Technology, Cambridge, MA 02139\\
    \email{idunning@mit.edu} \\
    J. Huchette \at
              Operations Research Center, Massachusetts Institute of Technology, Cambridge, MA 02139\\
    \email{huchette@mit.edu} \\
    M. Lubin \at
              Operations Research Center, Massachusetts Institute of Technology, Cambridge, MA 02139\\
    \email{mlubin@mit.edu} 
}	

\date{Jan 6th, 2015. Revised Mar 4th, 2015. Revised Feb 15th, 2016. Revised Jul 10th, 2016.}

\maketitle

\begin{abstract}
\red{In this paper we consider the use of  extended formulations in LP-based algorithms for mixed integer conic quadratic programming (MICQP).
Extended formulations have been used by Vielma, Ahmed and Nemhauser (2008) and Hijazi, Bonami and Ouorou (2013) to construct algorithms for MICQP that can provide a significant computational advantage. The first approach is based on an extended or lifted polyhedral relaxation of the Lorentz cone by Ben-Tal and Nemirovski (2001) that is extremely economical, but whose approximation quality cannot be iteratively improved. The second is based on a lifted polyhedral relaxation of the euclidean ball that can be constructed using techniques introduced by Tawarmalani and Sahinidis (2005). This relaxation is less economical, but its approximation quality can be iteratively improved. Unfortunately, while the approach of Vielma, Ahmed and Nemhauser is applicable for general MICQP problems, the approach of Hijazi, Bonami and Ouorou can only be used for MICQP problems with convex quadratic constraints. In this paper we show how a homogenization procedure can be combined with the technique by Tawarmalani and Sahinidis to adapt the extended formulation used  by Hijazi, Bonami and Ouorou to a class of conic mixed integer programming problems that include general MICQP problems.  We then compare the effectiveness of this new extended formulation against traditional and extended formulation-based algorithms for MICQP. We find that this new  formulation can be used to improve various LP-based algorithms. In particular, the formulation provides an easy-to-implement procedure that, in our benchmarks, significantly improved the performance of commercial MICQP solvers.}
\end{abstract}


\section{Introduction}\label{intro}

Most state of the art solvers for (convex) mixed integer nonlinear programming (MINLP) can be classified as nonlinear programming (NLP) based branch-and-bound algorithms or linear programming (LP) based branch-and-bound algorithms. NLP-based algorithms (e.g. \cite{BorchersMitchell1994,GuptaRavindran1985,Leyffer2001,StubbsMehrotra1995}) are the natural extension of LP-based branch-and-bound algorithms for  mixed integer linear programming (MILP) and solve the NLP relaxation of the MINLP in each node of the branch-and-bound tree. In contrast,  LP-based algorithms (e.g. \cite{abhishek2010filmint,bonami2008algorithmic,DuranGrossmann1986,FletcherLeyffer1994,Geoffrion1972,QuesadaGrossmann1992,WesterlundPettersson1995,Westerlund94}) try to avoid solving the NLP relaxation as much as possible. To achieve this they use a polyhedral relaxation  of the nonlinear constraints to construct an LP that is solved in the nodes of the branch-and-bound tree. The relaxation is iteratively refined through the branch-and-bound procedure, which is combined with sporadic solutions to the NLP relaxation. 

Traditional LP-based algorithms construct the polyhedral relaxations in the original variable space used to describe the nonlinear constraints. However, an emerging trend  is to use auxiliary variables to construct extended or \emph{lifted} polyhedral relaxations of the nonlinear constraints. This approach exploits the fact that the projection of a polyhedron can have significantly more inequalities than the original polyhedron. Hence, it is sometimes possible to construct small lifted polyhedral relaxations with the same approximation quality as significantly larger relaxations in the original space. \red{Two algorithms that use lifted polyhedral relaxations to solve quadratic MINLP problems are the \emph{lifted LP branch-and-bound}
 (LiftedLP) algorithm developed in \cite{A-Lifted-Linear-Programming} and the variant of BONMIN introduced in \cite{hijazi2013outer}. The two main differences between these approaches are the class of lifted polyhedral approximation used and the range of problems to which they are applicable. The LiftedLP algorithm from \cite{A-Lifted-Linear-Programming} uses a lifted polyhedral relaxation of the Lorentz (or second order) cone introduced by Ben-Tal and Nemirovski \cite{BenTalNemirovski2001} and hence is applicable to any mixed integer conic quadratic programming (MICQP) problem. This relaxation can be tailored  to any approximation quality, but once this approximation is fixed the relaxation cannot be easily refined. More specifically, the only known way to improve the approximation quality is to reconstruct the relaxation from scratch. In this paper we  refer to approximations with this property as \emph{static lifted polyhedral relaxations}. The inability to easily refine a static relaxation is clearly undesirable. However, the number of constraints and additional variables in the approximation by Ben-Tal and Nemirovski grows so slowly that the LiftedLP algorithm is able to use a fixed relaxation throughout its execution. While the approximation quality has to be selected a priori, this choice is often a minor calibration exercise and the resulting algorithm can significantly outperform traditional LP and NLP-based algorithms. Still, if the approximation quality is chosen poorly, the only way to adjust it is to reconstruct the relaxation and restart the algorithm. One potential solution to this issue is the variant of BONMIN from \cite{hijazi2013outer}. This algorithm uses a lifted polyhedral relaxation of the euclidean ball that can be constructed using techniques introduced by Tawarmalani and Sahinidis \cite{tawarmalani2005polyhedral}. This approximation is not as efficient as the one by Ben-Tal and Nemirovski, but it is straightforward to iteratively refine it to improve the approximation accuracy. In this paper we refer to approximations with this property as \emph{dynamic lifted polyhedral relaxations}. Using this approximation can also significantly improve the performance of traditional LP-based algorithms. However, the approach is only applicable for MINLP problems with convex quadratic constraints and hence cannot be used for all problems for which the LiftedLP algorithm is applicable.   }

\red{In this paper we study the effectiveness of dynamic lifted polyhedral relaxations for the solution of general MICQP problems. In particular, we introduce three dynamic lifted polyhedral relaxations of the Lorentz cone. The first relaxation is a straightforward adaptation of a portion of the original static relaxation introduced by Ben-Tal and Nemirovski \cite{BenTalNemirovski2001}. The second relaxation we construct by combining a homogenization procedure with the relaxation technique by Tawarmalani and Sahinidis \cite{tawarmalani2005polyhedral}. This yields a relaxation for various conic sets including the Lorentz cone. The final relaxation is a simple combination of the previous two. All three approximations can be used to improve the LiftedLP algorithm from \cite{A-Lifted-Linear-Programming}. However, we show that they can also be used directly by interpreting them as lifted conic quadratic formulations of the Lorentz cone, which lead to extended MICQP re-formulations of any MICQP problem.  We computationally evaluate the effectiveness of all three formulations when used by themselves and as an improvement of the  LiftedLP algorithm. Our computational results show that the extended MICQP re-formulation version  of the homogenized variant of the Tawarmalani and Sahinidis relaxation can provide a significant computational advantage over traditional LP and NLP algorithms and other lifted polyhedral relaxation approaches.}

The remainder of this paper is organized as follows. Section~\ref{previouswork} provides a brief introduction to LP-based algorithms for MINLP, a detailed description of existing lifted polyhedral relaxations and a brief description of their use in LP-based algorithms. This section also compares and contrasts the advantages of extremely economical \red{static} lifted polyhedral relaxations that cannot be iteratively refined, to less economical \red{dynamic} relaxations that can be refined. In particular, the section introduces the relaxation from \cite{hijazi2013outer,tawarmalani2005polyhedral} as an extremely effective compromise between these objectives, that nonetheless has some modeling limitations. Then, Section~\ref{dynamicsec} shows how the modeling limitations of the relaxation from \cite{hijazi2013outer,tawarmalani2005polyhedral} can be eliminated through a homogenization procedure that adapts it to all MICQP problems. \red{In this section we also discuss the relation of the homogenization procedure with existing perspective reformulation techniques.} The new lifted polyhedral relaxation leads to some variants of the LiftedLP algorithm from \cite{A-Lifted-Linear-Programming}, which are described in detail in Section~\ref{algosec}. \red{ This section also shows how the relaxations can be used directly by interpreting them as extended MICQP reformulations.} Section~\ref{compres} then provides a computational comparison of these algorithms and traditional LP-based and NLP-based algorithms. Finally, Section~\ref{futurework} discusses some possible improvements for specific classes of problems \red{ and some open questions concerning the best possible approximation quality for dynamic lifted polyhedral approximations}. Omitted proofs and additional figures and tables are provided in Appendices \ref{proofappendix} and \ref{resultappendix}.

To improve readability we use the convention of naming polyhedral sets in a natural original variable space with upper case letters (e.g. $P$) and non-polyhedral sets in the same space with bold upper case letters (e.g. $\mathbf{C}$). Lifted sets defined in the space of original and auxiliary variables additionally appear with a ``hat'' (e.g. $\widehat{P}$ for lifted polyhedral sets and $\widehat{\mathbf{C}}$ for lifted non-polyhedral sets). Additional definitions and notation will be introduced as they are needed through the paper.

\section{LP Based Algorithms and Lifted Polyhedral Relaxations}\label{previouswork}

We consider a MICQP of the form  
\begin{subequations}\label{micp}
\begin{alignat}{3}
  \label{probstart}\opt_{\micp}:=&\max & \quad  c x  \\
 \notag&s.t.\\
 &      & E x  &\leq h,  &\\
 &       & \norm{A^l x + b^l}_2 &\leq a^l \cdot x + b_0^l, & \quad& \forall l\in \set{1,\ldots,q},\label{conicc}\\
 \label{relaxend}& & x &\in \Real^{n},\\
\label{probend}& & x_j &\in \Int^{n},&\quad& \forall j\in I,
 \end{alignat}
\end{subequations}
 where $c\in \Real^n$, $E \in \Real^{m\times n}$, $h\in \Real^{m}$, $\norm{\cdot}_2$ is the euclidean norm, \red{$a^l\cdot x$ is the inner product between $a^l$ and $b$,} and where for each $l\in \set{1,\ldots,q}$ there exist $d_l$ such that $A^l\in \Real^{d_l\times n}$, $b^l\in\Real^{d_l}$, $a^l\in \Real^n$ and $b^l_0\in \Real$. We denote problem \eqref{micp} as $\micp$.

 We start by noting that constraint \eqref{conicc} can be written as
 \begin{align*}
    A^l x + b^l = y,\quad a^l \cdot x + b_0^l=y_0,\quad (y_0,y)&\in \mathbf{L}^{d_l},
 \end{align*}
 where $\mathbf{L}^d$  is the $(d+1)$-dimensional Lorentz cone given by
   $\mathbf{L}^d:=\set{ (y_0,y)\in \Real^{d+1}\,:\, \norm{y}_2\leq y_0}$.
Hence, to get an LP-based algorithm for \eqref{micp}  it suffices to construct a polyhedral relaxation of $\mathbf{L}^d$. 

We can construct  a polyhedral relaxation of $\mathbf{L}^d$ through its semi-infinite representation given by 
\[\mathbf{L}^d=\set{ (y_0,y)\in \Real^{d+1}\,:\, \sum_{j=1}^d \omega_j y_j\leq y_0, \forall \omega \in \mathbf{S}^{d-1}}.\]
where $\mathbf{S}^{d-1}:=\set{ \omega\in \Real^d\,:\,  \norm{\omega}_2=1}$ is the unit sphere. Then for any $\Omega\subseteq \mathbf{S}^{d-1}$ with $\abs{\Omega}<\infty$, the set 
\begin{equation}\label{simplerelaxation}
O^d\bra{\Omega}:=\set{ (y_0,y)\in \Real^{d+1}\,:\, \sum_{j=1}^d \omega_j y_j\leq y_0, \forall \omega \in \Omega}
\end{equation}
is a polyhedron such that $\mathbf{L}^d\subseteq O^d\bra{\Omega}$.   This polyhedral relaxation can easily be refined by noting that if $(y_0,y) \in O^d\bra{\Omega}\setminus \mathbf{L}^d$ then 
$\sum_{j=1}^d \omega_j(y) y_j > y_0$ where $\omega(y)=\frac{y}{\;\,\norm{y}_2}\in \mathbf{S}^{d-1}$ and hence $(y_0,y) \notin O^d\bra{\Omega\cup \set{\omega(y)}}$. Polyhedral relaxations such as $O^d$, which do not use auxiliary variables and can be iteratively refined, lead to traditional LP-based algorithms usually known as outer approximation algorithms. The simplest version of such algorithms is the MILP-based algorithm by Duran and Grossmann \cite{DuranGrossmann1986}. In the context of MICQP, this algorithm is based on the MILP relaxation of $\micp$ given by 
\begin{subequations}\label{milp}
\begin{alignat}{3}
  \opt_{\milp\bra{\set{\Omega^l}_{l=1}^k}}:=&\max & \quad cx  \\
 \notag&s.t.\\
 &      & E x  &\leq h,  &\\
 &      & A^l x+b^l &= y^l,&\quad&\forall  l\in \set{1,\ldots,q},\\
 &      & a^l \cdot x + b_0^l&=y_0^l,&\quad&\forall l\in \set{1,\ldots,q},\\
 &&     \bra{y^l_0,y^l} &\in O^{d_l}\bra{\Omega^l},&\quad&\forall l\in \set{1,\ldots,q},\\
 & & x &\in \Real^{n},\\
& & x_j &\in \Int^{n},&\quad&\forall j\in I.
 \end{alignat}
\end{subequations}
We denote this relaxation by $\milp\bra{\set{\Omega^l}_{l=1}^q}$. 

A basic version of the MILP-based outer approximation algorithm for solving $\micp$ is given in Algorithm~\ref{basicoa}, where for simplicity we have assumed the existence of an initial approximating set $\set{\tilde{\Omega}^l}_{l=1}^q$ for which $\milp\bra{\set{\tilde{\Omega}^l}_{l=1}^q}$ is bounded.

 \begin{algorithm}[htpb]
 \DontPrintSemicolon
  Set $\Omega^l=\tilde{\Omega}^l$ for all $l\in\set{1,\ldots,q}$.\;
  Solve $\milp\bra{\set{\Omega^l}_{l=1}^q}$.\;\label{solvealg}
  \eIf{$\milp\bra{\set{\Omega^l}_{l=1}^q}$ is infeasible}{Declare $\micp$ infeasible. }
  {Let $\bra{\bar{x},\bar{y},\set{\bra{\bar{y}^l_0,\bar{y}^l}}_{l=1}^q}$ be an optimal solution to $\milp\bra{\set{\Omega^l}_{l=1}^q}$.\;}
  \eIf{$\bra{\bar{y}^l_0,\bar{y}^l}\in \mathbf{L}^{d_l}$ for all $l\in \set{1,\ldots,q}$}{
  Return $\bra{\bar{x},\bar{y}}$ as the optimal solution to $\micp$.
  }
  {
  \For{$l= 1$ \KwTo $q$}{
  \If{$\bra{\bar{y}^l_0,\bar{y}^l}\notin \mathbf{L}^{d_l}$}{
    Set $\Omega^l=\Omega^l\cup \set{\omega\bra{\bar{y}^l}}$.
  }
  }} 
  Go to \ref{solvealg}\;
    \caption{A Basic LP Based Algorithm.}\label{basicoa}
\end{algorithm} 

Algorithm~\ref{basicoa} converges in finite time for pure integer problems ($p=0$) with bounded feasible regions and simple modifications can ensure convergence for the mixed integer case too (e.g. see \cite{bonami2008algorithmic,DuranGrossmann1986,FletcherLeyffer1994}).  However, the following example  from \cite{hijazi2013outer} shows that an exponential number of iterations may be needed even for very simple pure integer problems.

\begin{example}\label{ex1}
   
Let $\mathbf{F}^n:=\left\{x\in \mathbb{R}^n\,:\, \sum_{j=1}^n \left(x_j-\frac{1}{2}\right)^2 \leq \frac{n-1}{4}\right\}$ and consider the description of $\mathbf{F}^n$ as the feasible region of $\micp$ given by 
\begin{subequations}
\begin{alignat}{3}
x_j-\frac{1}{2}&=y_j,&\quad& \forall j\in \set{1,\ldots,n},\label{Ex11}\\
\sqrt{\frac{n-1}{4}}&=y_0,\label{Ex12}\\
\bra{y_0,y}&\in \mathbf{L}^n.
\end{alignat}
\end{subequations}
It is shown in \cite{hijazi2013outer} that if $P\subseteq \Real^n$ is a polyhedron such that $\mathbf{F}^n\subseteq P$ and $P\cap \mathbb{Z}^n=\emptyset$, then $P$ must have at least $2^n$ facets. In particular, $\set{\bra{x,y_0,y}\in \Real^{2n+1}\,:\, \text{\eqref{Ex11}--\eqref{Ex12}},\quad \bra{y_0,y}\in O^n(\Omega)}\neq \emptyset $ for any $\Omega\subseteq \mathbf{S}^{n-1}$ such that $\abs{\Omega}<2^n$ and hence Algorithm~\ref{basicoa} will require an exponential number of iterations to prove that $\mathbf{F}^n\cap \mathbb{Z}^n=\emptyset$.
\end{example}

We can check that $\tilde{F}^n:=\left\{x\in \mathbb{R}^n\,:\, \sum_{j=1}^n \abs{x_j-\frac{1}{2}} \leq \frac{\sqrt{n}\sqrt{n-1}}{2}\right\}$ is such that $\mathbf{F}^n\subseteq \tilde{F}^n$ and $\tilde{F}^n\cap \mathbb{Z}^n=\emptyset$. As expected $\tilde{F}^n$ has an exponential number of facets, but using standard LP techniques we can construct the lifted LP representation of  $\tilde{F}^n$ given by 
\begin{alignat*}{3}
x_j-\frac{1}{2}&\leq z_j,&\quad& \forall j\in \set{1,\ldots,n},\\
\frac{1}{2}-x_j&\leq z_j,&\quad& \forall j\in \set{1,\ldots,n},\\
\sum_{j=1}^n z_j&\leq \frac{\sqrt{n}\sqrt{n-1}}{2}.
\end{alignat*}
Then, this formulation is a lifted polyhedral relaxation of $\tilde{F}^n$ with a linear number of constraints and additional variables that can be used to prove $\mathbf{F}^n\cap \mathbb{Z}^n=\emptyset$. Hence, the exponential behavior from Example~\ref{ex1} can be avoided by using auxiliary variables.
In the following subsection we describe two lifted polyhedral relaxation approaches that can be used to prove $\mathbf{F}^n\cap \mathbb{Z}^n=\emptyset$ and are also applicable to more general MICQPs.

\subsection{Lifted Polyhedral Relaxations}

We begin with a formal definition of a lifted polyhedral relaxation of $\mathbf{L}^d$ (the definition naturally generalizes to arbitrary convex sets). 
\begin{definition}\label{approximationdef} A polyhedron $P:=\set{\bra{y_0,y,z}\in \mathbb{R}^{d+1+m_2}\,:\, A\begin{pmatrix}y_0\\y\end{pmatrix} +Dz\leq b}$ for $A \in \Real^{m_1\times \bra{d+1}}$, $D \in \Real^{m_1\times m_2}$ and $b \in \Real^{m_1}$ is a \emph{lifted polyhedral relaxation} of $\mathbf{L}^d$ if and only if
\[ \mathbf{L}^d\subseteq \proj_{\bra{y_0,y}}\bra{ P},\]
where $\proj_{\bra{y_0,y}}$ is the orthogonal projection onto the $\bra{y_0,y}$ variables.
\end{definition}
Two desirable properties of a lifted polyhedral relaxation are small number of inequalities ($m_1$) and number of auxiliary variables ($m_2$), and high approximation quality (quantified through various measures). For instance, 
the first approximation we consider additionally satisfies 
\begin{equation}\label{appprox}
  \mathbf{L}^d\subseteq \proj_{\bra{y_0,y}}\bra{ P} \subseteq \set{(y_0,y)\in \Real^{d+1}\,:\, ||y||_2\leq (1+\varepsilon) y_0},
\end{equation}
which we refer to as having approximation quality $\varepsilon$, and $m_1,m_2\leq \psi\bra{d,\ln\bra{1/\varepsilon}}$ for a low degree polynomial $\psi$. This approximation was introduced by  Ben-Tal and Nemirovski \cite{BenTalNemirovski2001} (see also \cite{glineur2000}) and its construction begins by building the lifted polyhedral relaxation of $\mathbf{L}^{2}$ given in the following proposition.
\begin{proposition}\label{2dnemi}
Let 
\[
   \widehat{N}^2_s:=\set{(y_0,y_1,y_2,v)\in\Real_+\times \Real^2\times \Real^{2s}\,:\, \begin{alignedat}{3}
  y_0&=v_{2s-1}\cos\left(\frac{\pi}{2^s}\right)+v_{2s} \sin\left(\frac{\pi}{2^s}\right),\\
  v_{1}&=y_1 \cos\left(\pi\right)+y_2\sin\left(\pi\right), \quad&&\\
  v_2&\geq \left|y_2 \cos\left(\pi\right)-y_1\sin\left(\pi\right)\right|, \quad&&\\
  v_{2\bra{i+1}-1}&=v_{2i-1} \cos\left(\frac{\pi}{2^i}\right)+v_{2i}\sin\left(\frac{\pi}{2^i}\right), &\quad& \forall i\in \set{1,\ldots,s-1},\\
  v_{2\bra{i+1}}&\geq \left|v_{2i} \cos\left(\frac{\pi}{2^i}\right)-v_{2i-1}\sin\left(\frac{\pi}{2^i}\right)\right|, &\quad& \forall i\in \set{1,\ldots,s-1}&&\end{alignedat}}.
 \]
 Then, $\widehat{N}^2_s$ is a lifted polyhedral relaxation of  $\mathbf{L}^{2}$ with approximation quality $\varepsilon=\cos\bra{\frac{\pi}{2^s}}^{-1}-1$.
\end{proposition}
To construct a lifted polyhedral relaxation of $\mathbf{L}^{d}$ for $d>2$, Ben-Tal and Nemirovski observe that
 \begin{align*}
   \mathbf{L}^d=\set{ (y_0,y)\in \Real^{d+1}\,:\, \exists t\in \Real^{\lfloor d/2\rfloor}\text{ s.t. } \begin{alignedat}{3}\sum_{k=1}^{\lceil d/2\rceil} t_k^2  &\leq y_0^2,\\  y_{2k-1}^2+y_{2k}^2 &\leq t_k^2,  &\quad& \forall  k\in \set{1,\ldots, \lfloor d/2\rfloor},\\
                                                                                                                                                                          t_{\lceil d/2\rceil}&=y_d, &\quad& \text{if $d$ is odd}\end{alignedat} }.
\end{align*}
Recursively repeating this construction for the $\lceil d/2\rceil+1$ dimensional Lorentz cone inside this representation we can construct a version with additional variables that only uses $d-1$ three dimensional Lorentz cones. This construction was denoted the \emph{tower of variables} by Ben-Tal and Nemirovski and  yields the lifted representation $\mathbf{L}^{d}=\proj_{\bra{y_0,y}}\bra{ \widehat{\mathbf{T}}^{d}}$  for
 \begin{equation}\label{nolineartower}
   \widehat{\mathbf{T}}^d=\set{ \bra{y_0,y,t}\in \Real^{d+1+R(d)}\,:\,\begin{alignedat}{5}
   y_0&=t^{K}_1,&&&\\
   t^0_i&=y_i, &&\forall i&&\in \{1,\ldots, d\},\\
   (t^{k+1}_i,t^k_{2i-1},t^k_{2i})&\in \mathbf{L}^2, &\quad&\forall i&&\in\{1,\ldots,\lfloor r_k/2 \rfloor\},\\
                &&& \;\, k&&\in\{0,\ldots,K-1\},\\
   t^k_{r_k}&=t^{k+1}_{\lceil r_k/2 \rceil},          &&\forall k&&\in\{0,\ldots,K-1\} \\
                  &    &&&&\text{s.t. $r_k$ is odd} \end{alignedat} }
 \end{equation}
 where  $K=\lceil \log_2(d)\rceil$, $\{r_k\}_{k=0}^K$ is defined by the recursion $r_0=d$, ${r_{k+1}=\lceil r_k /2 \rceil}$ for $k\in \{0,\ldots,K-1\}$ and $R(d)=\sum_{k=0}^K r_k$. A lifted polyhedral relaxation of $\mathbf{L}^d$ can then be constructed by replacing every occurrence of $\mathbf{L}^2$ in $\widehat{\mathbf{T}}^d$ with an appropriate polyhedral relaxation. The original approximation of Ben-Tal and Nemirovski uses $\widehat{N}^2_s$ as a lifted polyhedral relaxation of $\mathbf{L}^2$ and leads to essentially the  smallest possible approximation of  $\mathbf{L}^d$ (see Section 3 of \cite{BenTalNemirovski2001}). However, we will see that using alternative approximations could lead to a computational advantage in our context, so the following proposition first presents a generic version of the approximation and then specializes it to the original Ben-Tal and Nemirovski approximation.
 \begin{proposition}\label{nemiprop}
 Let   $K=\lceil \log_2(d)\rceil$, $\{r_k\}_{k=0}^K$ be defined by the recursion $r_0=d$, ${r_{k+1}=\lceil r_k /2 \rceil}$ for $k\in \{0,\ldots,K-1\}$ and $R(d)=\sum_{k=0}^K r_k$. Let $D:=\set{d_{i,k}\,:\, i\in\{1,\ldots,\lfloor r_k/2 \rfloor\},\,k\in \set{0,\ldots,K-1}}\subseteq \mathbb{Z}_+$, $G(D):=\sum_{k=0}^{K-1}\sum_{i=1}^{\lfloor r_k/2 \rfloor} d_{i,k}$ and  
$\mathcal{P}:=\set{P_{i,k}\,:\, i\in\{1,\ldots,\lfloor r_k/2 \rfloor\},\,k\in \set{0,\ldots,K-1} }$ be a family of polyhedra such that, for each $i$ and $k$, $P_{i,k}\subseteq \Real^{3+d_{i,k}}$ is a (lifted) polyhedral relaxation of $\mathbf{L}^2$ with approximation quality $\varepsilon_k\in(0,1]$.  Then 
 \begin{equation*}
   \widehat{T}^d\bra{\mathcal{P},D}:=\set{ (y_0,y,t,v)\in \times\Real^{d+1+R(d)+G(D)}\,:\,\begin{alignedat}{5}
   y_0&=t^{K}_1,&&&\\
   t^0_i&=y_i, &&\forall i&&\in \{1,\ldots, d\},\\
   (t^{k+1}_i,t^k_{2i-1},t^k_{2i},v^{i,k})&\in P_{i,k}, &\quad&\forall i&&\in\{1,\ldots,\lfloor r_k/2 \rfloor\},\\
                &&& \;\, k&&\in\{0,\ldots,K-1\},\\
   t^k_{t_k}&=t^{k+1}_{\lceil r_k/2 \rceil},          &&\forall k&&\in\{0,\ldots,K-1\} \\
                  &    &&&&\text{s.t. $r_k$ is odd} \end{alignedat} }.
 \end{equation*}
is a lifted polyhedral relaxation of $\mathbf{L}^d$ with approximation quality $\varepsilon=\bra{\prod_{k=0}^{K-1} 1+\varepsilon_k} -1$. In particular, let  
$\widehat{L}^d_s:=\widehat{T}^d\bra{\mathcal{N}_s,D_{2s}}$ where $\mathcal{N}_s:=\set{\widehat{N}^2_s\,:\, i\in\{1,\ldots,\lfloor r_k/2 \rfloor\},\,k\in \set{0,\ldots,K-1}}$\footnote{Or more precisely, $\mathcal{N}:=\set{N_{i,k}\,:\, i\in\{1,\ldots,\lfloor r_k/2 \rfloor\},\,k\in \set{0,\ldots,K-1}}$ where $N_{i,k}=\widehat{N}^2_s$ for each $i,k$}
and  \[D_{2s}:=\set{2s\,:\, i\in\{1,\ldots,\lfloor r_k/2 \rfloor\},\,k\in \set{0,\ldots,K-1}}.\]
Then, for any $\varepsilon\in(0,1)$ there exists an $s\bra{\varepsilon}\in \Int_+$ such that $\widehat{L}^d_{s\bra{\varepsilon}}$ has $O\bra{d\ln\bra{K/\varepsilon}}$ variables and constraints, and is a lifted polyhedral relaxation of $\mathbf{L}^d$ with approximation quality $\varepsilon$.
 \end{proposition}

It is not hard to see that for an appropriately chosen $\varepsilon$,  $\widehat{L}^d_{s\bra{\varepsilon}}$ is a lifted polyhedral relaxation with a polynomial number of variables and constraints that can prove $\mathbf{F}^n\cap \mathbb{Z}^n=\emptyset$. However, this can also be achieved with the tower of variables construction without the use of $\widehat{N}^2_s$. To show that we will use the following straightforward corollary of  Proposition~\ref{nemiprop}, which approximates each $\mathbf{L}^2$ in $\widehat{\mathbf{T}}^d$ with non-lifted approximation $O^d\bra{\Omega}$ introduced in \eqref{simplerelaxation}.

 \begin{corollary}\label{dynamictower}
 Let $\Omega:=\set{\Omega_{i,k}\,:\, i\in\{1,\ldots,\lfloor t_k/2 \rfloor\},\quad k\in\{0,\ldots,K-1\}}$ be such that $\Omega_{i,k}\subseteq \mathbf{S}^1$ and $\abs{\Omega_{i,k}}<\infty$ for all $i,k$. For any such, $\Omega$ define  $\widehat{T}^d\bra{\Omega}:=\widehat{T}^d\bra{\mathcal{O},D_0}$ for \[\mathcal{O}:=\set{O^2\bra{\Omega_{i,k}}\,:\,  i\in\{1,\ldots,\lfloor r_k/2 \rfloor\},\ k\in \set{0,\ldots,K-1}}\] and $D_0:=\set{0\,:\,  i\in\{1,\ldots,\lfloor r_k/2 \rfloor\},\ k\in \set{0,\ldots,K-1}}$.
Then $\widehat{T}^d\bra{\Omega}$, is a lifted polyhedral relaxation of $\mathbf{L}^d$ with \[\sum_{k=0}^{K-1} \sum_{i=1}^{\lfloor t_k/2 \rfloor} \abs{\Omega_{i,k}}\leq \bra{d-1} \max_{\Omega_{i,k}\in \Omega} \abs{\Omega_{i,k}}\] constraints and $d-2$ auxiliary variables\footnote{After removing redundant variables through the equalities of the formulation.}.
 \end{corollary} 

The following example shows how the tower of variables approximation based on both $\widehat{N}^2_s$ and $O^2\bra{\Omega}$ result in small lifted polyhedral approximations that prove $\mathbf{F}^n\cap \mathbb{Z}^n=\emptyset$.
\begin{example}\label{exampletower}Let $\Omega_{0,0}=\set{\bra{s_1/\sqrt{2},s_2/\sqrt{2}}}_{s\in \set{-1,1}^2}$, then we can check that \[\proj_{\bra{y_0,y}}\bra{\widehat{N}^2_s}=O^2\bra{\Omega_{0,0}}=\set{\bra{y_0,y}\in \Real_+\times \Real^2\,:\, \abs{y_1}+\abs{y_2}\leq \sqrt{2} y_0}.\]
Then, if we let $\Omega:=\set{\Omega_{0,0}\,:\, i\in\{1,\ldots,\lfloor t_k/2 \rfloor\},\quad k\in\{0,\ldots,K-1\}}$ then 
 \[\proj_{\bra{y_0,y}}\bra{\widehat{L}^d_2}=\proj_{\bra{y_0,y}}\bra{\widehat{T}^d\bra{\Omega}}.\] We can also check that 
 \[\min\set{y_0\,:\, \bra{y_0,\bar{y}}\in \proj_{\bra{y_0,y}}\bra{\widehat{L}^d_2}}=\frac{1}{2}\bra{\frac{2}{\sqrt{2}}}^{\log_2(d)}=\sqrt{\frac{d}{4}}\]
 for all $\bar{y}\in \set{0,1}^d$.
 Hence 
 \begin{equation}\label{empty}\set{\bra{x,y,y_0}\in \Real^{2n+1}\,:\, x_j-\frac{1}{2}=y_j\; \forall j\in \set{1,\ldots,n},\; \sqrt{\frac{n-1}{4}}=y_0,\; \bra{y_0,y}\in \proj_{\bra{y_0,y}}\bra{P}}\cap \mathbb{Z}^n=\emptyset\end{equation}
\end{example}
for $P=\widehat{L}^d_2$  or $P=\widehat{T}^d\bra{\Omega}$ and  both of these formulations can be used to prove $\mathbf{F}^n\cap \mathbb{Z}^n=\emptyset$. Formulation $\widehat{T}^d\bra{\Omega}$ achieves this with $4n-4$ constraints and $n-2$ auxiliary variables and $\widehat{L}^d_2$ achieves it with $4n-4$ constraints and $2n-1$ auxiliary variables. 

From Examples~\ref{ex1} and \ref{exampletower} we have that the number of constraints in the projection onto the $\bra{y_0,y}$ variables of both $\widehat{T}^d\bra{\Omega}$ and $\widehat{L}^d_s$ can have a number of facets that is exponential on their original number of variables and constraints. Hence, both formulations   can  provide a significant advantage over the traditional outer approximation formulation $O^d(\Omega)$.  However, the full potential of the approximation by Ben-Tal and Nemirovski approximation is only achieved when $\widehat{N}^2_s$ is used (e.g. as in $\widehat{L}^d_s$). For instance, if $d=2$, $\widehat{L}^d_s$ reduces to $\widehat{N}^2_s$, which has a linear (in $s$) number of variables and constraints and which has a projection onto the $\bra{y_0,y}$ variables with an exponential  number of constraints\footnote{For instance, one can check that $\set{y\in \Real^2\,:\, \proj_{\bra{1,y}}\bra{\widehat{N}^2_s}}$ is a regular $2^s$ polygon.}. In contrast, for $d=2$, $\widehat{T}^d\bra{\Omega}$  reduces to $O^d(\Omega)$ and hence we do not get any constraint multiplying effect through projections. Nonetheless,  $\widehat{T}^d\bra{\Omega}$ does have one advantage over $\widehat{L}^d_s$. While $\widehat{T}^d\bra{\Omega}$ can be iteratively refined by simply augmenting $\Omega$, the only way to refine  $\widehat{L}^d_s$ is to refine each (or some) of the approximations $\widehat{N}^2_s$ used in its construction. Regrettably, the complexity of $\widehat{N}^2_s$ does not easily lend itself to an iterative refinement and the only known way to refine $\widehat{N}^2_s$ is to replace it with $N^2_{s'}$ for $s'>s$ constructed from scratch. Fortunately, there is an alternative lifted approximation that provides a middle ground between $\widehat{T}^d\bra{\Omega}$ and $\widehat{L}^d_s$. This approximation can be constructed using the following proposition from \cite{tawarmalani2005polyhedral}.
\begin{proposition}\label{separableprop} Let $f_j:\Real \to \Real$ be strictly convex differentiable functions for each $j\in \set{1,\ldots,n}$ and let  $\Gamma:=\set{\Gamma_{j}\,:\, j\in \set{1,\ldots,n}}$ be such that $\Gamma_j\subseteq \Real$ and $\abs{\Gamma_j}<\infty$. If $\mathbf{G}:=\set{(y_0,y)\in \Real^{n+1}\,:\, \sum_{j=1}^n f_j(y_j)\leq y_0}$, then 
\begin{subequations}\label{generalsep}
\begin{alignat}{3}
f_j\bra{\gamma} + f_j'\bra{\gamma}\bra{y_j -\gamma}  &\leq w_j,&\quad& \forall \gamma\in \Gamma_j,\quad j\in \set{1,\ldots,n},\\
\sum_{j=1}^n w_j &\leq y_0
\end{alignat}
\end{subequations}
is a lifted polyhedral relaxation of $\mathbf{G}$ with $1+\sum_{j=1}^n \abs{\Gamma_j} $ constraints and whose
 projection onto the $(y_0,y)$ variables can have up to $\prod_{j=1}^n \abs{\Gamma_j}$ constraints. 
\end{proposition}
While set $\mathbf{G}$ in Proposition~\ref{generalsep} has a somewhat restrictive structure,  \cite{hijazi2013outer} uses it to constructs the following linear sized polyhedral approximation that is able to prove $\mathbf{F}^n\cap \mathbb{Z}^n=\emptyset$.
\begin{example}\label{pierreexample} An alternative extended formulation of set $\mathbf{F}^n$ from Example~\ref{ex1} is given by 
\begin{subequations}\label{pierre}
\begin{alignat}{3}
x_j-\frac{1}{2}&=y_j,&\quad& \forall j\in \set{1,\ldots,n},\\
y_0&=\sqrt{\frac{n-1}{4}}\label{pierre0}\\
\sum_{j=1}^n y_j^2 &\leq y_0^2.\label{pierre2}
\end{alignat}
\end{subequations}
Constraint \eqref{pierre2} is not convex, but because $y_0$ is fixed to a constant, without loss of generality, we can replace  \eqref{pierre0} by $y_0=\frac{n-1}{4}$ and  \eqref{pierre2} by $\sum_{j=1}^n y_j^2 \leq y_0$. Applying Proposition~\ref{separableprop} for $f_j(y_j)=y_j^2$  and 
$\Gamma_j=\set{-1/2,1/2}$ for all $j\in \set{1,\ldots,n}$ yields the polyhedral relaxation of \eqref{pierre}, given by
\begin{subequations}\label{pierreap}
\begin{alignat}{3}
x_j-\frac{1}{2}&=y_j,&\quad& \forall j\in \set{1,\ldots,n},\\
 y_j -\frac{1}{4} &\leq w_j,&\quad& \forall j\in \set{1,\ldots,n},\\
- y_j - \frac{1}{4} &\leq w_j,&\quad& \forall j\in \set{1,\ldots,n},\\
\sum_{j=1}^n w_j &\leq \frac{n-1}{4}.
\end{alignat}
\end{subequations}
We can then check that  $\set{\bra{x,y,w}\in \Real^{3n}\,:\, \text{\eqref{pierreap}},\quad x\in \mathbb{Z}^n}=\emptyset$.  
\end{example}

Formulation \eqref{generalsep} shares with $\widehat{T}^d\bra{\Omega}$ the possibility of being iteratively refined. Furthermore, while it does not reach the efficiency of $\widehat{L}^d_s$, it does provide a tangible improvement over $\widehat{T}^d\bra{\Omega}$. For instance for $d=2$ \eqref{generalsep} with $r_1+r_2+1$ constraints projects to a polyhedron with at most $r_1\times r_2$ constraints. Hence \eqref{generalsep} can have an up to quadratic constraint multiplying effect (cf. the exponential constraint multiplying effect of $L^2_s$ and the lack of  constraint multiplying effect of $T^2\bra{\Omega}$). However, while formulation \eqref{generalsep} yields a polyhedral relaxation of $\mathbf{F}^n$,  $\norm{y}_2$ cannot be written as a sum of univariate functions and hence formulation  \eqref{generalsep} cannot be used directly to construct a polyhedral relaxation of $\mathbf{L}^d$. Fortunately, as we will show in the following section, a simple homogenization technique can be used to adapt \eqref{generalsep} to yield lifted polyhedral relaxations of $\mathbf{L}^d$ and other \emph{conic} sets. These lifted polyhedral relaxations share with \eqref{generalsep} the possibility of being iteratively refined. To simplify exposition, from now on we refer to relaxations that can be iteratively refined as \emph{dynamic} relaxations and to those that cannot be easily refined as \emph{static} relaxations. 

\section{Dynamic Lifted Polyhedral Relaxations of Conic Sets}\label{dynamicsec}

Proposition~\ref{separableprop} cannot be directly used to construct a polyhedral relaxation of  $\mathbf{L}^d$, but it can be used to construct relaxations of the euclidean ball $\mathbf{B}^d:=\set{y\in \mathbb{R}^d\,:\, \sum_{i=1}^d y_i^2\leq  1}=\set{y\in \mathbb{R}^d\,:\, \bra{1,y}\in \mathbf{L}^d}$. Furthermore, $\mathbf{L}^d$ can be constructed from $\mathbf{B}^d$ through the homogenization $\mathbf{L}^d=\cone\bra{\set{1}\times \mathbf{B}^d}$. Hence, any relaxation of $\mathbf{B}^d$ can be transformed to a relaxation of $\mathbf{L}^d$ through a similar homogenization. This approach can be formally stated through the following proposition, which we prove in Appendix~\ref{proofappendix} (see Section~\ref{futurework} for an alternative derivation).

\begin{proposition}\label{nosepprop} Let $f_j:\Real\to \Real $ be  closed convex functions such that $\lim_{x\to \infty} \frac{f_j(x)}{\abs{x}}=\infty$ for all $j\in \set{1,\ldots, d}$ so that the closure of the perspective function of $f$ is given by 
\[ \bra{cl\tilde{f}}(t,x)=\begin{cases}t f(x/t) &t>0\\ 0 & x=0 \text{ and } t=0\\\infty &\text{o.w.}\end{cases}.\]
If $\mathbf{C}:=\set{y\in \Real^d\,:\, \sum_{j=1}^d f_j(y_j)\leq 1}$ and 
\begin{equation}\label{eee}
\widehat{\mathbf{C}}:= \set{\bra{y_0,y,w}\in \Real^{2d+1}\,:\, \bra{cl\tilde{f_j}}(y_0,y_j)\leq w_j,\quad \forall j\in \set{1,\ldots,d}, \quad \sum_{j=1}^d w_j\leq y_0},
\end{equation} 
then 
$\cone\bra{\set{1}\times \mathbf{C}}=\proj_{\bra{y_0,y}}\bra{\widehat{\mathbf{C}}}$.

Furthermore, if functions $f_j$ are differentiable, then for any   $\Gamma:=\set{\Gamma_{j}\,:\, j\in \set{1,\ldots,n}}$ such that $\Gamma_j\subseteq \Real$ and $\abs{\Gamma_j}<\infty$, 
the set 
\begin{equation*}
\widehat{C}\bra{\Gamma}:=\set{\bra{y_0,y,w}\in \Real^{2d+1}\,:\,
\begin{alignedat}{3}
\bra{f\bra{\gamma}-\gamma f'\bra{\gamma}} y_0 + f'\bra{\gamma}y_j   &\leq w_j,&\quad& \forall \gamma \in \Gamma_j,\, j\in \set{1,\ldots,n}\\
\sum_{j=1}^n w_j &\leq y_0
\end{alignedat}}
\end{equation*}
is a lifted polyhedral relaxation of $\cone\bra{\set{1}\times \mathbf{C}}$  with $1+\sum_{j=1}^n \abs{\Gamma_j}$ constraints.  If the functions are additionally strictly convex the 
 projection onto the  variables $(y_0,y)$ of this relaxation can have up to $\prod_{j=1}^n\abs{\Gamma_j}$ constraints.

 Finally, if  $\bra{y_0,y,w}\in \widehat{C}\bra{\Gamma}\setminus \widehat{\mathbf{C}}$, then there exists $j\in \set{1,\ldots, n}$ and $\gamma\in \Real$ such that $\bra{y_0,y,w}\notin \widehat{C}\bra{\Gamma}$ if we augment $\Gamma_j$ to $\Gamma_j\cup\set{\gamma}$.
\end{proposition}

Obtaining a dynamic lifted polyhedral relaxation for $\mathbf{L}^d$  is a direct corollary of Proposition~\ref{nosepprop} by letting $f_j(x_j)=x_j^2$ and using $\mathbf{L}^d=\cone\bra{\set{1}\times \mathbf{B}^d}$  ( See \cite[p. 109]{ben2001lectures} for an alternative derivation). One notable difference between the general approximation from Proposition~\ref{nosepprop} and the approximation of $\mathbf{L}^d$ from Corollary~\ref{coro2} below is the inclusion of constraint $y_0\geq 0$ in the later. This constraint is excluded in Proposition~\ref{nosepprop}, but  is implied by the nonlinear constraints defining $\widehat{\mathbf{C}}$ and can be arbitrarily approximated by the linear constraints of $\widehat{C}\bra{\Gamma}$ by appropriately selecting $\Gamma$ (e.g. see the proof of Lemma~\ref{lemma3} in Appendix~\ref{proofappendix}). The constraint is explicitly included in Corollary~\ref{coro2} to ensure conic quadratic representability of $\widehat{\mathbf{H}}^d$ and for computational convenience and numerical stability in $\widehat{H}^d\bra{\Gamma}$.
\begin{corollary}\label{coro2}
Let 
\begin{equation}\label{separablequadratic}
\widehat{\mathbf{H}}^d:=\set{\bra{y_0,y,w}\in \Real^{2d+1}\,:\,
y_j^2 \leq w_j y_0,\quad \forall j\in \set{1,\ldots,d},\quad
\sum_{j=1}^d w_j \leq y_0,\quad 0\leq y_0},
\end{equation}
then 
$\mathbf{L}^d=\proj_{\bra{y_0,y}}\bra{\widehat{\mathbf{H}}^d}$ and hence $\widehat{\mathbf{H}}^d$ is a  lifted reformulation of $\mathbf{L}^d$ with $d$ rotated two-dimensional conic quadratic constraints, one linear constraint\footnote{Note that the complete description of each rotated two-dimensional conic quadratic constraints is $y_j^2 \leq w_j y_0$ and $0\leq y_0$.} and $d$ auxiliary variables.

Furthermore, for any $\Gamma:=\set{\Gamma_j\,:\, j\in \set{1,\ldots,n}}$ for $\Gamma_j\subseteq \Real$ with $\abs{\Gamma_j}<\infty$, the set 
\begin{equation*}
   \widehat{H}^d\bra{\Gamma}:=\set{ (y_0,y,w)\in \Real^{2d+1}\,:\,\begin{alignedat}{5}2\gamma y_j- \gamma^2 y_0  &\leq w_j, &\quad& \forall j\in \set{1,\ldots,d},\,\gamma\in \Gamma_j,\\
\sum_{j=1}^d w_j &\leq y_0,\\ 0&\leq y_0. \end{alignedat} }
 \end{equation*}
is a lifted polyhedral relaxation of $\mathbf{L}^d$. 

Finally, if  $\bra{y_0,y,w}\in \widehat{H}^d\bra{\Gamma}\setminus \widehat{\mathbf{H}}^d$ and $y_0\neq 0$, then there exists $j\in \set{1,\ldots, d}$ such that $\bra{y_0,y,w}\notin \widehat{H}^d\bra{\Gamma}$ if we augment $\Gamma_j$ to $\Gamma_j\cup\set{\gamma\bra{y_0,y_j}}$
 where $\gamma\bra{y_0,y_j}:=y_j/y_0$. Similarly if $\bra{y_0,y,w}\in \widehat{H}^d\bra{\Gamma}\setminus \widehat{\mathbf{H}}^d$ and $y_0= 0$, then $\bra{y_0,y,w}\notin \widehat{H}^d\bra{\Gamma}$ if we augment $\Gamma_j$ for all $j\in \set{1,\ldots,d}$ to $\Gamma_j\cup\set{-\gamma,\gamma}$ for any $\gamma>0$.
\end{corollary}

With a slight abuse of terminology,  we refer to $\widehat{H}^d\bra{\Gamma}$ as the separable relaxation to emphasize the fact that it considers nonlinearities one variable at a time in a similar way to Proposition~\ref{separableprop}.

\subsection{Relation to Perspective Reformulations}\label{prespectivesubsection}

Perspective functions have been used to  model unions of convex sets \cite{ceriasoares99,grossmann2003generalized,StubbsMehrotra1995} for many years, with a recent emphasis on  modeling of the union of a point and a single convex set \cite{frangioni2006perspective,frangioni2009computational,frangioni2011projected,DBLP:conf/ipco/GunlukL08,DBLP:journals/mp/GunlukL10,perspsurvey,hijazi2012mixed}. We now show how these techniques can be adapted to give an alternative construction of the lifted separable reformulation $\widehat{\mathbf{C}}$.

The alternate construction of $\widehat{\mathbf{C}}$ is based on the lifted reformulation of $\mathbf{C}:=\set{y\in \Real^d\,:\, \sum_{j=1}^d f_j(y_j)\leq 1}$ given by 
\[ \mathbf{C}_{+,1} :=\set{ \bra{y,y_0,w}\in \mathbb{R}^{2d+1}\,:\, f_j(y_j)\leq w_j, \quad \forall j\in \set{1,\ldots,d},\quad \sum_{j=1}^d w_j\leq y_0, \quad y_0= 1} \]
and the set $\mathbf{C}_{+,0}:=\set{ \bra{y,y_0,w}\in \mathbb{R}^{2d+1}\,:\,y_0=y_i=w_i=0, \quad \forall i\in \set{1,\ldots,d}}$. Using known results (e.g. Section 3 of \cite{DBLP:journals/mp/GunlukL10} ) we have that, under the assumptions of Proposition~\ref{nosepprop},  
\begin{equation}\label{alternate}\hspace{-0.018in}\conv\bra{ \mathbf{C}_{+,1}\cup \mathbf{C}_{+,0}}=\set{ \bra{y,y_0,w}\in \mathbb{R}^{2d+1}\,:\, \begin{alignedat}{3}\bra{cl\tilde{f_j}}(y_0,y_j)&\leq w_j, \; \forall j\in \set{1,\ldots,d},\;\\ \sum_{j=1}^d w_j&\leq y_0,\\ \; 0\leq y_0&\leq 1\end{alignedat}}.\end{equation}
 We can also check that $\cone\bra{\set{1}\times \mathbf{C}}=\proj_{\bra{y_0,y}}\bra{\cone\bra{\conv\bra{\mathbf{C}_{+,1}\cup \mathbf{C}_{+,0}}}}$ and that $\cone\bra{\conv\bra{\mathbf{C}_{+,1}\cup \mathbf{C}_{+,0}}}$ is obtained by removing $y_0\leq 1$ from the right hand side of \eqref{alternate}, from which we  precisely obtain  $\widehat{\mathbf{C}}$\footnote{Remember that, while $\widehat{\mathbf{C}}$ does not explicitly include $y_0\geq 0$, any point in $\widehat{\mathbf{C}}$ satisfies this constraint under the assumptions of Proposition~\ref{nosepprop} (e.g. see proof of Lemma~\ref{lemma3} in Appendix~\ref{proofappendix}).}. Finally, inequalities $\bra{f\bra{\gamma}-\gamma f'\bra{\gamma}} y_0 + f'\bra{\gamma}y   \leq w_j$ from the polyhedral approximation of $\mathcal{C}$ can be obtained from the \emph{perspective cuts} from \cite{frangioni2006perspective}.

\section{Lifted LP-Based Algorithms }\label{algosec}

We can construct lifted LP-based algorithms for MICQP by combining lifted polyhedral relaxations of $\mathbf{L}^d$ with the general algorithmic framework of traditional LP-based branch-and-bound algorithms \cite{abhishek2010filmint,bonami2008algorithmic,bonami2012algorithms,FletcherLeyffer1994,QuesadaGrossmann1992}. For static relaxations such as $\widehat{L}^d_{s\bra{\varepsilon}}$,  \cite{A-Lifted-Linear-Programming}  introduces the LiftedLP algorithm,  which uses branching and the sporadic solutions of nonlinear relaxations to avoid the need to refine the polyhedral relaxation. Similarly, for dynamic polyhedral relaxations such  as $\widehat{H}^d\bra{\Gamma}$ we can follow the approach of \cite{hijazi2013outer} and adapt a traditional LP-based branch-and-bound algorithm to construct and refine the polyhedral relaxation on the space of original and auxiliary variables. However, for the polyhedral relaxations that are based on a nonlinear reformulation of $\mathbf{L}^d$, such as $\widehat{\mathbf{T}}^d$ defined in \eqref{nolineartower}, we can more easily adapt a traditional LP-based algorithm by simply giving it this reformulation. We now describe how to do this for three relaxations. We then describe two versions of the LiftedLP algorithm \cite{A-Lifted-Linear-Programming} that we test in Section~\ref{compres}.

\subsection{Algorithms from Nonlinear Reformulations}\label{nonlinearRelax}
The first two relaxations we consider correspond to the tower of variables reformulation $\widehat{\mathbf{T}}^d$  and the separable reformulation $\widehat{\mathbf{H}}^d$  from Corollary~\ref{coro2}. The third one is a combination of these two relaxations. 
 For all three relaxations we give to the LP-based branch-and-bound algorithm a reformulation of $\micp$ of the form
\begin{subequations}\label{milpr}
\begin{alignat}{3}
  \opt_{\micp\bra{\set{\widehat{\mathbf{Q}}^{d_l}}_{l=1}^k}}:=&\max & \quad cx  \\
 \notag&s.t.\\
 &      & E x  &\leq h,  &\\
 &      & A^l x+b^l &= y^l,&\quad& \forall l\in \set{1,\ldots,q},\\
 &      & a^l \cdot x + b_0^l&=y_0^l,&\quad& \forall l\in \set{1,\ldots,q},\\
 &&     \bra{y_0^l,y^l,z^l} &\in \widehat{\mathbf{Q}}^{d_l},&\quad& \forall l\in \set{1,\ldots,q}, \label{extnonlinear}\\
 & & x &\in \Real^{n},\\
& & x_j &\in \Int^{n},&\quad& \forall j\in I,
 \end{alignat}
\end{subequations}
where $\widehat{\mathbf{Q}}^{d_l}$ is a lifted reformulation of $\mathbf{L}^{d_l}$ such that $\proj_{\bra{y_0^l,y^l}}\bra{\widehat{\mathbf{Q}}^{d_l}}=\mathbf{L}^{d_l}$. Because the LP-based algorithm will consider auxiliary variables $z^l$ as formulation variables it will construct and refine a polyhedral relaxation of nonlinear constraints \eqref{extnonlinear} in the $\bra{y_0^l,y^l,z^l}$ variable space. This will effectively construct a lifted polyhedral relaxation of $\mathbf{L}^{d_l}$ using auxiliary variables $z^l$. 

The nonlinear reformulation associated to the  tower of variables relaxation from Corollary~\ref{dynamictower} is equal $ \widehat{\mathbf{T}}^d$ defined in  \eqref{nolineartower}. This reformulation uses auxiliary variables $z^l=t^l\in \Real^{R(d_l)}$ for $R(d)$ defined in Proposition~\ref{nemiprop} 
and corresponds to replacing \eqref{extnonlinear} with
\begin{subequations}\label{towerreform}
\begin{alignat}{4}
   y_0^l&=t^{l,K_l}_1,&\quad& \forall l\in \set{1,\ldots,q},\\
   t^{l,0}_i&=y_i^l, &\quad&\forall i \in \{1,\ldots, d_l\}, l\in \set{1,\ldots,q},\\
   (t^{l,k+1}_i,t^{l,k}_{2i-1},t^{l,k}_{2i})&\in \mathbf{L}^2, &\quad&\forall i\in\{1,\ldots,\lfloor r_{l,k/2} \rfloor\}, k\in\{0,\ldots,K_l-1\},l\in \set{1,\ldots,q},\label{repl2}\\
   t^{l,k}_{r_{l,k}}&=t^{l,k+1}_{\lceil r_{l,k}/2 \rceil},          &&\forall k\in\{0,\ldots,K_l-1\} \text{ s.t. $r_{l,k}$ is odd}, l\in \set{1,\ldots,q},
\end{alignat}
\end{subequations}
where $K_l$ and $\{r_{l,k}\}_{k=0}^{K_l}$ correspond to $K$ and $\{r_{k}\}_{k=0}^K$ defined in Proposition~\ref{nemiprop} for $d=d_l$.

The nonlinear reformulation associated to the separable relaxation described in Corollary~\ref{coro2} is equal to $\widehat{\mathbf{H}}^d$ defined in \eqref{separablequadratic}. This reformulation 
uses auxiliary variables $z=w\in \Real^d$ and corresponds to replacing \eqref{extnonlinear} with
\begin{subequations}\label{sepreform}
\begin{alignat}{4}
   \bra{y_j^l}^2  &\leq w_j^l y_0^l,&\quad& \forall j\in \set{1,\ldots,d_l},l\in \set{1,\ldots,q},\label{rotatedcone} \\
   \sum_{j=1}^d w_j^l &\leq y_0^l,&&\forall l\in \set{1,\ldots,q},\\
   0&\leq y_0^l,&& \forall l\in \set{1,\ldots,q}.
\end{alignat}
\end{subequations}

The final reformulation is a combination of the previous two that replaces $\mathbf{L}^2$ in \eqref{repl2} with $\widehat{\mathbf{H}}^2$. This reformulation uses auxiliary variables $z^l=\bra{t^l,v^l}\in \Real^{R(d_l)+G\bra{D^l_2}}$ for $R(d)$ and $G(D)$ defined in Proposition~\ref{nemiprop} and $D_2^l:=\set{2\,:\,  i\in\{1,\ldots,\lfloor r_{l,k}/2 \rfloor\},\ k\in \set{0,\ldots,K_l-1}}$
where $K_l$ and $\{r_{l,k}\}_{k=0}^{K_l}$ correspond to $K$ and $\{r_{k}\}_{k=0}^K$ defined in Proposition~\ref{nemiprop} for $d=d_l$. The version of constraint \eqref{extnonlinear} for this case is
\begin{subequations}\label{towersepreform}
\begin{alignat}{4}
   y_0^l&=t^{l,K}_1,&&\forall l\in \set{1,\ldots,q},\\
   t^{l,0}_i&=y_i^l, &&\forall i\in \{1,\ldots, d\},l\in \set{1,\ldots,q},\\
   \bra{t^{l,k}_{2i-1}}^2 &\leq v^{l,k}_{i,1} t^{l,k+1}_i,&\quad&\forall i\in\{1,\ldots,\lfloor r_{l,k}/2 \rfloor\}, k\in\{0,\ldots,K_l-1\},l\in \set{1,\ldots,q},\\
                  \bra{t^{l,k}_{2i}}^2 &\leq v^{l,k}_{i,2} t^{l, k+1}_i,&\quad&\forall i\in\{1,\ldots,\lfloor r_{l,k}/2 \rfloor\},k\in\{0,\ldots,K_l-1\},l\in \set{1,\ldots,q},\\
   v^{l,k}_{i,1} +v^{l,k}_{i,1}  & \leq t^{l,k+1}_i,&\quad&\forall i\in\{1,\ldots,\lfloor r_{l,k}/2 \rfloor\},k\in\{0,\ldots,K_l-1\},l\in \set{1,\ldots,q},\\
   t^{l,k}_{r_{l,k}}  &=t^{l,k+1}_{\lceil r_{l,k}/2 \rceil},          &&\forall k\in\{0,\ldots,K_l-1\} \text{ s.t. $r_{l,k}$ is odd}, l\in \set{1,\ldots,q}. 
\end{alignat}
\end{subequations}

In Section~\ref{compres} we compare the effectiveness of these  nonlinear reformulations with two versions of the LiftedLP algorithm of \cite{A-Lifted-Linear-Programming}, which we now describe in detail.

\subsection{Branch-based LiftedLP  Algorithm and Cut-based Adaptation}\label{liftedLPimplementation}

The original LiftedLP algorithm of \cite{A-Lifted-Linear-Programming}  uses a version of $\widehat{L}^d_{s\bra{\varepsilon}}$ introduced in \cite{glineur2000}. This version corresponds to the static lifted polyhedral relaxation described by the following corollary.
\begin{corollary}\label{glineurcoro}Let   $K=\lceil \log_2(d)\rceil$ and $\{r_k\}_{k=0}^K$ be defined by the recursion $r_0=d$, ${r_{k+1}=\lceil r_k /2 \rceil}$ for $k\in \{0,\ldots,K-1\}$. Furthermore, for any $\varepsilon\in(0,1/2)$ let
\begin{equation*}
  s_k(\varepsilon)=\left\lceil \frac{k+1}{2}\right\rceil -\left\lceil \log_4 \left( \frac{16}{9} \pi^{-2} \log(1+\varepsilon)\right)\right\rceil
\end{equation*}
for each $k\in \set{0,\ldots,K-1}$. Finally, let   
$L^d_\varepsilon:=\widehat{T}^d\bra{\mathcal{N}_\varepsilon,D_{\varepsilon}}$ where \[\mathcal{N}_\varepsilon:=\set{N^2_{s_k\bra{\varepsilon}}\,:\, i\in\{1,\ldots,\lfloor r_k/2 \rfloor\},\,k\in \set{0,\ldots,K-1}}\]
and  $D_{\varepsilon}:=\set{2s_k\bra{\varepsilon}\,:\, i\in\{1,\ldots,\lfloor r_k/2 \rfloor\},\,k\in \set{0,\ldots,K-1}}$.
Then, $L^d_{\varepsilon}$ is a lifted polyhedral relaxation of $\mathbf{L}^d$ with approximation quality $\varepsilon$ and $O(d\log(1/\varepsilon))$ variables and constraints.
\end{corollary}

The LiftedLP algorithm replaces every occurrence $\mathbf{L}^d$ with $L^d_{\varepsilon}$ for an appropriately chosen $\varepsilon$. To avoid the need to refine these approximations the algorithm sporadically solves some nonlinear relaxations and follows a specialized branching convention. This  branch-based approach to the LiftedLP algorithm was motivated by the ineffectiveness of traditional polyhedral relaxations of $\mathbf{L}^d$ in certain classes of problems. However, the advent of dynamic lifted polyhedral relaxations suggests the possible effectiveness of a cut-based version of the LiftedLP algorithm. We now concurrently describe  the original branch-based  and the  cut-based variant of the LiftedLP algorithm. The cut-based variant combines static relaxation $L^d_{\varepsilon}$ with a dynamic relaxation of $\mathbf{L}^d$. While any of the relaxations described in the previous section can be used, preliminary computational test showed that $\widehat{H}^d\bra{\Gamma}$ provides comparable or better performance than the other relaxations. For this reason, and because the extension to other dynamic relaxations is straightforward,  we here only describe the variant for $\widehat{H}^d\bra{\Gamma}$.  

Both versions of the LiftedLP algorithm are based on the LP relaxation of $\micp$ defined in \eqref{milp} given by
\begin{subequations}\label{lp}
\begin{alignat}{3}
  \opt_{\lp\bra{\mathbf{l},\mathbf{u},\set{\Gamma^l}_{l=1}^k}}:=&\max & \quad  cx  \\
 \notag&s.t.\\
 &      &  E y &\leq h,  &\\
 &      & A^l x+b^l &= y^l,&\quad& l\in \set{1,\ldots,q},\\
 &      &a^l \cdot x + b_0^l &=y_0^l,&\quad& l\in \set{1,\ldots,q},\\
 &&     \bra{y^l,y_0^l,t^l,v^l} &\in L^{d_l}_{\varepsilon},&\quad& l\in \set{1,\ldots,q}, \label{coneslP}\\
  &&     \bra{y^l,y_0^l,w^l} &\in \widehat{H}^{d_l}\bra{\Gamma^l},&\quad& l\in \set{1,\ldots,q}, \label{coneslPD}\\
 &&            \mathbf{l}_j \leq x_j &\leq \mathbf{u}_j,&\quad& j\in \set{1,\ldots,n}.
 \end{alignat}
\end{subequations}
where $\Gamma^l$ corresponds to the set of inequalities currently used by the algorithms. If $\Gamma^l=\emptyset$ for all $l\in \set{1,\ldots,q}$ constraints \eqref{coneslPD} do not restrict $\bra{y^l,y_0^l}$ in any way, so we assume that in such case \eqref{coneslPD} is simply omitted. Vectors  $\mathbf{l},\mathbf{u}\in \bra{\Real\cup\set{-\infty,\infty}}^n$ correspond to variable lower and upper bounds that are used to define a node in the branch-and-bound tree. These bounds are initially infinite ($\mathbf{l}_j=-\infty$ an $\mathbf{u}_j=\infty$ for all $j\in \set{1,\ldots,n}$) and are only modified for integer constrained variables $x_j$ for $j\in I$. The infinite bounds for the continuous variables are included to simplify notation. We denote the  relaxation defined in \eqref{lp} $\lp\bra{\mathbf{l},\mathbf{u},\set{\Gamma^l}_{l=1}^k}$. The algorithms also use the following nonlinear relaxation to sporadically compute node-bounds or as a heuristic to find feasible solutions. 
 \begin{subequations}\label{nlp}
\begin{alignat}{3}
  \opt_{\nlp\bra{\mathbf{l},\mathbf{u}}}:=&\max & \quad  cx  \\
 \notag&s.t.\\
 &      &  E y &\leq h,  &\\
 &      & A^l x+b^l &= y^l,&\quad &l\in \set{1,\ldots,q},\\
 &      &a^l \cdot x + b_0^l &=y_0^l,&\quad &l\in \set{1,\ldots,q},\\
 &&     \bra{y^l,y_0^l} &\in \mathbf{L}^{d_l},&\quad &l\in \set{1,\ldots,q},\\
  &&            \mathbf{l}_j \leq x_j &\leq \mathbf{u}_j,&\quad& j\in \set{1,\ldots,n}. \end{alignat}
\end{subequations}
We denote this relaxation $\nlp\bra{\mathbf{l},\mathbf{u}}$. 

The final ingredient in the algorithms is a list of branch-and-bound nodes to be processed, which we denote $\mathcal{M}$. The nodes in this list are characterized by variable upper and lower bounds $\bra{\mathbf{l},\mathbf{u}}$ and an estimated upper bound of the nonlinear relaxation $\nlp\bra{\mathbf{l},\mathbf{u}}$. The objective of the algorithm is to simulate a nonlinear branch-and-bound based on $\nlp\bra{\mathbf{l},\mathbf{u}}$, but by solving $\lp\bra{\mathbf{l},\mathbf{u},\set{\Gamma^l}_{l=1}^k}$ in each node and only sporadically solving $\nlp\bra{\mathbf{l},\mathbf{u}}$. A generic version of this procedure is described in Algorithm~\ref{LPNLPBB}.
 \begin{algorithm}[htpb]
   \DontPrintSemicolon
  Set global lower bound $\LB:=-\infty$.\;
  Set $\mathbf{l}^0_j:=-\infty$, $\mathbf{u}^0_j:=+\infty$ for all $j\in\{1,\ldots,n\}$.\;
  Set $\NB^0=+\infty$.\;
  Set node list $\mathcal{M}:=\{(\mathbf{l}^0,\mathbf{u}^0, \NB^0)\}$.\;
  Set initial approximation $\Gamma^l:=\Gamma_0^l$ for all $l\in \set{1,\ldots,q}$.\;
   \While{$\mathcal{M}\neq \emptyset$}{
     Select and remove a node $(\mathbf{l}^k,\mathbf{u}^k,\NB^k)\in\mathcal{M}$.\;  
     Solve $\lp\bra{\mathbf{l}^k,\mathbf{u}^k,\set{\Gamma^l}_{l=1}^q}$.\;
    \If{$\lp\bra{\mathbf{l}^k,\mathbf{u}^k,\set{\Gamma^l}_{l=1}^q}$ is feasible {\bf and} $\opt_{\lp\bra{\mathbf{l}^k,\mathbf{u}^k,\set{\Gamma^l}_{l=1}^q}} > \LB$ }{
      Let $\bra{\bar{x},\set{\bra{\bar{y}^l,\bar{y}_0^l,\bar{z}^l}}_{l=1}^q}$ be the optimal solution to $\lp\bra{\mathbf{l}^k,\mathbf{u}^k,\set{\Gamma^l}_{l=1}^q}$.\;
      \eIf{$\bar{x}_j\in\Int$ for all $j\in I$}{
        \eIf{$\bra{\bar{y}^l,\bar{y}_0^l}\in \mathbf{L}^{d_l}$ for all $l\in \set{1,\ldots,q}$}{
          Set $\LB:=\opt_{\lp\bra{\mathbf{l}^k,\mathbf{u}^k,\set{\Gamma^l}_{l=1}^q}}$.\;
          }(\tcc*[f]{Run Heuristic and Refine})
        {
        Let $\mathbf{l}_j=\mathbf{l}^k_j$, $\mathbf{u}_j=\mathbf{u}^k_j$  for all $j\in \{1,\ldots,n\}\setminus I$.\; 
        Let $\mathbf{l}_j=\mathbf{u}_j=\bar{x}_j$  for all $j\in  I$.\;
        Solve $\nlp(\mathbf{l},\mathbf{u})$.\;
        \If{$\nlp(\mathbf{l},\mathbf{u})$ is feasible {\bf and} $\opt_{\nlp(\mathbf{l},\mathbf{u})} > \LB$ }{
            $\LB:=\opt_{\nlp(\mathbf{l},\mathbf{u})}$.\;
        }
        Call $\proc\bra{\mathbf{l}^k,\mathbf{u}^k, \set{\bra{\bar{y}^l,\bar{y}_0^l,\bar{z}^l}}_{l=1}^q,\LB,\opt_{\lp(\mathbf{l}^k,\mathbf{u}^k)}, \mathcal{M},\set{\Gamma^l}_{l=1}^q}$.\;
        }
        }(\tcc*[f]{Branch on $\bar{x}$}){
        Pick $j_0$ in $\{j\in I\,:\, \bar{x}_j\notin\Int\}$.\;
        Let $\mathbf{l}_j=\mathbf{l}^k_j$, $\mathbf{u}_j=\mathbf{u}^k_j$  for all $j\in \{1,\ldots,n\}\setminus\{j_0\}$.\;  
        Let $\mathbf{u}_{j_0}=\lfloor \bar{x}_{j_0} \rfloor$, $\mathbf{l}_{j_0}=\lfloor \bar{x}_{j_0} \rfloor +1$.\;
        $\mathcal{M}:=\mathcal{M}\cup\set{ \bra{\mathbf{l}^k,\mathbf{u},\opt_{\lp(\mathbf{l}^k,\mathbf{u}^k)}} ,\bra{\mathbf{l},\mathbf{u}^k,\opt_{\lp(\mathbf{l}^k,\mathbf{u}^k)}}}$
        }
     }
     Remove every node $(\mathbf{l}^k,\mathbf{u}^k,\NB^k)\in\mathcal{M}$ such that $\NB^k\leq \LB$.\;  
    } 
    \caption{A Generic lifted LP-based Branch-and-Bound Algorithm.}\label{LPNLPBB}
\end{algorithm} 
The basis of Algorithm~\ref{LPNLPBB} is a branch-and-bound procedure that solves $\lp\bra{\mathbf{l},\mathbf{u},\set{\Gamma^l}_{l=1}^k}$ in each branch-and-bound node. If the optimal value of this relaxation is worse than that of the incumbent feasible solution the node is fathomed by bound in line 9. If the optimal value is better than the incumbent and the solution to the LP relaxation does not satisfy the integrality constraints of $\micp$, the algorithm branches on an integer constrained variable with a fractional value in the traditional way. This is done in lines 23--28 of the algorithm. If the optimal value is better than the incumbent and the solution of the LP relaxation satisfies the integrality constraints of $\micp$, the algorithm first checks if the nonlinear constraints of $\micp$ happen to be satisfied in line 12. If the nonlinear constraints are satisfied the incumbent solution is updated in line 13 (as stated the algorithm only updates the value of the solution; the modification to update the solution itself is straightforward). If the nonlinear constraints are not satisfied, the algorithm first attempts to find a feasible solution to $\micp$ that has the same values in the integer variables as the solution to $ \lp\bra{\mathbf{l},\mathbf{u},\set{\Gamma^l}_{l=1}^k}$ for the current node. This is done by solving $\nlp\bra{\mathbf{l},\mathbf{u}}$ for appropriately chosen bounds in lines 15--17. If this heuristic is successful and yields a better solution the incumbent is updated in lines 18--20. Finally, a generic refinement procedure for the current branch-and-bound node is called in line 21.

The original LiftedLP algorithm from \cite{A-Lifted-Linear-Programming} is obtained when we let $\Gamma_0^l=\emptyset$ for all $l\in \set{1,\ldots,q}$ and we use the branch-based refinement procedure described in Algorithm~\ref{branchbased}. 
 \begin{algorithm}[htpb]
   \DontPrintSemicolon
   \KwIn{Lower and upper variable bounds $\bra{\mathbf{l}^k,\mathbf{u}^k}$, solution $\set{\bra{\bar{y}^l,\bar{y}_0^l,\bar{w}^l}}_{l=1}^q$, lower bound $\LB$, node bound $y_{\lp(\mathbf{l}^k,\mathbf{u}^k)}$, node list $\mathcal{M}$ and cut list $\set{\Gamma^l}_{l=1}^q$.}
        Solve $\nlp(\mathbf{l}^k,\mathbf{u}^k)$.\;
        \If{$\nlp(\mathbf{l}^k,\mathbf{u}^k)$ is feasible {\bf and} $\opt_{\nlp(\mathbf{l}^k,\mathbf{u}^k)} > \LB$ }{
          Let $(\tilde{x}^k,\tilde{y}^k)$ be the optimal solution to $\nlp(\mathbf{l}^k,\mathbf{u}^k)$.\;
          \eIf(\tcc*[f]{Fathom by Integrality}){$\tilde{x}^k\in\Int^{n}$}{
            $\LB:=\opt_{\nlp(\mathbf{l}^k,\mathbf{u}^k)}$.\;
            }(\tcc*[f]{Branch on $\tilde{x}^k$}){
            Pick $j_0$ in $\{j\in I\,:\, \tilde{x}^k_j\notin\Int\}$.\;
            Let $\mathbf{l}_j=\mathbf{l}^k_j$, $\mathbf{u}_j=\mathbf{u}^k_j$  for all $j\in \{1,\ldots,n\}\setminus\{j_0\}$.\;  
            Let $\mathbf{u}_{j_0}=\lfloor \tilde{x}^k_{j_0} \rfloor$, $\mathbf{l}_{j_0}=\lfloor \tilde{x}^k_{j_0} \rfloor +1$.\;
            $\mathcal{M}:=\mathcal{M}\cup\set{\bra{\mathbf{l}^k,\mathbf{u},\opt_{\nlp(\mathbf{l}^k,\mathbf{u}^k)}},\bra{\mathbf{l},\mathbf{u}^k,\opt_{\nlp(\mathbf{l}^k,\mathbf{u}^k)}}}$
            }
          }
    \caption{Branch-based refinement,\white{-------------------------------------------------------------------------------} $\bproc\bra{\mathbf{l}^k,\mathbf{u}^k, \set{\bra{\bar{y}^l,\bar{y}_0^l,\bar{w}^l}}_{l=1}^q,\LB,\opt_{\lp(\mathbf{l}^k,\mathbf{u}^k)}, \mathcal{M},\set{\Gamma^l}_{l=1}^q}$}\label{branchbased}
\end{algorithm} 
The procedure begins by solving the nonlinear relaxation that would have been solved by an NLP-based branch-and-bound algorithm in the current node. The procedure then processes this node exactly as an NLP-based branch-and-bound algorithm would. In line 2 it first checks if the node can be fathomed by bound or infeasibility. If this fails, the procedure attempts to fathom by integrality in lines 4--6. Finally, if all the previous steps fail, the procedure branches on an integer constrained variable with a fractional value in lines 6--11. 

Finally, a cut-based version of the algorithm that truly uses $H^{d_l}\bra{\Gamma^l}$ is obtained when we use the cut-based refinement procedure described in Algorithm~\ref{cutrefinement}. 
\begin{algorithm}[htpb]
 \DontPrintSemicolon
  \KwIn{Lower and upper variable bounds $\bra{\mathbf{l}^k,\mathbf{u}^k}$, solution $\set{\bra{\bar{y}^l,\bar{y}_0^l,\bar{w}^l}}_{l=1}^q$, lower bound $\LB$, node bound $y_{\lp(\mathbf{l}^k,\mathbf{u}^k)}$, node list $\mathcal{M}$ and cut list $\set{\Gamma^l}_{l=1}^q$.}
  \For{$l= 1$ \KwTo $q$}{
  \If{$\bra{\bar{y}^l,\bar{y}_0^l}\notin \mathbf{L}^{d_l}$}{
    \For{$j= 1$ \KwTo $d_l$}{
  \If{$\bra{\bar{y}_j^l}^2  > \bar{w}_j^l \bar{y}_0^l$}{
    Set $\Gamma_j^l:=\Gamma_j^l\cup\set{\bar{y}_j^l/\bar{y}_0^l}$\;
  }
  }
  }
  
  }
  $\mathcal{M}:=\mathcal{M}\cup\set{\bra{\mathbf{l}^k,\mathbf{u}^k,\opt_{\lp(\mathbf{l}^k,\mathbf{u}^k)}}}$
    \caption{Cut-based refinement, $\cproc\bra{\mathbf{l}^k,\mathbf{u}^k, \set{\bra{\bar{y}^l,\bar{y}_0^l,\bar{w}^l}}_{l=1}^q,\LB,\opt_{\lp(\mathbf{l}^k,\mathbf{u}^k)}, \mathcal{M},\set{\Gamma^l}_{l=1}^q}$}\label{cutrefinement}
\end{algorithm} 
This procedure first checks for a violation of the nonlinear constraints by the current node's solution in lines 1--2. If a violation is found, then the separation procedure for $H^{d_l}\bra{\Gamma^l}$ is called in line 3. This separation procedure updates $\Gamma^l$ with one or more inequality violated by the node solution.

\section{Computational Experiments}\label{compres}

In this section we compare the  performance of the different algorithms on the portfolio optimization instances considered in \cite{A-Lifted-Linear-Programming}. We begin by describing the portfolio optimization instances and  how the different algorithms are implemented. We then present the results of the computational experiments. 

\subsection{Instances}\label{instances}

The instances from \cite{A-Lifted-Linear-Programming} correspond to three classes of portfolio optimization instances with limited diversification or cardinality constraints  \cite{bertsimas2009algorithm,Bienstock94,ChangMeade2000,MaringerKellerer2003} which are formulated as MICQPs. 
All three problems  construct a portfolio out of $n$ assets with an expected return  $ \bar{a}\in \Real^n$. The objective is to maximize the return of the portfolio subject to various risk constraints and limitation on the number of assets considered in the portfolio. To simplify the description of the three versions we deviate from the conventions of \eqref{micp} and  use different names for continuous and integer constrained variables. With this the first class of problems we consider are of the form 
 \begin{subequations}\label{markowitzprob}
 \begin{alignat}{3}
   \label{markbegin}  &\max & \quad{\bar{a}}x\\
 \notag&s.t.\\
  \label{variance}&      & \norm{Q^{1/2} x}_2 &\leq \sigma,  &
\\
\label{convex}&      & \sum_{j=1}^n x_j &= 1,  &
\\
\label{knpmark1}& & x_j & \leq z_j, &\forall j\in\{1,\ldots,n\}, \\
&      & \sum_{j=1}^n z_j &\leq K,  &
\\
\label{markenddd}&      &z &\in \{0,1\}^{n},\\
\label{markend}&      &x &\in \Real_+^{n},
 \end{alignat}
  \end{subequations}
 where  $x$ indicates the fraction of the portfolio invested in each asset,  $Q^{1/2}$ is the positive semidefinite square root of the covariance matrix of the returns of the stocks, $\sigma$ is the maximum allowed risk and $K<n$ is the maximum number of assets that can be included  in the portfolio. We refer to this class of instances as the \emph{classical} instances. 

The second class of problems  is obtained by replacing constraint \eqref{variance} with a shortfall constraint considered in   \cite{LoboFazelBoyd2005,LoboVandenbergheBoyd1998}. This constraint can be formulated as 
  \begin{equation*}
    \Phi^{-1}(\eta_i)\norm{Q^{1/2} y }_2 \leq {\bar{a}} y  - W_i^{low}, \quad  \quad \forall i\in\{1,2\}
  \end{equation*}
  where $W_i^{low}$ and $\eta_i$ are given parameters and  $\Phi(\cdot)$ is the cumulative distribution function of the Normal distribution with zero mean and  unit standard deviation. We refer to this class of instances as the \emph{shortfall} instances.

The third and final class of problems correspond to a robust version of \eqref{markowitzprob} introduced in \cite{CeriaStubbs2006}. This model introduces an additional continuous variable $t$, replaces the objective function  \eqref{markbegin} with  $\max  t$ and adds the constraint $ {\bar{a}} x- \alpha \norm{R^{1/2} y }_2  \geq  t$ where $R^{1/2}$ is the positive semidefinite square root of a given matrix $R$ and $\alpha $ is a given scalar parameter. We refer to this class of instances as the \emph{robust} instances.

We note that only the classical instances can be handled  by the lifted polyhedral relaxation considered in  \cite{hijazi2013outer}.

\subsection{Implementation and Computational Settings}

All algorithms and models were implemented using the JuMP modeling language \cite{jump,lubin2013computing} and solved with CPLEX v12.6 \cite{cplex} and Gurobi v5.6.3 \cite{gurobi}. The complete code for this implementation is available at \url{https://github.com/juan-pablo-vielma/extended-MIQCP}. 

Our base benchmark algorithms are CPLEX and Gurobi's standard algorithms for solving $\micp$. Both solvers implement an NLP-based branch-and-bound algorithm and a standard LP-based branch-and-bound algorithm. Each of these implementations include advanced features such as cutting planes, heuristics, preprocessing and elaborate branching and node selection strategies. We refer to the NLP-based algorithms as CPLEXCP and GurobiCP, and to the LP-based   algorithms as CPLEXLP and GurobiLP. All four algorithms can be used by simply giving the model to the appropriate solver and setting a specialized parameter value. 

We also implemented a branch-based and a cut-based version of Algorithm~\ref{LPNLPBB}. The branch-based version corresponds to the LiftedLP algorithm from  \cite{A-Lifted-Linear-Programming} and its implementation requires access to a branch callback, which is not provided by Gurobi. For this reason we only implemented a CPLEX version similar to the original implementation from \cite{A-Lifted-Linear-Programming}. This version was developed using the branch, heuristic and incumbent callbacks for CPLEX provided by the CPLEX.jl library \cite{cplexjl}. We refer to this algorithm as LiftedLP. Implementing the  cut-based versions of Algorithm~\ref{LPNLPBB} only requires access to a \emph{lazy constraint} callback and to a heuristic callback. These callbacks can be accessed for CPLEX and Gurobi through the solver independent callback interface  provided by JuMP, which allowed us to implement a version of this algorithm that is not tied to either solver. We implemented the cut-based LiftedLP algorithm for all three lifted polyhedral relaxations considered in Section~\ref{nonlinearRelax}. However, as noted in Section~\ref{liftedLPimplementation}, preliminary experiments showed that the version based on $\widehat{H}^d\bra{\Gamma}$ provides comparable or better performance than the other two versions. For this reason we here only present results for that version, which corresponds to Algorithm~\ref{LPNLPBB} using the cut-based refinement described in Algorithm~\ref{cutrefinement}. We refer to the implementation of this algorithm using CPLEX and Gurobi as base solvers as CPLEXSepLazy and GurobiSepLazy respectively. 

Finally, instead of implementing LP-based algorithms that only use a dynamic lifted polyhedral relaxation (i.e. that does not use $L_\varepsilon^d$ or $L_{s(\varepsilon)}^d$), we simply solve the three lifted reformulations described in Section~\ref{nonlinearRelax} with CPLEX and Gurobi's LP-based algorithms. We refer to the implementations based on reformulation \eqref{towerreform} as CPLEXTowerLP and GurobiTowerLP, to the implementations based on reformulation \eqref{sepreform} as CPLEXSepLP and GurobiSepLP and to the implementations based on reformulation \eqref{towersepreform} as CPLEXTowerSepLP and GurobiTowerSepLP.

\subsection{Results}

All computational results in this section are from tests on a Intel i7-3770 3.40GHz Linux workstation with 32GB of RAM. All algorithms were limited to a single thread by appropriately setting CPLEX and Gurobi parameters and to a total run time of 3600 seconds. We consider the same portfolio optimization instances from \cite{A-Lifted-Linear-Programming}, which correspond to the three classes of problems described in Section~\ref{instances} for $K=10$ and $n\in \set{10,20,30,40,50,60,100,200,300}$. For each problem class and choice of $n$ we test $100$ randomly generated instances. We refer the reader to \cite{A-Lifted-Linear-Programming} for more details on how the instances were generated. All instances and results are available at \url{https://github.com/juan-pablo-vielma/extended-MIQCP}. Results are presented in two types of charts. The first type are box-and-whisker charts generated by the BoxWhiskerChart function in Mathematica v10 \cite{mathematica}. These charts consider solve times in seconds in a logarithmic scale and show the median solve times, 25\% and 75\% quantiles of the solve times and minimum and maximum solve times excluding outliers, which are shown as dots. We note that to ensure the graphs are easily legible in black and white print we use the same colors for each chart. Hence, a given algorithm may be assigned different colors in different charts. The second type are  performance profiles introduced by Dolan and Mor\'{e} \cite{DolanMore2002} with solve time as a performance metric.  Finally, tables with summary statistics for all methods and instances are included in the Appendix~\ref{apendixA}.

\subsection{Comparison with Traditional Algorithms and Initial Calibration}

In this section we present some initial results that evaluate the difficulty of the considered instances, compare the lifted algorithms to standard algorithms and compare the dynamic lifted relaxations. Because the Classical and Shortfall instances are extremely difficult for many algorithms tested  in this section we only consider $n\in \set{10,20,30,40,50,60}$ for these instances. Similarly we exclude $n=300$ for the Robust instances.

The first set of results are presented through  box-and-whisker charts in Figure~\ref{LPVSNLP}. These charts compare the performance of CPLEX and Gurobi's NLP and LP-based algorithms (CPLEXCP, GurobiCP, CPLEXLP and GurobiLP). 
\begin{figure}[htb]
\centering 
\subfigure[Classical.]{\includegraphics[scale=.85]{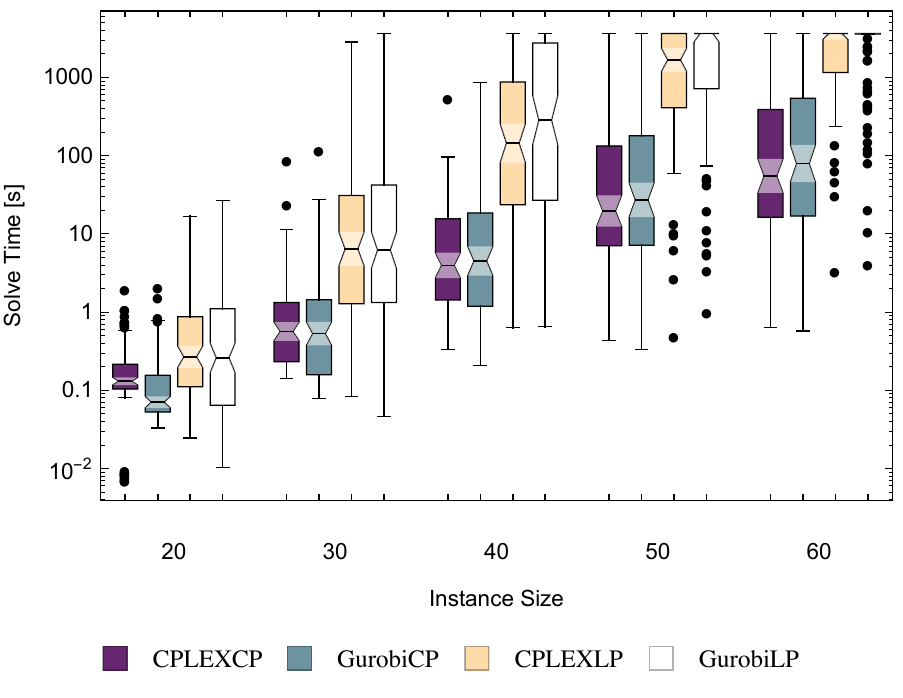}}
\subfigure[Shortfall.]{\includegraphics[scale=.85]{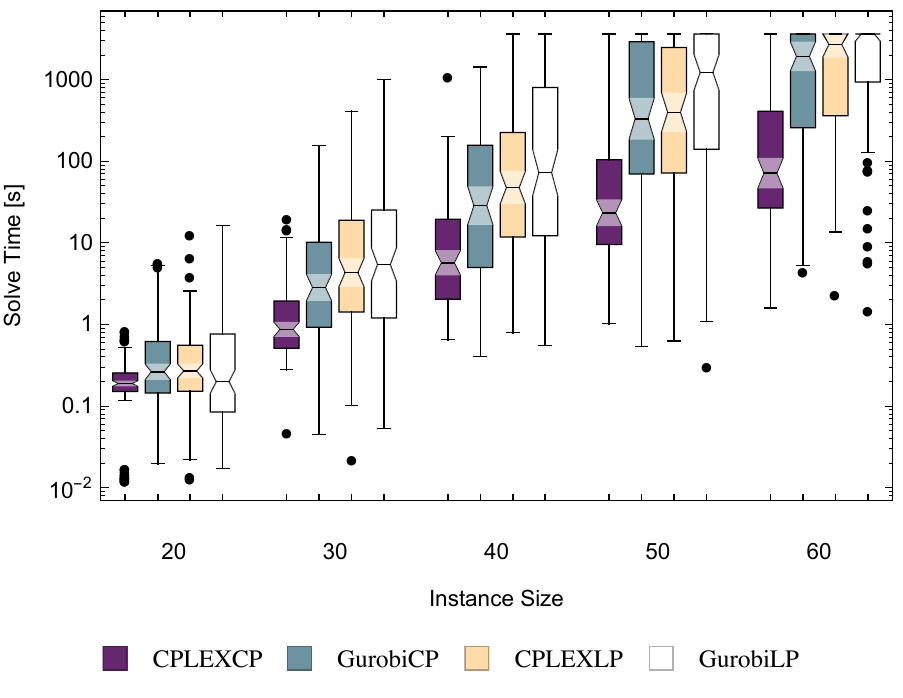}}
\subfigure[Robust.]{\includegraphics[scale=.85]{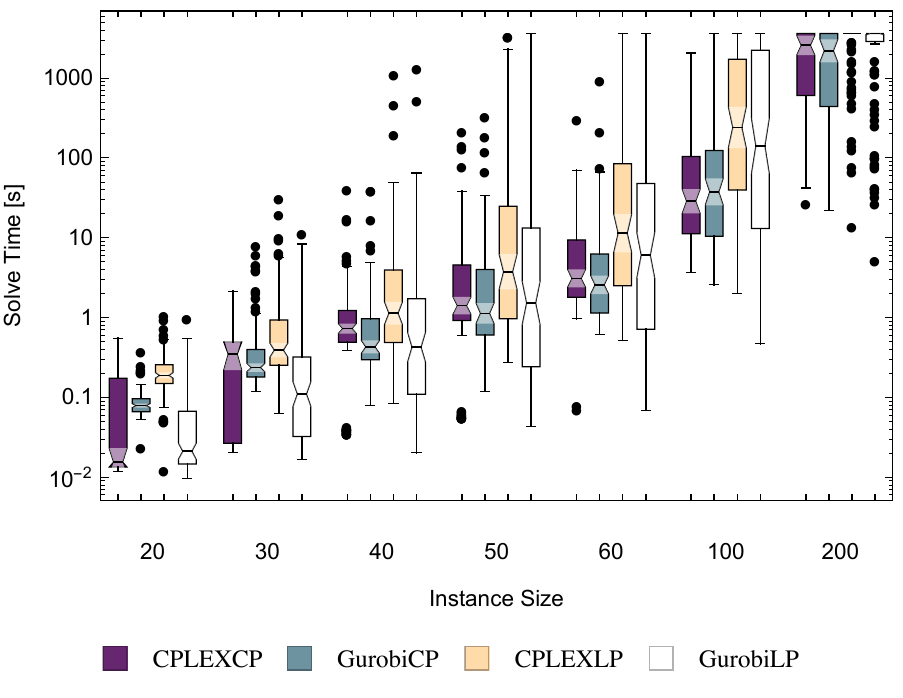}}
\caption{Solution times for standard  LP-based and NLP-based algorithms [s].}\label{LPVSNLP}
\end{figure}
The results in \cite{A-Lifted-Linear-Programming} showed that NLP-based algorithms had a significant advantage over traditional LP-based algorithms for the portfolio optimization instances. Figure~\ref{LPVSNLP} shows that, for sufficiently large values of $n$, this advantage still holds for current versions of CPLEX and Gurobi.

The second set of results are also presented through box-and-whisker charts in Figure~\ref{LIFTEDLPVSNLP}. These charts compare the performance of the branch-based LiftedLP algorithm (LiftedLP) and the two NLP-based algorithms (CPLEXCP and GurobiCP).
\begin{figure}[htb]
\centering 
\subfigure[Classical.]{\includegraphics[scale=.85]{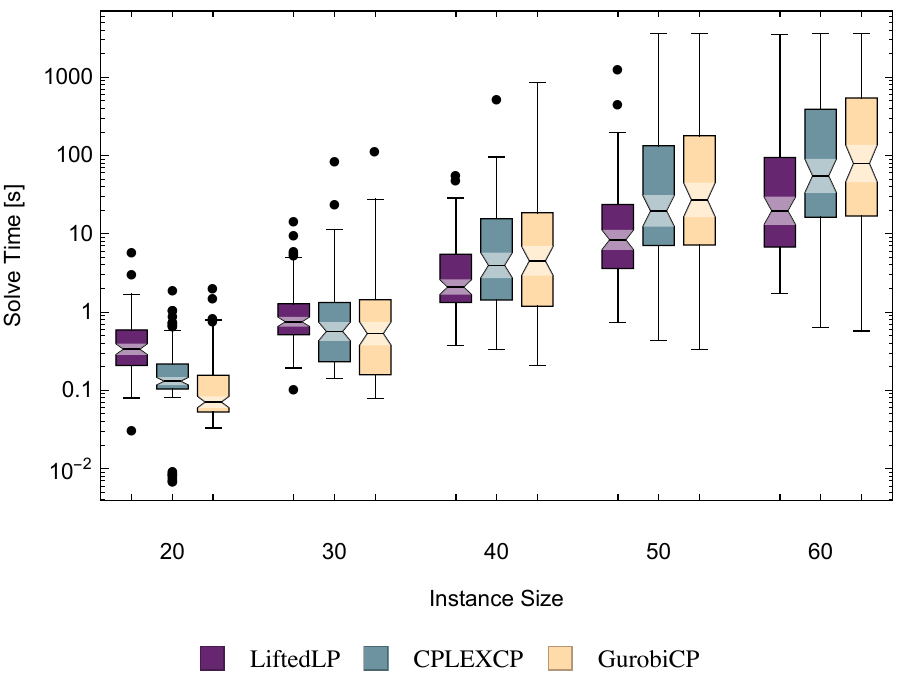}}
\subfigure[Shortfall.]{\includegraphics[scale=.85]{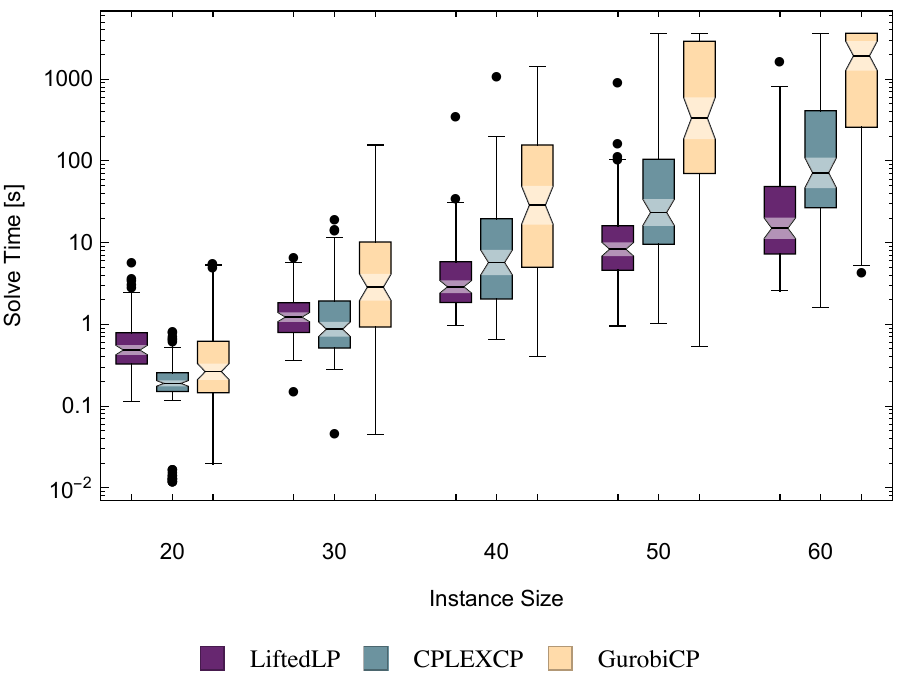}}
\subfigure[Robust.]{\includegraphics[scale=.85]{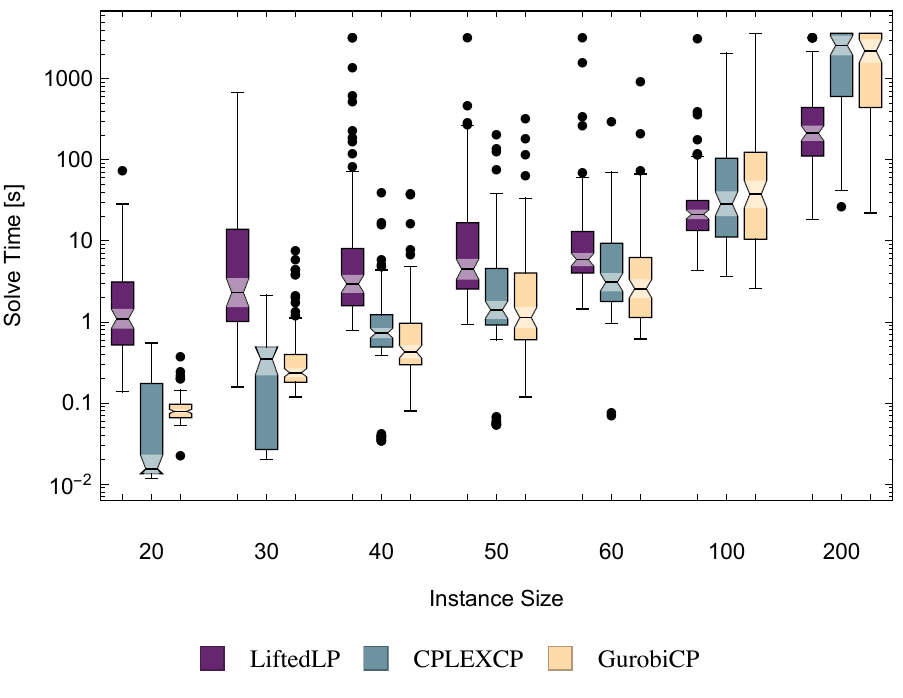}}
\caption{Solution times for branch-based LiftedLP and NLP-based algorithms [s].}\label{LIFTEDLPVSNLP}
\end{figure}
Figure~\ref{LIFTEDLPVSNLP} confirms that,  for sufficiently large values of $n$,  LiftedLP still provides an advantage over CPLEXCP and GurobiCP for the considered instances.

The final set of results in this section compares the performance of the three dynamic lifted polyhedral relaxations in their nonlinear reformulation versions \redd{described in Section~\ref{nonlinearRelax}} (CPLEXSepLP, GurobiSepLP, CPLEXTowerLP, GurobiTowerLP, CPLEXTowerSepLP, GurobiTowerSepLP). This time the results are presented through a performance profile in Figure~\ref{perprofformulation}.
\begin{figure}[htb]
\centering 
\includegraphics[scale=.85]{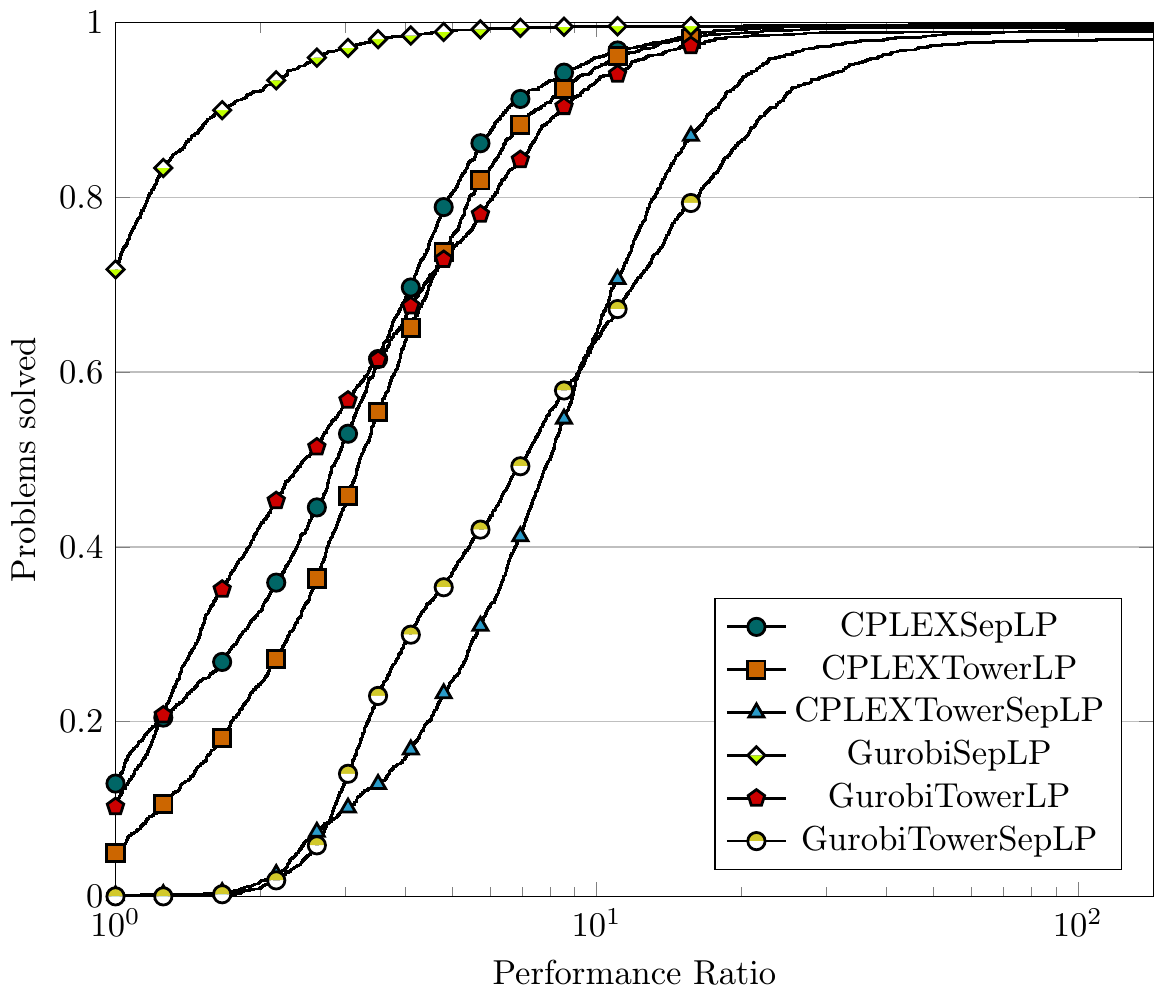}
\caption{Performance profiles of solution times for dynamic lifted polyhedral relaxations solved by standard LP-based algorithms.}\label{perprofformulation}
\end{figure}
Figure~\ref{perprofformulation} shows that for a fixed solver (CPLEX or Gurobi), the separable relaxation outperforms the other relaxation (i.e. CPLEXSepLP outperforms CPLEXTowerLP and CPLEXTowerSepLP, and GurobiSepLP outperforms GurobiTowerLP and GurobiTowerSepLP). Furthermore, CPLEXSepLP and GurobiSepLP have comparable or better performance than all other combinations of solver and relaxation. These results align with our preliminary experiments which showed that the separable relaxation performed best among the cut-based LiftedLP algorithms considered in Section~\ref{liftedLPimplementation}. Because the separable relaxation is additionally the simplest, from now on we concentrate on this relaxation among the three dynamic ones. Box-and-whisker charts with more details concerning this experiment are presented in  Figure~\ref{dynanmicliftedform} in Appendix~\ref{appendixgraphs}. \red{For instance, Figure~\ref{dynanmicliftedform} shows that the pure tower relaxation (CPLEXTowerLP and GurobiTowerLP) can provide an advantage over the separable relaxations (CPLEXSepLP and GurobiSepLP) for the smallest instances, but the combined tower-separable relaxation (CPLEXTowerSepLP and GurobiTowerSepLP) is consistently outperformed by the other two. Additional details are also included in summary tables in Appendices~\ref{apendixA} and \ref{qualityappendix}. The tables in Appendix~\ref{apendixA} present summary statistics for solve times by class of instance and size, which can be used for even more detailed comparisons. For example, the tables confirm the potential advantage of the tower relaxation for the smallest instances by showing that GurobiTowerLP is the fastest option in 54 out of the 100 shortfall instances for $n=20$. The tables in Appendix~\ref{qualityappendix} present summary statistics for the feasibility quality of the solutions obtained. These statistics show that using the separable reformulation resulted in a significant reduction on the feasibility quality of the solutions obtained, particularly when using Gurobi. However, this reduction in quality is not surprising as errors in the $3$-dimensional rotated conic constraints \eqref{rotatedcone} can easily add up to a larger error in the original conic constraint. Fortunately, this can be easily resolved by increasing the precision for constraints \eqref{rotatedcone} or by simply correcting the final or intermediate incumbent solutions using the original conic constraint \redd{(i.e. by solving the original conic quadratic relaxation with the integer variables fixed appropriately).} 
 }

\subsection{Comparison between Lifted LP-based Algorithms}

In this section we compare the performance of the lifted algorithms. Classical and Shortfall instances are still extremely difficult for many of these algorithms so we exclude results for $n\in\set{200,300}$ for these instances. We begin by comparing the branch-based (LiftedLP) and cut-based LiftedLP algorithms (CPLEXSepLazy and GurobiSepLazy) in Figure~\ref{branchvslazy}.
\begin{figure}[htb]
\centering 
\subfigure[Classical.]{\includegraphics[scale=.85]{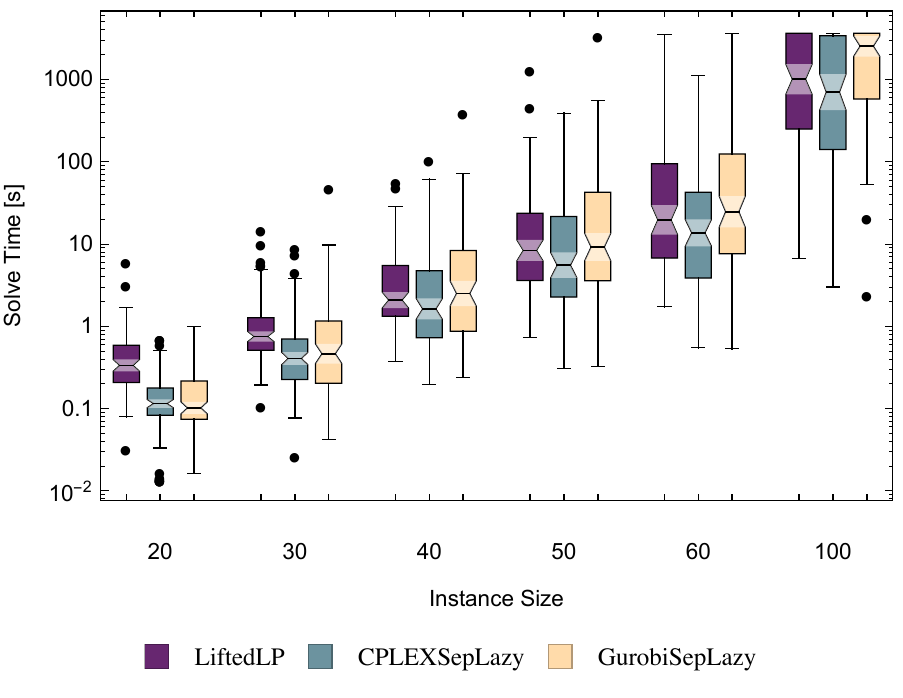}}
\subfigure[Shortfall.]{\includegraphics[scale=.85]{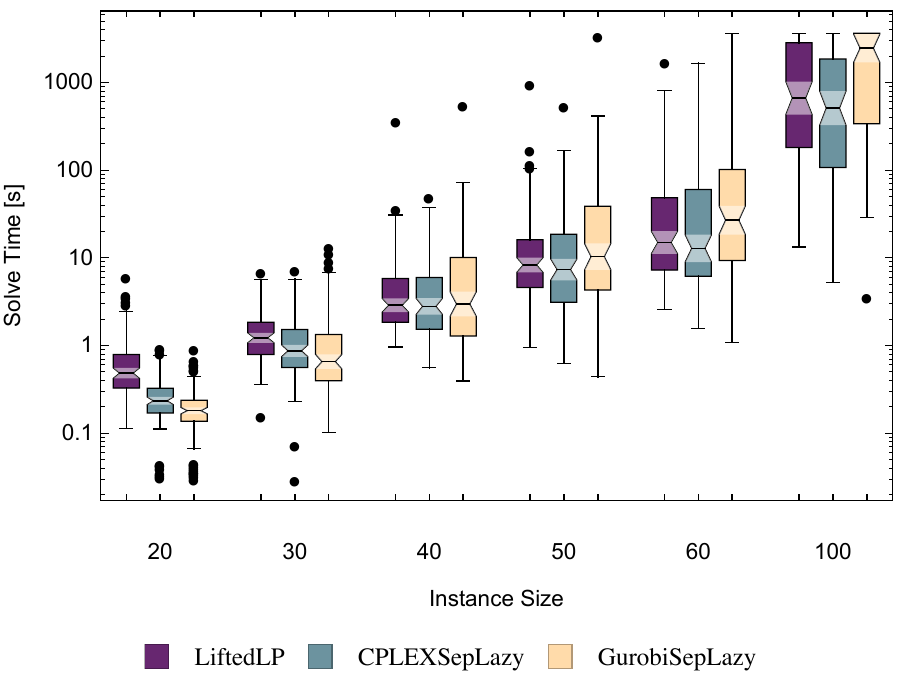}}
\subfigure[Robust.]{\includegraphics[scale=.85]{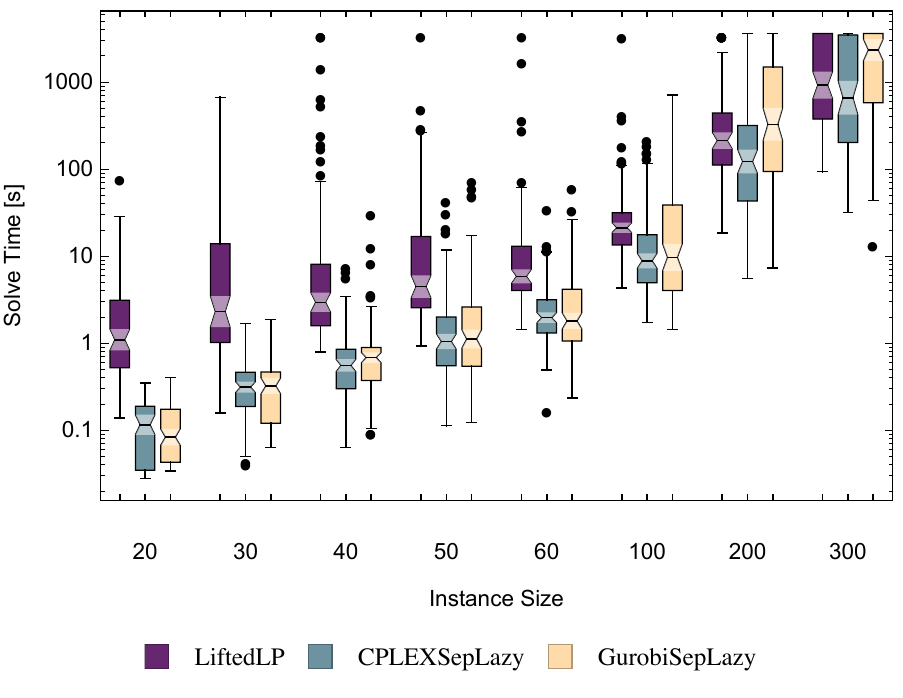}}
\caption{Solution times branch-based and cut-based  LiftedLP algorithms [s].}\label{branchvslazy}
\end{figure}
The results show that all three methods have comparable overall performances. The cut-based algorithms do sometimes provide a computational advantage, particularly for the smaller instances. Furthermore, cut-based algorithms also have the practical advantage of being easily implemented for both CPLEX and Gurobi through JuMP.

Our final set of experiments compare the branch-based and cut-based LiftedLP algorithms (LiftedLP, CPLEXSepLazy and GurobiSepLazy) the separable dynamic lifted polyhedral relaxation in its nonlinear reformulation version (CPLEXSepLP and GurobiSepLP). The results are presented through a performance profile in Figure~\ref{perprof}, \red{which also includes the four traditional algorithms (CPLEXCP, GurobiCP, CPLEXLP and GurobiLP) as a reference}. 
\begin{figure}[htb]
\centering 
\includegraphics[scale=.85]{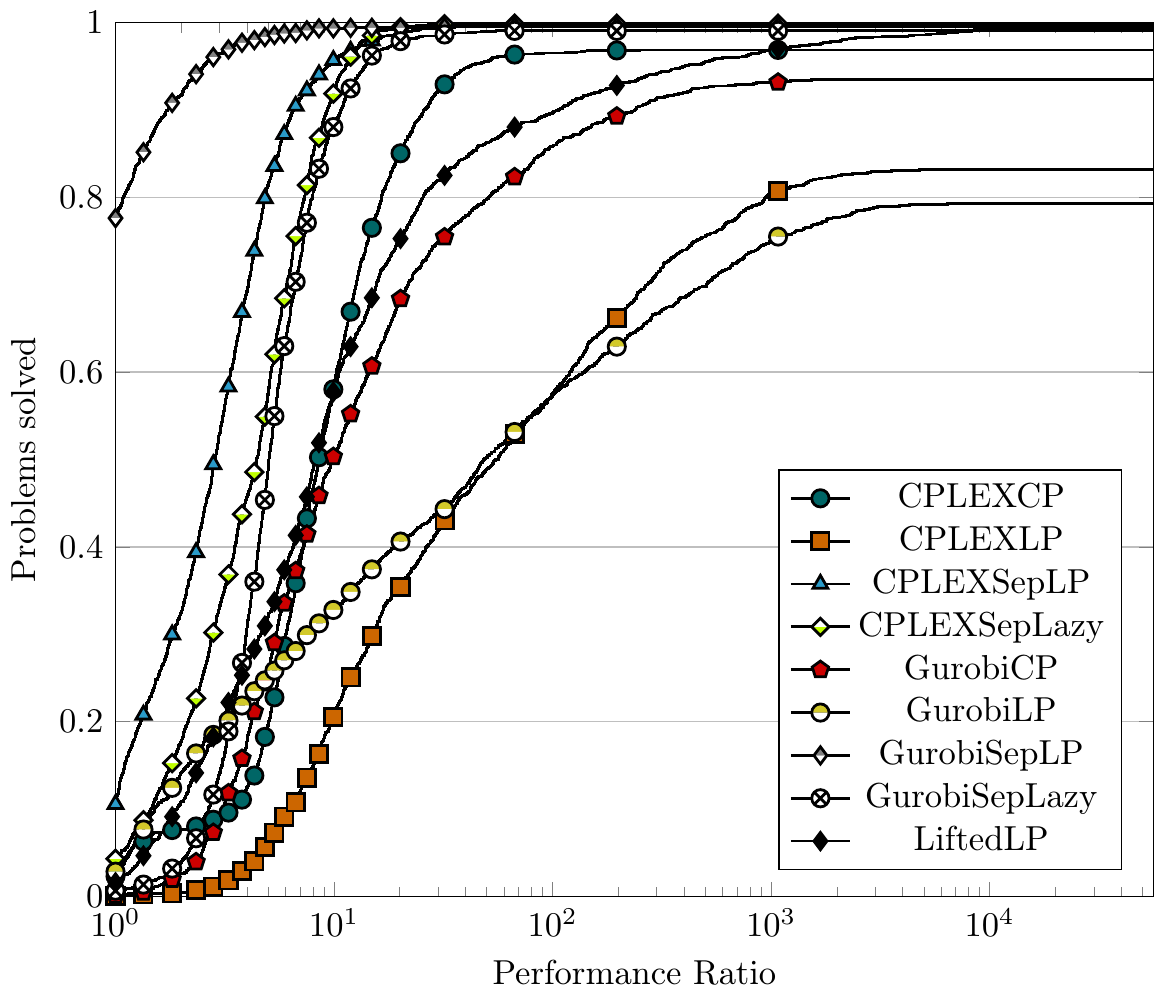}
\caption{Solution times for static and dynamic lifted LP-based, standard LP-based and standard NLP-based algorithms [s].}\label{perprof}
\end{figure}
Figure~\ref{perprof} confirms that the traditional LP-based algorithms have the worst performance and are rather consistently outperformed by the NLP-based algorithms. In addition, the branch-based LiftedLP algorithm outperforms the traditional algorithms with the exception of CPLEX's non-linear based algorithm (CPLEXCP). Furthermore, the additional advantage provided by the cut-based LiftedLP algorithms allows them to  consistently outperform all traditional algorithms.  However, the best performance is achieved by the separable dynamic lifted polyhedral relaxations in their nonlinear reformulation versions (CPLEXSepLP and GurobiSepLP). \red{Again, box-and-whisker charts and summary statistics tables with more details concerning solve times and solution quality are included in  Appendix~\ref{resultappendix}. One notable insight from these tables is that  the cut-based LiftedLP algorithms can be competitive for the hardest instances. For instance, CPLEXSepLazy was the fastest option in 37 out of the 100 shortfall instances for $n=100$.}

The JuMP implementation of the cut-based LiftedLP algorithms results on an extremely flexible framework that can significantly outperform traditional algorithms. On the other hand, the separable reformulation provides a consistently better performance without the need for callbacks or any additional programming beyond a simple transformation of the conic constraints. A possible explanation of this performance may be that callbacks interfere with (or even result in the deactivation of) advanced features of the MICQP solvers (e.g. CPLEX turns off the dynamic search feature when control callbacks such as heuristic, lazy constraint and branch callbacks are used \cite{cplexmanual}). Hence, it is possible that an internal implementation of the LiftedLP algorithms may provide an advantage in some classes of problems. \redd{Furthermore, from a purely theoretical standpoint (i.e. disregarding the mentioned algorithmic details and implementation issues), the only difference between the cut-based LiftedLP algorithms and the separable reformulation algorithms is the inclusion of the Ben-Tal and Nemirovski lifted polyhedral relaxation $L^d_{\varepsilon}$ by the first class of algorithms. This suggests that including $L^d_{\varepsilon}$ as an initial polyhedral relaxation can sometimes be useful for large instances.}

\section{Possible Extensions and Open Questions}\label{futurework}

\red{The computational results show that the separable reformulations provide a clear advantage over the original LiftedLP algorithm and standard LP-based and NLP-based algorithms for all three variants of the portfolio optimization problem. However, a further strengthening of the reformulation may be possible for the classical instances. In addition, some theoretical questions remain about the approximation quality provided by the reformulation.    }
 
\subsection{Formulation Strengthening for Portfolio Optimization Through Perspective Reformulations}

The relation between the separable reformulation and perspective formulations of unions of convex sets described in Section~\ref{prespectivesubsection} can also be used to strengthen the separable reformulation for the classical portfolio optimization instances \eqref{markowitzprob} for $Q^{1/2}=I$. If we rewrite \eqref{variance} using separable reformulation $\widehat{\mathbf{H}}^d$ defined in \eqref{separablequadratic} we obtain 
\begin{subequations} 
 \begin{alignat}{3}
  &      & x_0 &\leq \sigma,  &\label{replace1}
\\
&&x_j^2  &\leq w_j x_0,&\quad \forall j\in \set{1,\ldots,n},\label{replace2}\\
&&\sum_{j=1}^d w_j &\leq x_0,\\
&      & \sum_{j=1}^n x_j &= 1,  &
\\
& & x_j & \leq z_j, &\forall j\in\{1,\ldots,n\}, \\
&      & \sum_{j=1}^n z_j &\leq K,  &
\\
&      &z &\in \{0,1\}^{n},\\
&      &x &\in \Real_+^{n}.
 \end{alignat}
  \end{subequations}
However, using known results (e.g. \cite{Stubbs1996} and Section 3.4 of \cite{perspsurvey}), we may replace \eqref{replace1}--\eqref{replace2} with 
\begin{subequations}
 \begin{alignat}{3}
  &      & x_0 &\leq \sigma^2,  &
\\
&&x_j^2  &\leq w_j z_j,&\quad \forall j\in \set{1,\ldots,n}
 \end{alignat}
  \end{subequations}
to obtain a stronger formulation. Having $Q^{1/2}=I$ is essential for this improvement, but extending it to general $Q^{1/2}$ may be possible by using other known techniques \cite{DBLP:conf/ipco/DongL13,frangioni2007sdp,Hyemin2015}.
\red{
\subsection{Approximation Quality of Dynamic Lifted Polyhedral Relaxations}

It is well known that constructing a non-lifted (i.e. $m_2=0$ in Definition~\ref{approximationdef}) polyhedral relaxation of $\mathbf{L}^d$ with approximation quality $\varepsilon$ in \eqref{appprox}, requires at least $\exp\bra{\frac{d}{{2\bra{1+\varepsilon}^2}}}$ linear inequalities \cite{ball97}. Hence, as discussed just before Proposition~\ref{separableprop}, the approximation quality of $\widehat{L}^d_{s\bra{\varepsilon}}$ and $L^d_{\varepsilon}$ given by Proposition~\ref{nemiprop} and Corollary~\ref{glineurcoro} seems strongly dependent on the fact that the projection of  $\widehat{N}^2_s$ onto the variables $y$ has an exponential in $s$ number of inequalities. This suggests that neither the tower of variables nor the separable polyhedral relaxations can achieve the approximation efficiency of  $\widehat{L}^d_{s\bra{\varepsilon}}/L^d_{\varepsilon}$ with regard to number of linear inequalities. However it would still be interesting to understand what level of efficiency these approximations can achieve. We formalize this in the following open questions. 

\begin{question}[Smallest tower of variables polyhedral relaxation]
Let \[\Omega:=\set{\Omega_{i,k}\,:\, i\in\{1,\ldots,\lfloor t_k/2 \rfloor\},\quad k\in\{0,\ldots,K-1\}}\] be such that $\Omega_{i,k}\subseteq \mathbf{S}^1$
 and $m\bra{\Omega}:=\sum_{k=0}^{K-1} \sum_{i=1}^{\lfloor t_k/2 \rfloor} \abs{\Omega_{i,k}}$. 

 For a given $\varepsilon>0$, what is the smallest $m\bra{\Omega}$ such that $\widehat{T}^d\bra{\Omega}$ is a polyhedral relaxation of $\mathbf{L}^d$ with approximation quality $\varepsilon$?
\end{question}

\begin{question}[Smallest separable polyhedral relaxation]
Let $\Gamma:=\set{\Gamma_j\,:\, j\in \set{1,\ldots,n}}$
 and $m\bra{\Gamma}:=\sum_{j=1}^n \abs{\Gamma_j}$. 

 For a given $\varepsilon>0$, what is the smallest $m\bra{\Gamma}$ such that $\widehat{H}^d\bra{\Gamma}$ is a polyhedral relaxation of $\mathbf{L}^d$ with approximation quality $\varepsilon$?
\end{question}

Finally, \cite{ben2001lectures} shows that $\widehat{L}^d_{s\bra{\varepsilon}}/L^d_{\varepsilon}$ are essentially the smallest possible (static or dynamic) polyhedral relaxations of $\mathbf{L}^d$ with an approximation quality $\varepsilon$. Then some natural follow-up questions are: what is the smallest possible size of a dynamic polyhedral relaxation of $\mathbf{L}^d$, and how close are the tower of variables and separable approximations to this lower bound.
}

\begin{acknowledgements}
We thank the review team for their thoughtful and constructive comments, which  improved the exposition of the paper. 
Support for Juan Pablo Vielma was provided by the National Science Foundation under grant CMMI-1351619, support for Miles Lubin was provided by the DOE Computational Science Graduate Fellowship under grant DE-FG02-97ER25308 and support for Joey Huchette was provided by the National Science Foundation Graduate Research Fellowship under grant 1122374.
\end{acknowledgements}

\bibliographystyle{plain}
\bibliography{references}
\appendix
\section{Proof of Proposition~\ref{nosepprop}}\label{proofappendix}

To prove Proposition~\ref{nosepprop} we begin with the following lemma, which gives a characterization of the homogenization of a convex set described by nonlinear constraints.
\begin{lemma}\label{lemma2}
Let $f:\Real^d\to\Real$ be a closed convex function such that $\lim_{\norm{x}_2\to\infty} \frac{f(x)}{\norm{x}_2}=\infty$ so that the closure of the perspective function of $f$ is given by 
\[ \bra{cl\tilde{f}}(t,x)=\begin{cases}t f(x/t) &t>0\\ 0 & x=0 \text{ and } t=0\\\infty &\text{o.w.}\end{cases}.\]
If $\mathbf{C}:=\set{y\in \Real^d\,:\, f(y)\leq 1}$, then 
\begin{equation}
\cone\bra{\set{1}\times \mathbf{C}}=\set{\bra{y_0,y}\in \Real^{d+1}\,:\, \bra{cl\tilde{f}}(y_0,y)\leq y_0}.
\end{equation} 
\end{lemma}
\begin{proof}
Let $\mathbf{D}:=\set{\bra{y_0,y}\in \Real^{d+1}\,:\, \bra{cl\tilde{f}}(y_0,y)\leq y_0}$. We have $\set{1}\times \mathbf{C}\subseteq   \mathbf{D}$ and $\mathbf{D}$ is a convex cone \red{because $\bra{cl\tilde{f}}(y_0,y)$ is a homogeneous function}, so $\cone\bra{\set{1}\times \mathbf{C}}\subseteq \mathbf{D}$. 

For the reverse inclusion, let $\bra{y_0,y}\in \mathbf{D}$. If $y_0=0$ then $y=0$ and hence $\bra{y_0,y}\in \cone\bra{\set{1}\times \mathbf{C}}$. If $y_0>0$ then $\bra{y_0,y}/y_0\in \set{1}\times \mathbf{C}$ and hence $\bra{y_0,y}\in \cone\bra{\set{1}\times \mathbf{C}}$.
\end{proof}

The final ingredient for the proof of  Proposition~\ref{nosepprop} is the following lemma that shows how to translate the polyhedral approximation for univariate functions to their homogenization.  

\begin{lemma}\label{lemma3} Let $f:\Real \to \Real$ be a closed convex differentiable function such that $\lim_{x\to \infty} \frac{f(x)}{\abs{x}}=\infty$. Then 
\begin{align*}
\epi\bra{cl\tilde{f}}:&=\set{\bra{y_0,y,w}\in \Real^3\,:\, \bra{cl\tilde{f}}(y_0,y)\leq w}\\&=\set{\bra{y_0,y,w}\in \Real^3\,:\, \bra{f\bra{\gamma}-\gamma f'\bra{\gamma}} y_0 + f'\bra{\gamma}y   \leq w \quad\forall \gamma \in \Real}.
\end{align*}
Furthermore, $\bra{y_0,y,w}\in \epi\bra{cl\tilde{f}}$ if and only if either $y_0=y=0\leq w$ or if $y_0> 0$ and 
\begin{equation}\label{aaaa}
\bra{f\bra{\gamma\bra{y_0,y}}-\gamma\bra{y_0,y} f'\bra{\gamma\bra{y_0,y}}} y_0 + f'\bra{\gamma\bra{y_0,y}}y   \leq w
\end{equation} for $\gamma\bra{y_0,y}$ \red{defined in Corollary~\ref{coro2}}.
\end{lemma}
\begin{proof}
First note that  $\mathbf{D}:=\set{\bra{y_0,y,w}\in \Real^3\,:\, \bra{f\bra{\gamma}-\gamma f'\bra{\gamma}} y_0 + f'\bra{\gamma}y   \leq w \quad\forall \gamma \in \Real}$ is a closed convex cone, $\epi\bra{cl\tilde{f}}=\overline{\cone}\bra{\set{1}\times \epi\bra{f}}$ and \[\epi\bra{f}=\set{\bra{y,w}\in \Real^2\,:\, \bra{f\bra{\gamma}-\gamma f'\bra{\gamma}} + f'\bra{\gamma}y   \leq w \quad\forall \gamma \in \Real}.\] 
Then $\set{1}\times \epi\bra{f}\subseteq \mathbf{D}$ implies $\epi\bra{cl\tilde{f}}\subseteq \mathbf{D}$. 

For the reverse inclusion, let $\bra{y_0,y,w}\in \mathbf{D}$. We first show that $y_0\geq 0$, by assuming $y_0<0$ and reaching a contradiction. For this we consider cases $y\neq 0$ and $y=0$ separately. For both cases note that $\lim_{x\to \infty} \frac{f(x)}{\abs{x}}=\infty$ implies that $\lim_{x\to +\infty} f'(x)=+\infty$ and $\lim_{x\to -\infty} f'(x)=-\infty$. 

For case $y\neq 0$, note that if $y_0<0$, then $y_0 f(0)\leq y_0\bra{f\bra{\gamma}-\gamma f'\bra{\gamma}}$ for all $\gamma \in \Real$ by convexity of $f$. Hence, if $\bra{y_0,y,w}\in \mathbf{D}$ and $y_0<0$ then 
\begin{equation}\label{limiteq}
y_0 f\bra{0}+f'\bra{\gamma}y   \leq w
\end{equation}   
for all $\gamma\in \Real$.
Taking limit for $\gamma\to \infty$ when $y>0$ and for $\gamma\to -\infty$ when $y<0$ in \eqref{limiteq} we arrive at a contradiction with $w<\infty$.

For case $y=0$, note that by convexity of $f$ and the mean value theorem we have that $f\bra{\gamma}-\gamma f'\bra{\gamma}\leq f\bra{\gamma_0} - f'\bra{\gamma}\gamma_0$ for any $0<\gamma_0<\gamma$. Then, $\lim_{\gamma \to \infty} f\bra{\gamma}-\gamma f'\bra{\gamma}=-\infty$. Now, if $\bra{y_0,y,w}\in \mathbf{D}$ and $y=0$ then
\begin{equation}\label{limiteq2}
\bra{f\bra{\gamma}-\gamma f'\bra{\gamma}} y_0    \leq w
\end{equation} 
for all $\gamma\in \Real$.  Then, if $y_0<0$, by taking limit for $\gamma\to \infty$ in \eqref{limiteq2} we again arrive at a contradiction with $w<\infty$. 

We now show that any $\bra{y_0,y,w}\in \mathbf{D}$ with $y_0\geq 0$ also belongs to $\epi\bra{cl\tilde{f}}$. We divide the proof into cases $y_0=0$ and $y_0>0$.

For case, $y_0=0$, note that $\bra{y_0,y,w}\in \mathbf{D}$ implies $f'\bra{\gamma}y   \leq w$ for all $\gamma\in \Real$. Taking limit  for $\gamma\to \infty$ when $y>0 $ and for $\gamma\to -\infty$ when $y<0$, we conclude that $y=0$ and $w\geq 0$. Hence, $\bra{y_0,y,w}\in \epi\bra{cl\tilde{f}}$.

For case, $y_0>0$ we can check that $\bra{y,w}/y_0\in   \epi\bra{f}$ and then $\bra{y_0,y,w}/y_0\in   \set{1}\times \epi\bra{f}$. Hence, $\bra{y_0,y,w}\in \cone\bra{\set{1}\times \epi\bra{f}}\subseteq \epi\bra{cl\tilde{f}}$.

For the final statement, it suffices to prove that $\bra{y_0,y,w}\in \epi\bra{cl\tilde{f}}$ if $y_0>0$ and $\bra{y_0,y,w}$ satisfies \eqref{aaaa}. For this note that, if the last two conditions hold, then   $\bra{y,w}/y_0\in   \epi\bra{f}$, which we have already shown implies $\bra{y_0,y,w}\in  \epi\bra{cl\tilde{f}}$.
\end{proof}
Combining these lemmas we obtain the following straightforward proof  of Proposition~\ref{nosepprop}.

\begin{proof}[of Proposition~\ref{nosepprop}]
Let $f(y)=\sum_{j=1}^d f_j(y_j)$. Then \eqref{eee} follows by noting that by Lemma~\ref{lemma2} we have $\cone\bra{\set{1}\times \mathbf{C}}=\set{\bra{y_0,y}\in \Real^{d+1}\,:\, \bra{cl\tilde{f}}(y_0,y)\leq y_0}$ and by the definition of perspective function  we have $\bra{cl\tilde{f}}(y_0,y)=\sum_{j=1}^d \bra{cl\tilde{f_j}}(y_0,y_j)$. All other statements are directly from Lemma~\ref{lemma3} and from Proposition~\ref{separableprop} by fixing $y_0=1$.

\end{proof}

\section{Additional Graphs and Tables}\label{resultappendix}
\subsection{Summary Statistics for Solve Times}\label{apendixA}

Tables~\ref{firstsolvetable}--\ref{lastsolvetable} show some summary statistics for the solve times for the different algorithms and instances. These statistics include the minimum, average and maximum solve times, together with their standard deviation. The tables also include the number of each solver was the fastest (wins) and the number of times a solver has a solution time that was within 1\% and 10\% of the fastest solver (1\% and 10\% win).

\begin{table}[H]
\centering
\subtable[Classical.]{\begin{tabular}{lrrrrrrr}
Algorithm & min & avg & max & std
& wins & 1\% win & 10\% win 
\\
\hline
CPLEXCP&0.01& \bf0.23& 2.06& 0.30&0&0&2\\
GurobiCP&0.03& \bf0.17& 2.22& 0.31&0&0&0\\
CPLEXLP&0.02& \bf1.23& 16.67& 2.69&0&0&0\\
GurobiLP&0.01& \bf1.63& 26.48& 4.04&0&0&0\\
LiftedLP&0.03& \bf0.53& 6.44& 0.74&0&0&0\\
CPLEXSepLp&0.01& \bf0.09& 0.53& 0.09&4&4&6\\
CPLEXTowerLp&0.01& \bf0.11& 0.71& 0.10&0&0&3\\
CPLEXTowerSepLp&0.01& \bf0.23& 1.49& 0.23&0&0&0\\
GurobiSepLp&0.01& \bf0.03& 0.18& 0.02&58&58&76\\
GurobiTowerLp&0.01& \bf0.03& 0.21& 0.03&38&38&58\\
GurobiTowerSepLp&0.01& \bf0.15& 1.06& 0.22&0&0&0\\
CPLEXSepLazy&0.01& \bf0.16& 0.74& 0.14&0&0&0\\
GurobiSepLazy&0.02& \bf0.17& 0.98& 0.18&0&0&0
\end{tabular}
}
\subtable[Shortfall.]{\begin{tabular}{lrrrrrrr}
Algorithm & min & avg & max & std
& wins & 1\% win & 10\% win 
\\
\hline
CPLEXCP&0.01& \bf0.23& 0.91& 0.18&1&1&1\\
GurobiCP&0.02& \bf0.76& 6.12& 1.32&2&2&3\\
CPLEXLP&0.01& \bf0.69& 13.57& 1.59&0&0&1\\
GurobiLP&0.02& \bf0.97& 16.03& 2.18&0&0&0\\
LiftedLP&0.11& \bf0.82& 6.27& 0.95&0&0&0\\
CPLEXSepLp&0.01& \bf0.15& 0.81& 0.13&2&2&3\\
CPLEXTowerLp&0.01& \bf0.12& 0.42& 0.09&5&5&6\\
CPLEXTowerSepLp&0.03& \bf0.29& 0.80& 0.19&1&1&1\\
GurobiSepLp&0.01& \bf0.04& 0.12& 0.02&35&37&53\\
GurobiTowerLp&0.01& \bf0.04& 0.14& 0.02&54&55&67\\
GurobiTowerSepLp&0.03& \bf0.17& 1.15& 0.19&0&0&0\\
CPLEXSepLazy&0.03& \bf0.28& 0.98& 0.19&0&0&0\\
GurobiSepLazy&0.03& \bf0.21& 0.97& 0.15&0&0&0
\end{tabular}
}
\subtable[Robust.]{\begin{tabular}{lrrrrrrr}
Algorithm & min & avg & max & std
& wins & 1\% win & 10\% win 
\\
\hline
CPLEXCP&0.01& \bf0.10& 0.55& 0.10&19&20&33\\
GurobiCP&0.03& \bf0.09& 0.41& 0.05&0&0&0\\
CPLEXLP&0.01& \bf0.23& 1.14& 0.17&0&0&0\\
GurobiLP&0.01& \bf0.07& 1.05& 0.14&17&17&30\\
LiftedLP&0.14& \bf4.03& 81.41& 9.25&0&0&0\\
CPLEXSepLp&0.01& \bf0.08& 0.34& 0.08&7&9&30\\
CPLEXTowerLp&0.01& \bf0.07& 0.32& 0.07&16&18&33\\
CPLEXTowerSepLp&0.03& \bf0.15& 0.61& 0.14&0&0&0\\
GurobiSepLp&0.01& \bf0.03& 0.07& 0.01&39&40&48\\
GurobiTowerLp&0.01& \bf0.04& 0.20& 0.03&2&2&5\\
GurobiTowerSepLp&0.03& \bf0.08& 0.36& 0.06&0&0&0\\
CPLEXSepLazy&0.03& \bf0.12& 0.35& 0.09&0&0&0\\
GurobiSepLazy&0.03& \bf0.12& 0.40& 0.09&0&0&0
\end{tabular}
}
\caption{Summary statistics of Solve Times for $n=20$ [s].}\label{firstsolvetable}
\end{table}
\begin{table}[H]
\centering
\subtable[Classical.]{\begin{tabular}{lrrrrrrr}
Algorithm & min & avg & max & std
& wins & 1\% win & 10\% win 
\\
\hline
CPLEXCP&0.14& \bf2.44& 91.52& 9.54&0&0&0\\
GurobiCP&0.08& \bf2.95& 124.47& 12.81&0&0&0\\
CPLEXLP&0.08& \bf101.30& 2795.96& 391.38&0&0&0\\
GurobiLP&0.05& \bf174.52& 3600.01& 605.37&0&0&0\\
LiftedLP&0.11& \bf1.35& 15.74& 2.10&0&0&0\\
CPLEXSepLp&0.01& \bf0.49& 4.45& 0.72&1&2&2\\
CPLEXTowerLp&0.01& \bf0.68& 8.36& 1.15&0&0&1\\
CPLEXTowerSepLp&0.02& \bf1.49& 21.18& 2.65&0&0&0\\
GurobiSepLp&0.01& \bf0.21& 3.94& 0.46&65&66&72\\
GurobiTowerLp&0.02& \bf0.27& 5.83& 0.69&34&37&46\\
GurobiTowerSepLp&0.04& \bf2.65& 96.69& 10.10&0&0&0\\
CPLEXSepLazy&0.03& \bf0.80& 9.53& 1.38&0&0&0\\
GurobiSepLazy&0.04& \bf1.59& 51.04& 5.30&0&0&0
\end{tabular}
}
\subtable[Shortfall.]{\begin{tabular}{lrrrrrrr}
Algorithm & min & avg & max & std
& wins & 1\% win & 10\% win 
\\
\hline
CPLEXCP&0.05& \bf2.04& 20.97& 3.31&0&0&0\\
GurobiCP&0.04& \bf11.88& 154.60& 22.32&0&0&0\\
CPLEXLP&0.02& \bf22.79& 406.01& 62.02&0&0&0\\
GurobiLP&0.05& \bf42.13& 989.04& 145.49&0&0&0\\
LiftedLP&0.17& \bf1.55& 7.20& 1.23&0&0&0\\
CPLEXSepLp&0.03& \bf0.72& 8.32& 0.90&4&4&4\\
CPLEXTowerLp&0.02& \bf0.75& 7.74& 0.97&2&2&2\\
CPLEXTowerSepLp&0.05& \bf1.90& 21.70& 2.69&0&0&0\\
GurobiSepLp&0.02& \bf0.23& 2.87& 0.43&94&94&95\\
GurobiTowerLp&0.03& \bf0.83& 8.87& 1.53&0&0&1\\
GurobiTowerSepLp&0.05& \bf2.09& 25.97& 4.25&0&0&0\\
CPLEXSepLazy&0.03& \bf1.21& 7.64& 1.11&0&0&0\\
GurobiSepLazy&0.10& \bf1.39& 13.96& 2.24&0&0&0
\end{tabular}
}
\subtable[Robust.]{\begin{tabular}{lrrrrrrr}
Algorithm & min & avg & max & std
& wins & 1\% win & 10\% win 
\\
\hline
CPLEXCP&0.02& \bf0.41& 2.12& 0.39&4&4&8\\
GurobiCP&0.12& \bf0.69& 8.47& 1.35&0&0&0\\
CPLEXLP&0.06& \bf1.56& 33.38& 4.20&0&0&0\\
GurobiLP&0.02& \bf0.68& 11.97& 1.69&6&6&11\\
LiftedLP&0.16& \bf23.53& 666.01& 77.12&0&0&0\\
CPLEXSepLp&0.02& \bf0.26& 0.84& 0.23&15&16&23\\
CPLEXTowerLp&0.02& \bf0.25& 1.14& 0.25&11&11&18\\
CPLEXTowerSepLp&0.05& \bf0.55& 2.67& 0.54&0&0&0\\
GurobiSepLp&0.02& \bf0.10& 1.33& 0.14&62&62&64\\
GurobiTowerLp&0.02& \bf0.16& 0.96& 0.19&2&2&3\\
GurobiTowerSepLp&0.05& \bf0.38& 2.67& 0.48&0&0&0\\
CPLEXSepLazy&0.04& \bf0.35& 1.68& 0.26&0&0&0\\
GurobiSepLazy&0.06& \bf0.39& 1.84& 0.34&0&0&0
\end{tabular}
}
\caption{Summary statistics of Solve Times for $n=30$ [s].}
\end{table}
\begin{table}[H]
\centering
\subtable[Classical.]{\begin{tabular}{lrrrrrrr}
Algorithm & min & avg & max & std
& wins & 1\% win & 10\% win 
\\
\hline
CPLEXCP&0.33& \bf17.43& 568.02& 58.58&0&0&0\\
GurobiCP&0.21& \bf24.56& 847.80& 88.11&0&0&0\\
CPLEXLP&0.62& \bf731.57& 3600.42& 1132.98&0&0&0\\
GurobiLP&0.64& \bf1157.74& 3600.30& 1462.59&0&0&0\\
LiftedLP&0.37& \bf5.36& 60.57& 9.19&0&0&0\\
CPLEXSepLp&0.13& \bf3.42& 70.71& 7.78&3&4&5\\
CPLEXTowerLp&0.20& \bf4.89& 102.78& 12.01&1&1&1\\
CPLEXTowerSepLp&0.28& \bf10.43& 207.66& 24.01&0&0&0\\
GurobiSepLp&0.05& \bf1.28& 28.18& 3.24&94&95&96\\
GurobiTowerLp&0.11& \bf8.85& 249.85& 26.53&1&1&1\\
GurobiTowerSepLp&0.17& \bf21.59& 646.23& 67.72&0&0&0\\
CPLEXSepLazy&0.20& \bf5.78& 111.98& 14.02&0&0&0\\
GurobiSepLazy&0.24& \bf10.71& 415.53& 42.17&1&1&1
\end{tabular}
}
\subtable[Shortfall.]{\begin{tabular}{lrrrrrrr}
Algorithm & min & avg & max & std
& wins & 1\% win & 10\% win 
\\
\hline
CPLEXCP&0.64& \bf29.08& 1168.34& 118.88&0&0&0\\
GurobiCP&0.40& \bf152.28& 1416.22& 293.82&0&0&0\\
CPLEXLP&0.78& \bf325.52& 3600.01& 738.07&0&0&0\\
GurobiLP&0.55& \bf685.09& 3600.86& 1160.38&0&0&0\\
LiftedLP&0.95& \bf9.23& 379.92& 38.05&1&1&2\\
CPLEXSepLp&0.33& \bf4.34& 64.07& 8.61&3&4&6\\
CPLEXTowerLp&0.22& \bf4.43& 84.74& 9.68&1&1&1\\
CPLEXTowerSepLp&0.64& \bf11.51& 168.16& 21.32&0&0&0\\
GurobiSepLp&0.08& \bf2.45& 74.77& 7.90&93&93&96\\
GurobiTowerLp&0.12& \bf8.80& 218.17& 23.77&0&0&0\\
GurobiTowerSepLp&0.31& \bf31.74& 1393.77& 141.23&0&0&0\\
CPLEXSepLazy&0.55& \bf5.45& 52.10& 7.71&2&2&3\\
GurobiSepLazy&0.39& \bf13.79& 578.08& 58.35&0&0&0
\end{tabular}
}
\subtable[Robust.]{\begin{tabular}{lrrrrrrr}
Algorithm & min & avg & max & std
& wins & 1\% win & 10\% win 
\\
\hline
CPLEXCP&0.04& \bf1.86& 43.15& 4.95&2&2&2\\
GurobiCP&0.08& \bf1.92& 41.76& 6.11&0&0&0\\
CPLEXLP&0.08& \bf22.79& 1197.49& 130.23&0&0&0\\
GurobiLP&0.02& \bf21.77& 1392.90& 149.35&4&4&10\\
LiftedLP&0.78& \bf114.09& 3600.00& 530.85&0&0&1\\
CPLEXSepLp&0.03& \bf0.71& 8.93& 1.10&9&9&11\\
CPLEXTowerLp&0.03& \bf0.84& 13.61& 1.79&8&8&13\\
CPLEXTowerSepLp&0.07& \bf1.87& 34.23& 3.97&0&0&0\\
GurobiSepLp&0.03& \bf0.26& 4.54& 0.51&69&69&74\\
GurobiTowerLp&0.03& \bf0.83& 19.90& 2.66&7&7&13\\
GurobiTowerSepLp&0.07& \bf2.07& 52.12& 6.82&0&0&0\\
CPLEXSepLazy&0.06& \bf0.85& 7.78& 1.24&1&1&1\\
GurobiSepLazy&0.10& \bf1.30& 32.36& 3.54&0&0&0
\end{tabular}
}
\caption{Summary statistics of Solve Times for $n=40$ [s].}
\end{table}
\begin{table}[H]
\centering
\subtable[Classical.]{\begin{tabular}{lrrrrrrr}
Algorithm & min & avg & max & std
& wins & 1\% win & 10\% win 
\\
\hline
CPLEXCP&0.43& \bf188.32& 3600.25& 511.63&0&0&0\\
GurobiCP&0.33& \bf254.29& 3600.01& 619.08&0&0&0\\
CPLEXLP&0.52& \bf1941.58& 3600.07& 1515.59&0&0&0\\
GurobiLP&1.06& \bf2375.47& 3608.86& 1488.32&0&0&0\\
LiftedLP&0.73& \bf42.46& 1368.69& 147.18&1&1&2\\
CPLEXSepLp&0.14& \bf21.88& 471.47& 57.54&7&8&9\\
CPLEXTowerLp&0.24& \bf33.96& 858.26& 104.09&1&1&2\\
CPLEXTowerSepLp&0.57& \bf67.57& 1957.18& 210.80&0&0&0\\
GurobiSepLp&0.07& \bf21.63& 1255.85& 126.13&87&88&91\\
GurobiTowerLp&0.08& \bf102.69& 3600.00& 382.69&0&0&1\\
GurobiTowerSepLp&0.31& \bf197.25& 3600.00& 501.15&0&0&0\\
CPLEXSepLazy&0.30& \bf28.20& 384.63& 65.40&4&4&4\\
GurobiSepLazy&0.32& \bf88.79& 3600.00& 369.09&0&0&0
\end{tabular}
}
\subtable[Shortfall.]{\begin{tabular}{lrrrrrrr}
Algorithm & min & avg & max & std
& wins & 1\% win & 10\% win 
\\
\hline
CPLEXCP&1.01& \bf160.11& 3600.22& 427.35&0&0&0\\
GurobiCP&0.53& \bf1148.66& 3600.02& 1460.67&0&0&0\\
CPLEXLP&0.62& \bf1231.56& 3600.44& 1446.50&0&0&0\\
GurobiLP&0.32& \bf1773.43& 3658.96& 1614.40&0&0&0\\
LiftedLP&0.94& \bf28.42& 1000.65& 102.41&2&3&3\\
CPLEXSepLp&0.35& \bf16.64& 458.53& 48.68&12&12&15\\
CPLEXTowerLp&0.35& \bf24.49& 889.35& 90.93&4&4&5\\
CPLEXTowerSepLp&0.54& \bf77.90& 2216.07& 261.45&0&0&0\\
GurobiSepLp&0.12& \bf20.27& 876.51& 91.10&78&79&79\\
GurobiTowerLp&0.19& \bf74.46& 3600.00& 363.09&1&1&1\\
GurobiTowerSepLp&0.58& \bf147.20& 3600.00& 405.19&0&0&0\\
CPLEXSepLazy&0.62& \bf21.93& 572.22& 62.57&3&3&6\\
GurobiSepLazy&0.44& \bf74.11& 3562.78& 359.48&0&0&0
\end{tabular}
}
\subtable[Robust.]{\begin{tabular}{lrrrrrrr}
Algorithm & min & avg & max & std
& wins & 1\% win & 10\% win 
\\
\hline
CPLEXCP&0.06& \bf9.35& 226.54& 31.22&3&3&3\\
GurobiCP&0.12& \bf10.76& 355.45& 42.73&0&0&0\\
CPLEXLP&0.27& \bf133.90& 3600.02& 563.01&0&0&0\\
GurobiLP&0.04& \bf162.41& 3604.89& 709.57&3&4&4\\
LiftedLP&0.92& \bf62.95& 3600.01& 365.07&1&1&1\\
CPLEXSepLp&0.04& \bf1.95& 20.83& 3.70&14&14&19\\
CPLEXTowerLp&0.04& \bf2.92& 72.92& 8.85&1&1&3\\
CPLEXTowerSepLp&0.10& \bf7.49& 190.99& 23.24&0&0&0\\
GurobiSepLp&0.04& \bf1.64& 57.29& 6.40&71&71&77\\
GurobiTowerLp&0.05& \bf4.75& 103.85& 16.26&6&6&11\\
GurobiTowerSepLp&0.14& \bf11.63& 253.89& 39.10&0&0&0\\
CPLEXSepLazy&0.11& \bf2.65& 45.20& 6.24&1&1&1\\
GurobiSepLazy&0.12& \bf4.40& 76.78& 12.17&0&0&0
\end{tabular}
}
\caption{Summary statistics of Solve Times for $n=50$ [s].}
\end{table}
\begin{table}[H]
\centering
\subtable[Classical.]{\begin{tabular}{lrrrrrrr}
Algorithm & min & avg & max & std
& wins & 1\% win & 10\% win 
\\
\hline
CPLEXCP&0.63& \bf440.28& 3602.68& 851.14&0&0&0\\
GurobiCP&0.57& \bf574.65& 3600.00& 1016.49&0&0&0\\
CPLEXLP&3.50& \bf2664.89& 3601.10& 1354.02&0&0&0\\
GurobiLP&4.35& \bf2969.53& 3617.48& 1237.05&0&0&0\\
LiftedLP&1.71& \bf141.92& 3471.98& 452.83&1&1&2\\
CPLEXSepLp&0.31& \bf68.85& 1293.39& 174.87&4&4&8\\
CPLEXTowerLp&0.38& \bf110.91& 2728.20& 318.85&0&0&1\\
CPLEXTowerSepLp&0.67& \bf212.91& 3600.00& 526.40&0&0&0\\
GurobiSepLp&0.08& \bf80.51& 3600.00& 374.64&74&75&77\\
GurobiTowerLp&0.15& \bf254.80& 3600.00& 649.20&0&0&0\\
GurobiTowerSepLp&0.31& \bf477.73& 3600.00& 886.20&0&0&0\\
CPLEXSepLazy&0.54& \bf56.38& 1100.44& 149.15&20&20&20\\
GurobiSepLazy&0.53& \bf200.00& 3600.00& 541.28&1&1&1
\end{tabular}
}
\subtable[Shortfall.]{\begin{tabular}{lrrrrrrr}
Algorithm & min & avg & max & std
& wins & 1\% win & 10\% win 
\\
\hline
CPLEXCP&1.58& \bf447.80& 3600.00& 849.68&0&0&0\\
GurobiCP&4.71& \bf1947.87& 3600.03& 1594.42&0&0&0\\
CPLEXLP&2.51& \bf2065.17& 3601.08& 1569.51&0&0&0\\
GurobiLP&1.61& \bf2440.88& 3609.79& 1453.93&0&0&0\\
LiftedLP&2.55& \bf75.88& 1793.78& 211.91&9&9&10\\
CPLEXSepLp&0.82& \bf61.51& 2284.68& 234.88&9&9&13\\
CPLEXTowerLp&0.51& \bf97.22& 3600.95& 406.79&2&3&3\\
CPLEXTowerSepLp&1.93& \bf238.37& 3600.00& 558.70&0&0&0\\
GurobiSepLp&0.16& \bf84.67& 3600.00& 383.37&65&65&68\\
GurobiTowerLp&0.32& \bf181.66& 3600.00& 493.36&0&0&1\\
GurobiTowerSepLp&0.75& \bf386.34& 3600.01& 772.94&0&0&0\\
CPLEXSepLazy&1.56& \bf59.45& 1643.17& 175.33&15&15&18\\
GurobiSepLazy&1.07& \bf141.33& 3600.00& 416.08&0&0&0
\end{tabular}
}
\subtable[Robust.]{\begin{tabular}{lrrrrrrr}
Algorithm & min & avg & max & std
& wins & 1\% win & 10\% win 
\\
\hline
CPLEXCP&0.08& \bf12.35& 325.85& 34.59&1&1&1\\
GurobiCP&0.61& \bf21.46& 1006.61& 103.22&0&0&0\\
CPLEXLP&0.52& \bf197.03& 3600.02& 608.43&0&0&0\\
GurobiLP&0.07& \bf211.17& 3600.01& 681.26&1&1&2\\
LiftedLP&1.43& \bf71.35& 3600.01& 399.94&3&3&4\\
CPLEXSepLp&0.41& \bf2.77& 26.13& 3.25&10&10&16\\
CPLEXTowerLp&0.06& \bf4.12& 62.27& 7.32&2&2&3\\
CPLEXTowerSepLp&0.15& \bf9.01& 124.45& 14.65&0&0&0\\
GurobiSepLp&0.13& \bf2.28& 29.70& 4.60&71&72&73\\
GurobiTowerLp&0.09& \bf5.29& 113.54& 13.76&11&11&14\\
GurobiTowerSepLp&0.21& \bf14.59& 300.93& 37.99&0&0&0\\
CPLEXSepLazy&0.17& \bf3.31& 36.10& 4.39&0&0&2\\
GurobiSepLazy&0.23& \bf5.10& 64.45& 9.08&1&1&2
\end{tabular}
}
\caption{Summary statistics of Solve Times for $n=60$ [s].}
\end{table}

\begin{table}[H]
\centering
\subtable[Classical.]{\begin{tabular}{lrrrrrrr}
Algorithm & min & avg & max & std
& wins & 1\% win & 10\% win 
\\
\hline
LiftedLP&6.62& \bf1624.83& 3600.02& 1471.37&3&17&22\\
CPLEXSepLp&1.66& \bf1484.69& 3600.21& 1403.10&14&28&33\\
GurobiSepLp&0.34& \bf1439.88& 3600.31& 1437.49&42&56&57\\
CPLEXSepLazy&2.98& \bf1452.46& 3600.02& 1463.32&40&40&45\\
GurobiSepLazy&2.55& \bf2155.33& 3600.02& 1477.84&1&15&16
\end{tabular}
}
\subtable[Shortfall.]{\begin{tabular}{lrrrrrrr}
Algorithm & min & avg & max & std
& wins & 1\% win & 10\% win 
\\
\hline
LiftedLP&13.16& \bf1345.62& 3600.12& 1391.11&19&24&30\\
CPLEXSepLp&2.39& \bf1232.79& 3600.11& 1302.56&14&22&27\\
GurobiSepLp&0.58& \bf1493.44& 3600.19& 1493.08&29&37&40\\
CPLEXSepLazy&5.16& \bf1129.58& 3600.12& 1305.08&37&42&46\\
GurobiSepLazy&3.73& \bf2081.67& 3600.08& 1506.22&1&8&8
\end{tabular}
}
\subtable[Robust.]{\begin{tabular}{lrrrrrrr}
Algorithm & min & avg & max & std
& wins & 1\% win & 10\% win 
\\
\hline
CPLEXCP&3.62& \bf162.04& 2043.65& 375.08&0&0&0\\
GurobiCP&2.55& \bf196.83& 3600.03& 509.70&0&0&0\\
CPLEXLP&2.00& \bf1056.43& 3600.10& 1362.59&0&0&0\\
GurobiLP&0.47& \bf1079.13& 3603.90& 1472.19&1&1&1\\
LiftedLP&4.27& \bf71.45& 3430.89& 345.06&2&2&6\\
CPLEXSepLp&1.38& \bf17.48& 144.82& 26.37&25&25&30\\
CPLEXTowerLp&0.99& \bf29.77& 535.74& 62.07&0&0&0\\
CPLEXTowerSepLp&2.42& \bf104.37& 2712.22& 324.25&0&0&0\\
GurobiSepLp&0.39& \bf32.04& 595.06& 86.96&57&57&59\\
GurobiTowerLp&0.62& \bf58.55& 658.53& 124.77&5&6&7\\
GurobiTowerSepLp&1.01& \bf167.12& 1750.92& 343.94&0&0&0\\
CPLEXSepLazy&1.73& \bf23.22& 229.88& 41.66&7&7&10\\
GurobiSepLazy&1.43& \bf54.06& 706.76& 126.63&3&3&5
\end{tabular}
}
\caption{Summary statistics of Solve Times for $n=100$ [s].}
\end{table}

\begin{table}[H]
\centering
\subtable[Robust.]{\begin{tabular}{lrrrrrrr}
Algorithm & min & avg & max & std
& wins & 1\% win & 10\% win 
\\
\hline
CPLEXCP&28.85& \bf2172.25& 3600.07& 1441.31&0&1&1\\
GurobiCP&21.90& \bf2053.05& 3600.09& 1435.75&0&1&1\\
CPLEXLP&14.74& \bf3036.45& 3600.11& 1143.01&0&1&1\\
GurobiLP&5.63& \bf2841.03& 3604.06& 1342.22&0&1&1\\
LiftedLP&18.26& \bf529.57& 3600.04& 832.69&6&8&11\\
CPLEXSepLp&5.00& \bf426.90& 3600.04& 854.99&23&25&33\\
CPLEXTowerLp&6.77& \bf740.08& 3600.06& 1163.97&0&1&2\\
CPLEXTowerSepLp&12.89& \bf1151.67& 3600.40& 1271.52&0&1&1\\
GurobiSepLp&1.03& \bf593.35& 3600.05& 1064.61&54&55&56\\
GurobiTowerLp&3.02& \bf1033.21& 3600.05& 1294.31&1&2&2\\
GurobiTowerSepLp&6.31& \bf1685.98& 3600.07& 1411.79&0&1&1\\
CPLEXSepLazy&5.51& \bf406.56& 3600.01& 752.93&16&17&24\\
GurobiSepLazy&7.23& \bf954.15& 3600.09& 1245.66&0&1&2
\end{tabular}
}
\caption{Summary statistics of Solve Times for $n=200$ [s].}\label{summaryfortwohundred}
\end{table}

\begin{table}[H]
\centering
\subtable[Robust.]{\begin{tabular}{lrrrrrrr}
Algorithm & min & avg & max & std
& wins & 1\% win & 10\% win 
\\
\hline
LiftedLP&92.22& \bf1660.09& 3602.30& 1426.85&1&20&20\\
CPLEXSepLp&27.42& \bf1381.03& 3600.08& 1417.33&23&42&46\\
GurobiSepLp&4.11& \bf1429.40& 3600.12& 1432.61&45&65&66\\
CPLEXSepLazy&31.46& \bf1390.95& 3600.05& 1422.92&22&31&34\\
GurobiSepLazy&13.96& \bf2114.17& 3600.05& 1470.40&9&19&19
\end{tabular}
}
\caption{Summary statistics of Solve Times for $n=300$ [s].}\label{lastsolvetable}
\end{table}
\subsection{Summary Statistics for Solution Quality}\label{qualityappendix}

Tables~\ref{firstqualitytable}--\ref{lastqualitytable} present summary statistics for the maximum violation of constraints \eqref{conicc} by the optimal solution returned by the solver. This value is calculated as
\[
\max_{l=1}^q \norm{A^l \bar{x} + b^l}_2^2 - \bra{a^l \bar{x} + b_0^l}^2
\]
where $\bar{x}$ is the optimal solution. 
\begin{table}[H]
\centering
\subtable[Classical.]{\begin{tabular}{lrrrrrrr}
Algorithm & min & avg & max & std
\\
\hline
LiftedLP&0.00e+00& \bf5.87e-09& 8.31e-08& 1.20e-08\\
CPLEXSepLp&9.36e-11& \bf5.66e-09& 5.50e-08& 8.56e-09\\
CPLEXTowerLp&0.00e+00& \bf6.24e-09& 6.36e-08& 9.71e-09\\
CPLEXTowerSepLp&0.00e+00& \bf4.16e-09& 3.59e-08& 7.12e-09\\
GurobiSepLp&3.84e-04& \bf5.98e-04& 8.28e-04& 8.31e-05\\
GurobiTowerLp&3.88e-04& \bf5.68e-04& 7.76e-04& 8.14e-05\\
GurobiTowerSepLp&4.02e-06& \bf3.92e-05& 1.54e-04& 3.13e-05\\
CPLEXSepLazy&4.68e-06& \bf6.40e-06& 8.82e-06& 8.99e-07\\
GurobiSepLazy&4.11e-06& \bf6.31e-06& 8.72e-06& 9.17e-07
\end{tabular}
}
\subtable[Shortfall.]{\begin{tabular}{lrrrrrrr}
Algorithm & min & avg & max & std
\\
\hline
LiftedLP&0.00e+00& \bf0.00e+00& 0.00e+00& 0.00e+00\\
CPLEXSepLp&0.00e+00& \bf5.19e-10& 1.30e-08& 2.03e-09\\
CPLEXTowerLp&0.00e+00& \bf1.67e-11& 1.09e-09& 1.23e-10\\
CPLEXTowerSepLp&0.00e+00& \bf6.25e-10& 5.18e-08& 5.21e-09\\
GurobiSepLp&4.19e-04& \bf6.47e-04& 8.46e-04& 8.87e-05\\
GurobiTowerLp&3.81e-04& \bf5.65e-04& 8.39e-04& 9.37e-05\\
GurobiTowerSepLp&4.82e-06& \bf5.36e-05& 2.50e-04& 4.29e-05\\
CPLEXSepLazy&0.00e+00& \bf6.37e-06& 9.33e-06& 1.35e-06\\
GurobiSepLazy&0.00e+00& \bf6.17e-06& 8.90e-06& 1.47e-06
\end{tabular}
}
\subtable[Robust.]{\begin{tabular}{lrrrrrrr}
Algorithm & min & avg & max & std
\\
\hline
LiftedLP&0.00e+00& \bf6.59e-09& 8.05e-08& 1.35e-08\\
CPLEXSepLp&1.26e-11& \bf4.74e-09& 3.57e-08& 6.12e-09\\
CPLEXTowerLp&7.06e-11& \bf5.09e-09& 8.23e-08& 1.01e-08\\
CPLEXTowerSepLp&1.02e-11& \bf2.40e-09& 2.69e-08& 4.59e-09\\
GurobiSepLp&4.51e-04& \bf6.27e-04& 1.02e-03& 9.93e-05\\
GurobiTowerLp&1.92e-06& \bf3.07e-05& 1.19e-04& 2.11e-05\\
GurobiTowerSepLp&2.27e-06& \bf4.90e-05& 1.44e-04& 3.46e-05\\
CPLEXSepLazy&2.29e-06& \bf4.83e-06& 8.02e-06& 1.16e-06\\
GurobiSepLazy&2.29e-06& \bf4.71e-06& 7.50e-06& 1.05e-06
\end{tabular}
}
\caption{Summary statistics of conic constraint violation for $n=20$.}\label{firstqualitytable}
\end{table}
\begin{table}[H]
\centering
\subtable[Classical.]{\begin{tabular}{lrrrrrrr}
Algorithm & min & avg & max & std
\\
\hline
LiftedLP&0.00e+00& \bf7.78e-09& 9.01e-08& 1.57e-08\\
CPLEXSepLp&0.00e+00& \bf9.56e-09& 7.86e-08& 1.41e-08\\
CPLEXTowerLp&0.00e+00& \bf5.75e-09& 5.76e-08& 1.02e-08\\
CPLEXTowerSepLp&0.00e+00& \bf1.67e-09& 1.95e-08& 3.52e-09\\
GurobiSepLp&4.91e-04& \bf7.52e-04& 1.03e-03& 1.10e-04\\
GurobiTowerLp&4.77e-04& \bf7.10e-04& 1.04e-03& 1.18e-04\\
GurobiTowerSepLp&2.71e-06& \bf4.19e-05& 1.69e-04& 3.20e-05\\
CPLEXSepLazy&4.73e-06& \bf8.32e-06& 1.17e-05& 1.25e-06\\
GurobiSepLazy&4.93e-06& \bf8.23e-06& 1.19e-05& 1.34e-06
\end{tabular}
}
\subtable[Shortfall.]{\begin{tabular}{lrrrrrrr}
Algorithm & min & avg & max & std
\\
\hline
LiftedLP&0.00e+00& \bf0.00e+00& 0.00e+00& 0.00e+00\\
CPLEXSepLp&0.00e+00& \bf2.32e-09& 2.40e-08& 4.58e-09\\
CPLEXTowerLp&0.00e+00& \bf0.00e+00& 0.00e+00& 0.00e+00\\
CPLEXTowerSepLp&0.00e+00& \bf0.00e+00& 0.00e+00& 0.00e+00\\
GurobiSepLp&5.46e-04& \bf7.91e-04& 1.15e-03& 1.12e-04\\
GurobiTowerLp&0.00e+00& \bf3.93e-05& 2.42e-04& 3.31e-05\\
GurobiTowerSepLp&0.00e+00& \bf5.30e-05& 3.60e-04& 5.19e-05\\
CPLEXSepLazy&0.00e+00& \bf8.38e-06& 1.11e-05& 1.38e-06\\
GurobiSepLazy&0.00e+00& \bf8.20e-06& 1.20e-05& 1.95e-06
\end{tabular}
}
\subtable[Robust.]{\begin{tabular}{lrrrrrrr}
Algorithm & min & avg & max & std
\\
\hline
LiftedLP&0.00e+00& \bf6.47e-09& 1.12e-07& 1.50e-08\\
CPLEXSepLp&2.17e-11& \bf4.49e-09& 5.83e-08& 7.64e-09\\
CPLEXTowerLp&1.13e-11& \bf3.68e-09& 4.03e-08& 6.82e-09\\
CPLEXTowerSepLp&0.00e+00& \bf9.52e-10& 1.42e-08& 1.74e-09\\
GurobiSepLp&5.57e-04& \bf8.22e-04& 1.23e-03& 1.30e-04\\
GurobiTowerLp&1.04e-06& \bf2.95e-05& 1.42e-04& 2.65e-05\\
GurobiTowerSepLp&8.51e-07& \bf4.09e-05& 1.71e-04& 3.84e-05\\
CPLEXSepLazy&2.70e-06& \bf6.75e-06& 1.00e-05& 1.39e-06\\
GurobiSepLazy&2.21e-06& \bf6.31e-06& 9.71e-06& 1.46e-06
\end{tabular}
}
\caption{Summary statistics of conic constraint violation for $n=30$.}
\end{table}
\begin{table}[H]
\centering
\subtable[Classical.]{\begin{tabular}{lrrrrrrr}
Algorithm & min & avg & max & std
\\
\hline
LiftedLP&0.00e+00& \bf4.47e-09& 1.23e-07& 1.68e-08\\
CPLEXSepLp&3.96e-11& \bf9.40e-09& 8.67e-08& 1.56e-08\\
CPLEXTowerLp&0.00e+00& \bf3.38e-09& 3.01e-08& 6.07e-09\\
CPLEXTowerSepLp&0.00e+00& \bf1.02e-09& 4.12e-08& 4.36e-09\\
GurobiSepLp&6.61e-04& \bf9.23e-04& 1.31e-03& 1.33e-04\\
GurobiTowerLp&1.60e-06& \bf3.91e-05& 2.31e-04& 3.83e-05\\
GurobiTowerSepLp&3.59e-06& \bf6.26e-05& 2.20e-04& 5.33e-05\\
CPLEXSepLazy&6.73e-06& \bf9.94e-06& 1.32e-05& 1.32e-06\\
GurobiSepLazy&7.19e-06& \bf1.03e-05& 1.41e-05& 1.48e-06
\end{tabular}
}
\subtable[Shortfall.]{\begin{tabular}{lrrrrrrr}
Algorithm & min & avg & max & std
\\
\hline
LiftedLP&0.00e+00& \bf0.00e+00& 0.00e+00& 0.00e+00\\
CPLEXSepLp&0.00e+00& \bf4.15e-09& 4.05e-08& 6.54e-09\\
CPLEXTowerLp&0.00e+00& \bf0.00e+00& 0.00e+00& 0.00e+00\\
CPLEXTowerSepLp&0.00e+00& \bf1.75e-11& 1.55e-09& 1.56e-10\\
GurobiSepLp&7.17e-04& \bf9.36e-04& 1.26e-03& 1.16e-04\\
GurobiTowerLp&2.82e-07& \bf5.86e-05& 3.51e-04& 6.24e-05\\
GurobiTowerSepLp&5.73e-06& \bf6.74e-05& 6.29e-04& 8.16e-05\\
CPLEXSepLazy&7.61e-06& \bf1.05e-05& 1.41e-05& 1.32e-06\\
GurobiSepLazy&0.00e+00& \bf1.00e-05& 1.30e-05& 1.65e-06
\end{tabular}
}
\subtable[Robust.]{\begin{tabular}{lrrrrrrr}
Algorithm & min & avg & max & std
\\
\hline
LiftedLP&0.00e+00& \bf5.48e-09& 8.07e-08& 1.35e-08\\
CPLEXSepLp&6.25e-12& \bf3.72e-09& 3.74e-08& 6.68e-09\\
CPLEXTowerLp&0.00e+00& \bf2.25e-09& 3.65e-08& 4.38e-09\\
CPLEXTowerSepLp&0.00e+00& \bf1.54e-09& 2.68e-08& 3.31e-09\\
GurobiSepLp&7.19e-04& \bf9.86e-04& 1.35e-03& 1.57e-04\\
GurobiTowerLp&3.04e-06& \bf2.97e-05& 1.38e-04& 2.42e-05\\
GurobiTowerSepLp&2.77e-06& \bf4.00e-05& 2.34e-04& 3.66e-05\\
CPLEXSepLazy&3.38e-06& \bf8.47e-06& 1.25e-05& 1.90e-06\\
GurobiSepLazy&3.58e-06& \bf8.51e-06& 1.32e-05& 1.94e-06
\end{tabular}
}
\caption{Summary statistics of conic constraint violation for $n=40$.}
\end{table}
\begin{table}[H]
\centering
\subtable[Classical.]{\begin{tabular}{lrrrrrrr}
Algorithm & min & avg & max & std
\\
\hline
LiftedLP&0.00e+00& \bf5.00e-09& 1.46e-07& 2.11e-08\\
CPLEXSepLp&1.27e-11& \bf9.51e-09& 5.63e-08& 1.20e-08\\
CPLEXTowerLp&0.00e+00& \bf3.34e-09& 5.17e-08& 6.74e-09\\
CPLEXTowerSepLp&0.00e+00& \bf1.14e-09& 4.82e-08& 5.31e-09\\
GurobiSepLp&7.63e-04& \bf1.13e-03& 1.52e-03& 1.56e-04\\
GurobiTowerLp&0.00e+00& \bf6.45e-05& 3.19e-04& 5.47e-05\\
GurobiTowerSepLp&0.00e+00& \bf6.97e-05& 3.03e-04& 6.10e-05\\
CPLEXSepLazy&8.87e-06& \bf1.23e-05& 1.64e-05& 1.51e-06\\
GurobiSepLazy&7.96e-06& \bf1.21e-05& 1.62e-05& 1.43e-06
\end{tabular}
}
\subtable[Shortfall.]{\begin{tabular}{lrrrrrrr}
Algorithm & min & avg & max & std
\\
\hline
LiftedLP&0.00e+00& \bf0.00e+00& 0.00e+00& 0.00e+00\\
CPLEXSepLp&0.00e+00& \bf4.62e-09& 3.75e-08& 6.74e-09\\
CPLEXTowerLp&0.00e+00& \bf0.00e+00& 0.00e+00& 0.00e+00\\
CPLEXTowerSepLp&0.00e+00& \bf0.00e+00& 0.00e+00& 0.00e+00\\
GurobiSepLp&7.21e-04& \bf1.10e-03& 1.54e-03& 1.49e-04\\
GurobiTowerLp&0.00e+00& \bf8.81e-05& 3.75e-04& 9.06e-05\\
GurobiTowerSepLp&0.00e+00& \bf9.94e-05& 4.00e-04& 8.90e-05\\
CPLEXSepLazy&6.48e-06& \bf1.22e-05& 1.70e-05& 1.91e-06\\
GurobiSepLazy&0.00e+00& \bf1.22e-05& 1.66e-05& 2.05e-06
\end{tabular}
}
\subtable[Robust.]{\begin{tabular}{lrrrrrrr}
Algorithm & min & avg & max & std
\\
\hline
LiftedLP&0.00e+00& \bf6.56e-09& 2.55e-07& 2.76e-08\\
CPLEXSepLp&2.47e-11& \bf5.58e-09& 6.47e-08& 1.26e-08\\
CPLEXTowerLp&0.00e+00& \bf2.99e-09& 3.95e-08& 5.81e-09\\
CPLEXTowerSepLp&0.00e+00& \bf5.75e-10& 1.49e-08& 1.87e-09\\
GurobiSepLp&8.24e-04& \bf1.15e-03& 1.54e-03& 1.66e-04\\
GurobiTowerLp&2.00e-06& \bf3.71e-05& 2.38e-04& 4.06e-05\\
GurobiTowerSepLp&2.56e-06& \bf4.73e-05& 1.94e-04& 4.30e-05\\
CPLEXSepLazy&4.09e-06& \bf1.01e-05& 1.48e-05& 2.07e-06\\
GurobiSepLazy&3.59e-06& \bf1.01e-05& 1.65e-05& 2.17e-06
\end{tabular}
}
\caption{Summary statistics of conic constraint violation for $n=50$.}
\end{table}
\begin{table}[H]
\centering
\subtable[Classical.]{\begin{tabular}{lrrrrrrr}
Algorithm & min & avg & max & std
\\
\hline
LiftedLP&0.00e+00& \bf7.16e-09& 1.57e-07& 2.34e-08\\
CPLEXSepLp&0.00e+00& \bf1.14e-08& 1.32e-07& 2.02e-08\\
CPLEXTowerLp&0.00e+00& \bf2.23e-09& 2.41e-08& 4.45e-09\\
CPLEXTowerSepLp&0.00e+00& \bf3.57e-10& 2.38e-08& 2.41e-09\\
GurobiSepLp&0.00e+00& \bf1.31e-03& 1.69e-03& 2.19e-04\\
GurobiTowerLp&0.00e+00& \bf6.18e-05& 2.65e-04& 5.12e-05\\
GurobiTowerSepLp&0.00e+00& \bf8.76e-05& 4.45e-04& 8.17e-05\\
CPLEXSepLazy&1.11e-05& \bf1.39e-05& 1.83e-05& 1.56e-06\\
GurobiSepLazy&8.65e-06& \bf1.41e-05& 1.90e-05& 1.89e-06
\end{tabular}
}
\subtable[Shortfall.]{\begin{tabular}{lrrrrrrr}
Algorithm & min & avg & max & std
\\
\hline
LiftedLP&0.00e+00& \bf0.00e+00& 0.00e+00& 0.00e+00\\
CPLEXSepLp&0.00e+00& \bf6.95e-09& 1.61e-07& 1.79e-08\\
CPLEXTowerLp&0.00e+00& \bf1.06e-11& 1.06e-09& 1.06e-10\\
CPLEXTowerSepLp&0.00e+00& \bf0.00e+00& 0.00e+00& 0.00e+00\\
GurobiSepLp&0.00e+00& \bf1.24e-03& 1.63e-03& 1.85e-04\\
GurobiTowerLp&0.00e+00& \bf1.14e-04& 5.25e-04& 9.93e-05\\
GurobiTowerSepLp&0.00e+00& \bf1.16e-04& 5.57e-04& 1.09e-04\\
CPLEXSepLazy&9.24e-06& \bf1.43e-05& 2.01e-05& 1.97e-06\\
GurobiSepLazy&1.06e-05& \bf1.43e-05& 2.19e-05& 2.03e-06
\end{tabular}
}
\subtable[Robust.]{\begin{tabular}{lrrrrrrr}
Algorithm & min & avg & max & std
\\
\hline
LiftedLP&0.00e+00& \bf6.00e-09& 8.58e-08& 1.42e-08\\
CPLEXSepLp&2.80e-11& \bf7.93e-09& 1.14e-07& 1.46e-08\\
CPLEXTowerLp&0.00e+00& \bf2.10e-09& 4.10e-08& 4.76e-09\\
CPLEXTowerSepLp&0.00e+00& \bf5.02e-10& 1.48e-08& 1.83e-09\\
GurobiSepLp&9.63e-04& \bf1.31e-03& 1.56e-03& 1.27e-04\\
GurobiTowerLp&2.06e-06& \bf3.58e-05& 1.27e-04& 2.96e-05\\
GurobiTowerSepLp&2.53e-06& \bf4.82e-05& 2.06e-04& 4.36e-05\\
CPLEXSepLazy&8.41e-06& \bf1.27e-05& 1.71e-05& 1.45e-06\\
GurobiSepLazy&5.04e-06& \bf1.21e-05& 1.77e-05& 1.83e-06
\end{tabular}
}
\caption{Summary statistics of conic constraint violation for $n=60$.}\label{lastqualitytable}
\end{table}

\begin{table}[H]
\centering
\subtable[Classical.]{\begin{tabular}{lrrrrrrr}
Algorithm & min & avg & max & std
\\
\hline
LiftedLP&0.00e+00& \bf4.18e-09& 1.13e-07& 1.40e-08\\
CPLEXSepLp&0.00e+00& \bf9.51e-09& 1.13e-07& 1.97e-08\\
GurobiSepLp&0.00e+00& \bf1.68e-03& 3.04e-03& 9.73e-04\\
CPLEXSepLazy&1.58e-05& \bf2.12e-05& 2.59e-05& 1.97e-06\\
GurobiSepLazy&1.39e-05& \bf2.12e-05& 2.69e-05& 2.31e-06
\end{tabular}
}
\subtable[Shortfall.]{\begin{tabular}{lrrrrrrr}
Algorithm & min & avg & max & std
\\
\hline
LiftedLP&0.00e+00& \bf0.00e+00& 0.00e+00& 0.00e+00\\
CPLEXSepLp&0.00e+00& \bf9.92e-09& 1.21e-07& 2.03e-08\\
GurobiSepLp&0.00e+00& \bf1.39e-03& 2.29e-03& 8.42e-04\\
CPLEXSepLazy&1.41e-05& \bf2.14e-05& 2.98e-05& 2.61e-06\\
GurobiSepLazy&0.00e+00& \bf2.19e-05& 2.93e-05& 3.30e-06
\end{tabular}
}
\subtable[Robust.]{\begin{tabular}{lrrrrrrr}
Algorithm & min & avg & max & std
\\
\hline
CPLEXCP&0.00e+00& \bf6.93e-09& 1.78e-07& 2.69e-08\\
GurobiCP&0.00e+00& \bf2.55e-08& 1.18e-06& 1.30e-07\\
CPLEXLP&0.00e+00& \bf4.17e-09& 3.42e-08& 6.27e-09\\
GurobiLP&0.00e+00& \bf6.43e-05& 9.99e-05& 3.59e-05\\
LiftedLP&0.00e+00& \bf8.18e-09& 1.01e-07& 1.88e-08\\
CPLEXSepLp&0.00e+00& \bf1.57e-08& 1.61e-07& 3.26e-08\\
CPLEXTowerLp&0.00e+00& \bf1.20e-09& 3.95e-08& 4.32e-09\\
CPLEXTowerSepLp&0.00e+00& \bf1.15e-10& 7.55e-09& 7.60e-10\\
GurobiSepLp&1.51e-03& \bf1.86e-03& 2.35e-03& 1.94e-04\\
GurobiTowerLp&2.34e-06& \bf5.25e-05& 2.63e-04& 4.32e-05\\
GurobiTowerSepLp&2.60e-06& \bf6.64e-05& 4.05e-04& 6.14e-05\\
CPLEXSepLazy&1.43e-05& \bf2.04e-05& 2.77e-05& 2.52e-06\\
GurobiSepLazy&8.77e-06& \bf1.93e-05& 2.44e-05& 2.37e-06
\end{tabular}
}
\caption{Summary statistics of conic constraint violation for $n=100$ [s].}
\end{table}

\begin{table}[H]
\centering
\subtable[Robust.]{\begin{tabular}{lrrrrrrr}
Algorithm & min & avg & max & std
\\
\hline
CPLEXCP&0.00e+00& \bf1.73e-09& 8.96e-08& 9.49e-09\\
GurobiCP&0.00e+00& \bf2.69e-08& 7.83e-07& 1.07e-07\\
CPLEXLP&0.00e+00& \bf2.04e-09& 5.48e-08& 7.18e-09\\
GurobiLP&0.00e+00& \bf2.61e-05& 9.86e-05& 3.97e-05\\
LiftedLP&0.00e+00& \bf6.70e-09& 1.66e-07& 2.14e-08\\
CPLEXSepLp&0.00e+00& \bf2.19e-07& 7.79e-07& 1.74e-07\\
CPLEXTowerLp&0.00e+00& \bf2.92e-11& 1.14e-09& 1.61e-10\\
CPLEXTowerSepLp&0.00e+00& \bf4.47e-10& 1.28e-08& 2.09e-09\\
GurobiSepLp&0.00e+00& \bf3.40e-03& 4.28e-03& 9.77e-04\\
GurobiTowerLp&0.00e+00& \bf6.96e-05& 2.94e-04& 6.28e-05\\
GurobiTowerSepLp&0.00e+00& \bf7.14e-05& 2.58e-04& 6.26e-05\\
CPLEXSepLazy&2.34e-05& \bf3.84e-05& 4.78e-05& 3.89e-06\\
GurobiSepLazy&3.41e-08& \bf3.80e-05& 4.50e-05& 5.21e-06
\end{tabular}
}
\caption{Summary statistics of conic constraint violation for $n=200$ [s].}
\end{table}

\begin{table}[H]
\centering
\subtable[Robust.]{\begin{tabular}{lrrrrrrr}
Algorithm & min & avg & max & std
\\
\hline
LiftedLP&0.00e+00& \bf4.61e-09& 9.94e-08& 1.35e-08\\
CPLEXSepLp&0.00e+00& \bf2.97e-07& 9.07e-07& 2.48e-07\\
GurobiSepLp&0.00e+00& \bf4.09e-03& 6.35e-03& 2.39e-03\\
CPLEXSepLazy&4.47e-05& \bf5.79e-05& 6.91e-05& 4.48e-06\\
GurobiSepLazy&3.90e-05& \bf5.60e-05& 6.81e-05& 5.63e-06
\end{tabular}
}
\caption{Summary statistics of conic constraint violation for $n=300$ [s].}
\end{table}

\subsection{Additional Graphs}\label{appendixgraphs}

Figure~\ref{dynanmicliftedform}  compares the performance of the three dynamic lifted polyhedral relaxations in their nonlinear reformulation version (CPLEXSepLP, GurobiSepLP, CPLEXTowerLP, GurobiTowerLP, CPLEXTowerSepLP, GurobiTowerSepLP). We also include as a reference the NLP-based algorithms CPLEXCP and GurobiCP. 

\begin{figure}[htpb!]
\centering 
\subfigure[Classical-CPLEX.]{\includegraphics[scale=.85]{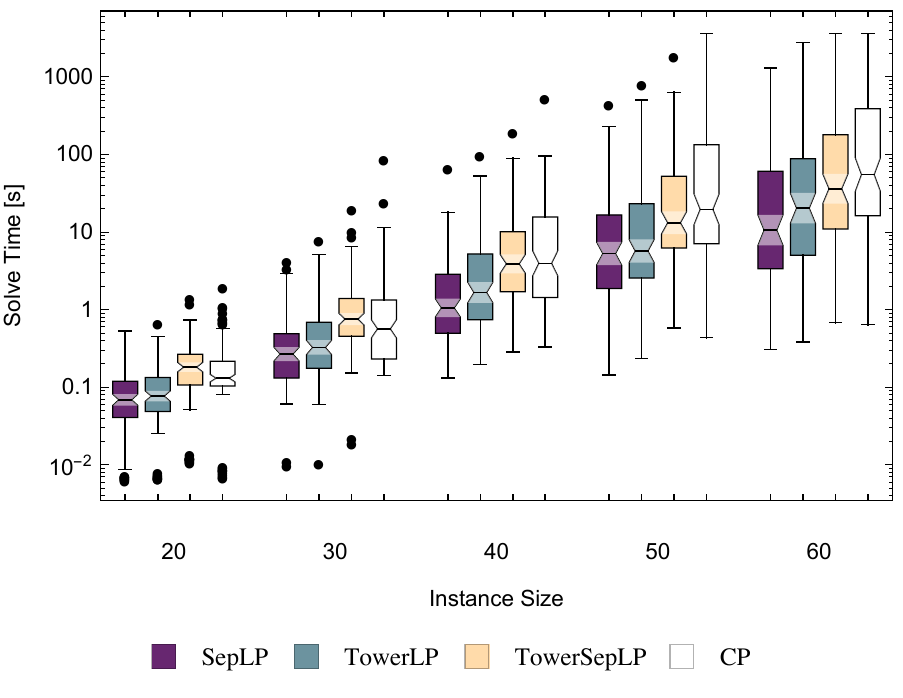}}
\subfigure[Classical-Gurobi.]{\includegraphics[scale=.85]{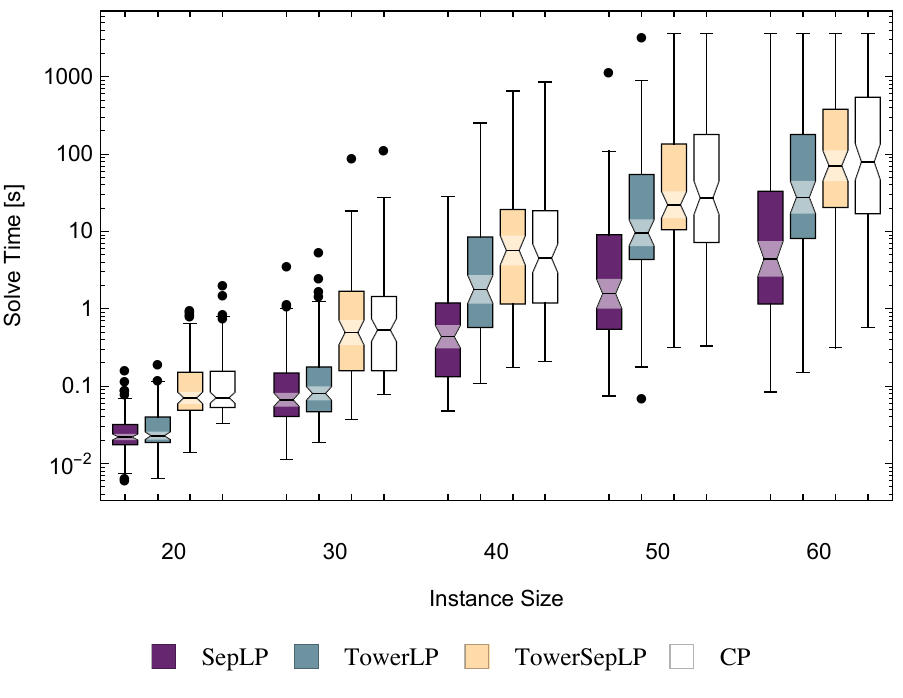}}
\subfigure[Shortfall-CPLEX.]{\includegraphics[scale=.85]{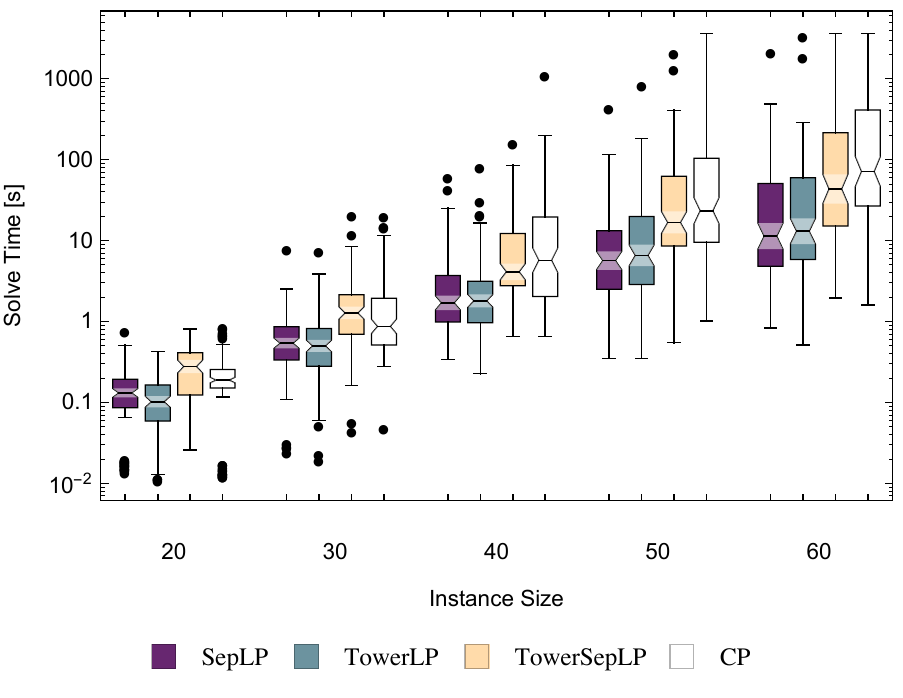}}
\subfigure[Shortfall-Gurobi.]{\includegraphics[scale=.85]{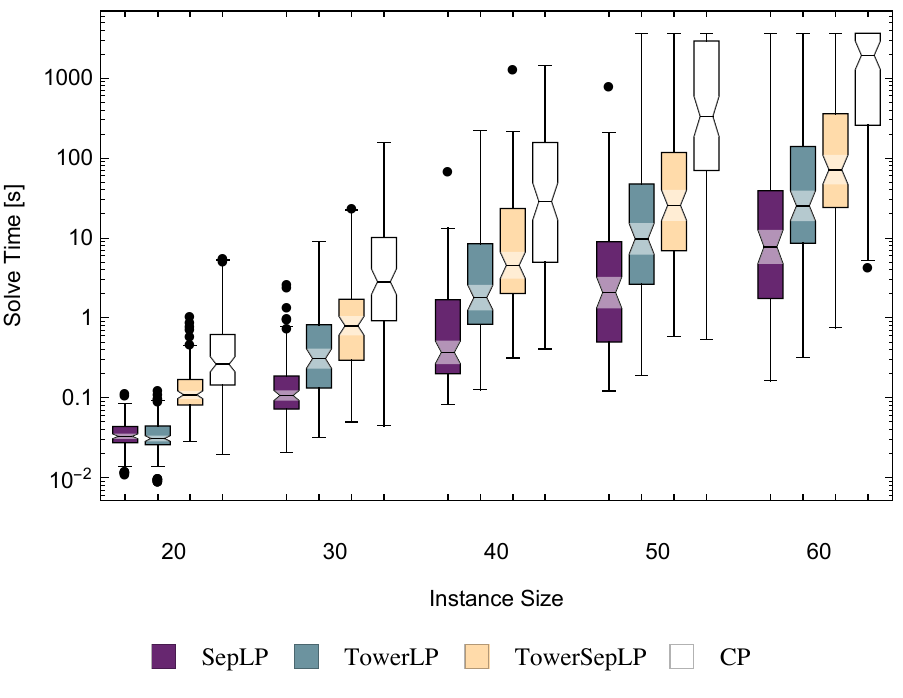}}
\subfigure[Robust-CPLEX.]{\includegraphics[scale=.85]{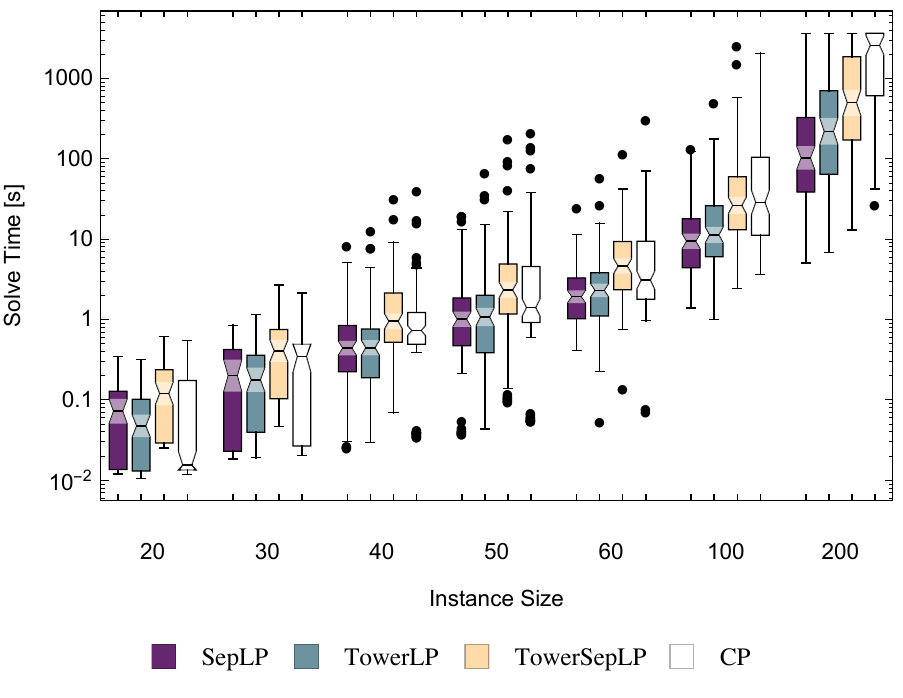}}
\subfigure[Robust-Gurobi.]{\includegraphics[scale=.85]{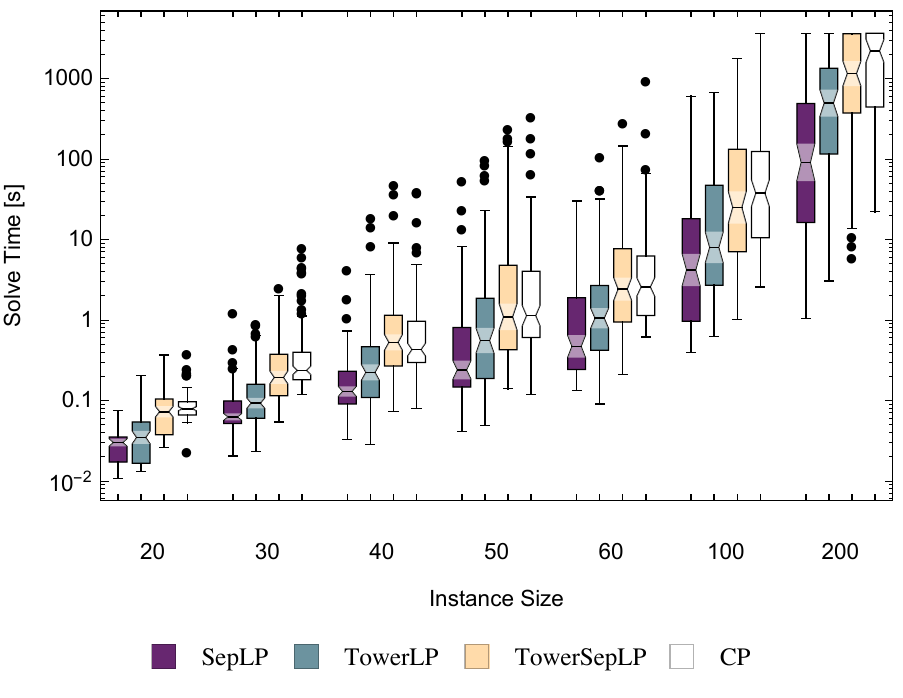}}
\caption{Solution times for dynamic lifted polyhedral relaxations solved by standard LP-based algorithms [s].}\label{dynanmicliftedform}
\end{figure}

LiftedLP
Our final set of charts compare the branch- and cut-based LiftedLP algorithms (LiftedLP, CPLEXSepLazy and GurobiSepLazy) with the separable dynamic lifted polyhedral relaxation in its nonlinear reformulation version (CPLEXSepLP and GurobiSepLP). To minimize the complexity of the graphics we divide the comparison into two parts.  In Figure~\ref{branchvsreform} we compare the branch-based algorithm with the separable reformulations and in Figure~\ref{cutvsreform} we compare the cut-based algorithms with the separable reformulations.
\begin{figure}[htb]
\centering 
\subfigure[Classical.]{\includegraphics[scale=.85]{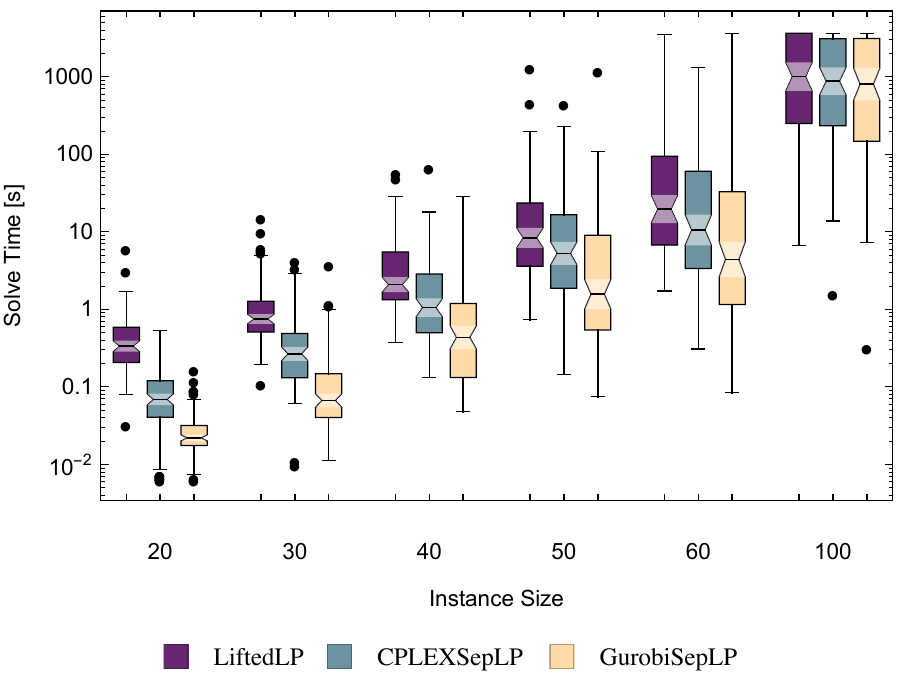}}
\subfigure[Shortfall.]{\includegraphics[scale=.85]{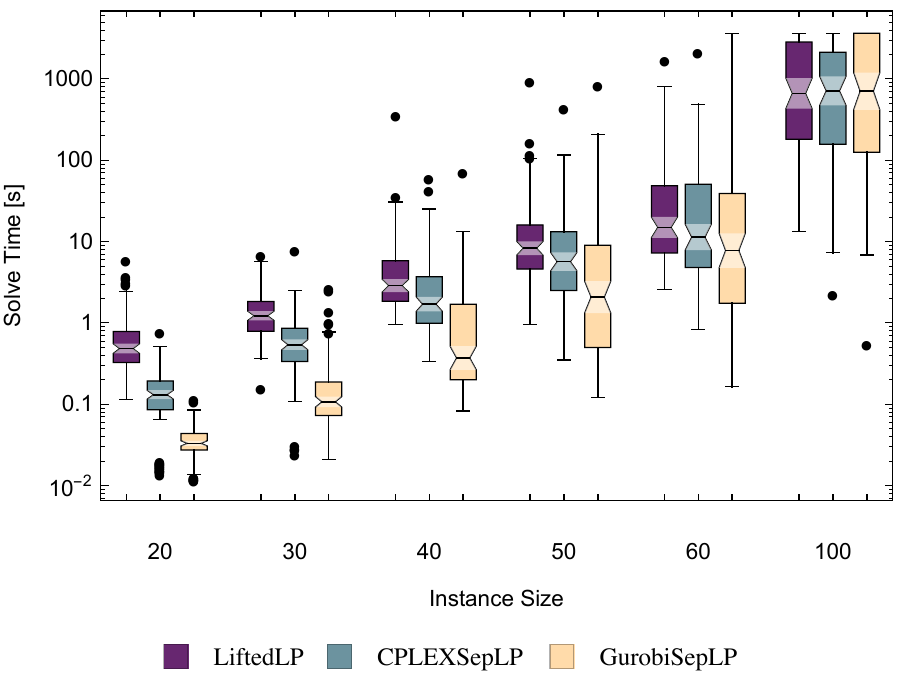}}
\subfigure[Robust.]{\includegraphics[scale=.85]{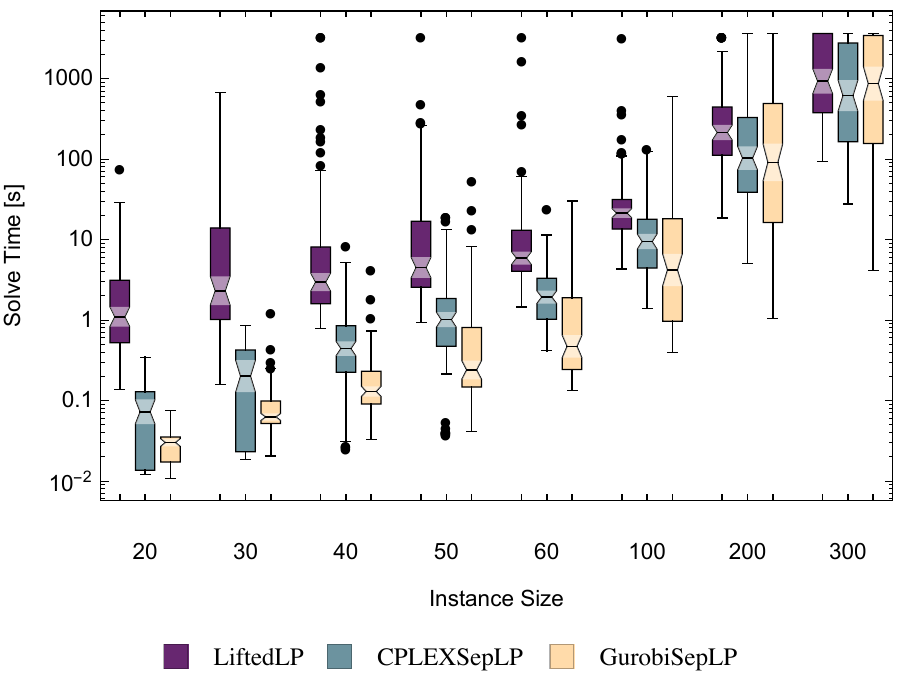}}
\caption{Solution times for the branch-based LiftedLP algorithm and best LP-based (separable re-formulation) algorithms [s].}\label{branchvsreform}
\end{figure}

\begin{figure}[htb]
\centering 
\subfigure[Classical.]{\includegraphics[scale=.85]{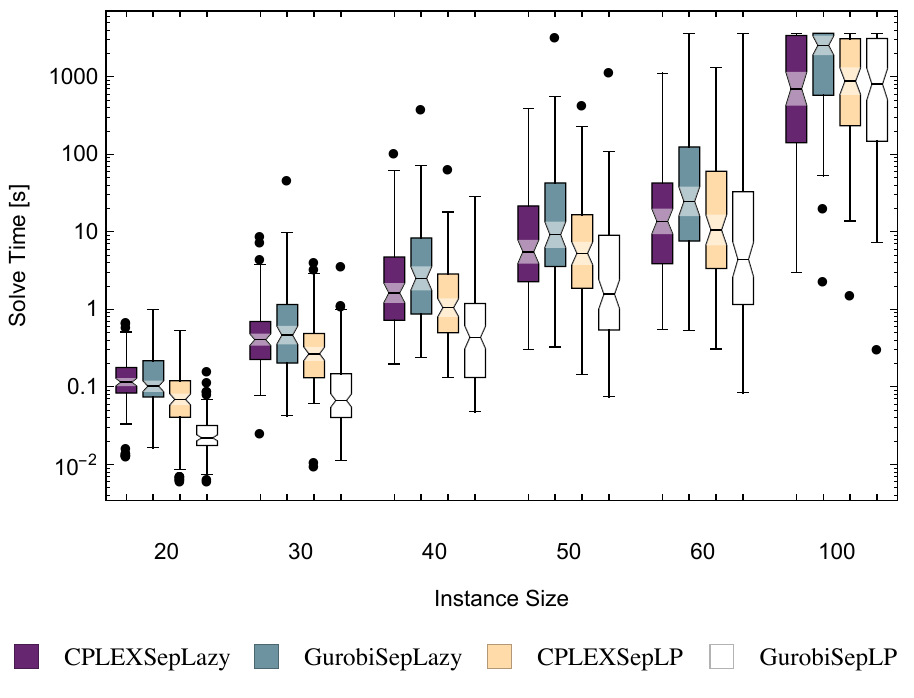}}
\subfigure[Shortfall.]{\includegraphics[scale=.85]{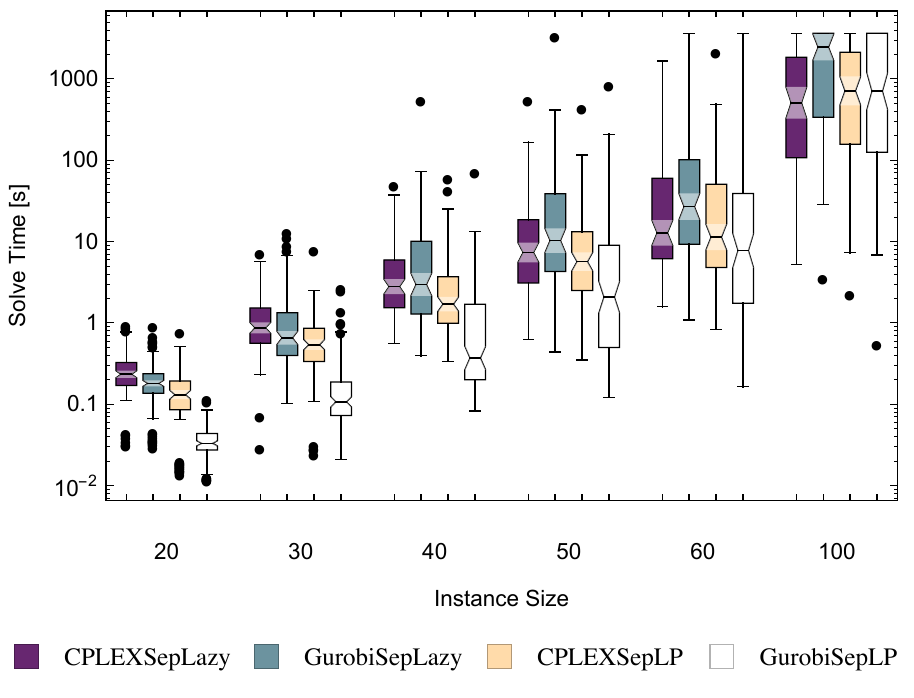}}
\subfigure[Robust.]{\includegraphics[scale=.85]{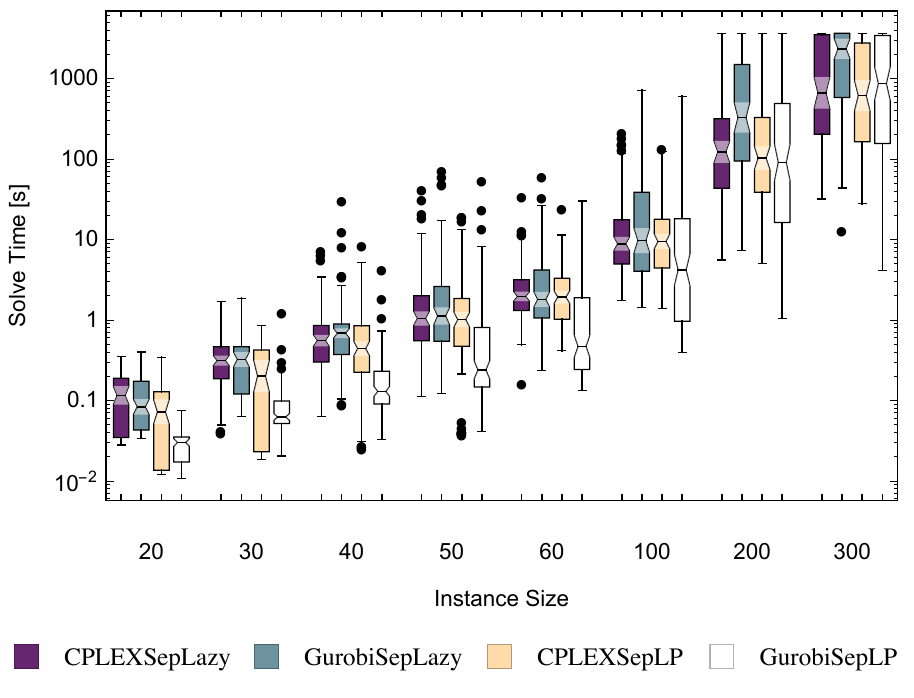}}
\caption{Solution times for the cut-based LiftedLP algorithms and best LP-based (separable re-formulation) algorithms [s].}\label{cutvsreform}
\end{figure}
From Figure~\ref{branchvsreform} we see that the separable reformulations provide an advantage over the branch-based LiftedLP algorithm for all instances. However, this advantage becomes smaller as the problems sizes increase and the LiftedLP algorithm can provide an advantage in some instances (e.g Table~\ref{summaryfortwohundred} shows the algorithms is the fastest in 10\% of the robust instances for $n=200$). Figure~\ref{cutvsreform} shows a similar behavior for the cut-based algorithms.

\end{document}